\documentclass[10pt,a4paper]{amsart}
\usepackage{amsmath}
\usepackage{amssymb}
\usepackage{amsthm}
\usepackage[cp1250]{inputenc}
\usepackage[T1]{fontenc}
\usepackage{enumerate}
\usepackage{latexsym}
\usepackage{hyperref}
\usepackage{mathrsfs}
\usepackage{graphicx}
\usepackage{color}

\usepackage{todonotes}

\newtheorem{tw}{Theorem}[section]
\newtheorem{lm}[tw]{Lemma}
\newtheorem{prop}[tw]{Proposition}
\newtheorem{wn}[tw]{Corollary}
\theoremstyle{definition}

\newtheorem{uw}[tw]{Remark}

\newcommand{\n}{\mathbb{N}}
\newcommand{\z}{\mathbb{Z}}
\newcommand{\q}{\mathbb{Q}}
\newcommand{\re}{\mathbb{R}}
\newcommand{\al}{\alpha}

\newcommand{\la}{\lambda}
\newcommand{\g}{\gamma}
\newcommand{\ep}{\varepsilon}
\newcommand{\de}{\delta}
\newcommand{\numberthis}{\addtocounter{equation}{1}\tag{\theequation}}
\title[Translation flows disjoint with their inverse]
{On typicality of translation flows which are disjoint with their inverse}

\author[P.~Berk]{Przemys\l aw Berk}
\address{Faculty of Mathematics and Computer Science, Nicolaus
Copernicus University, ul. Chopina 12/18, 87-100 Toru\'n, Poland}
\email{zimowy@mat.umk.pl}

\author[K.~Fr\k{a}czek]{Krzysztof Fr\k{a}czek}
\address{Faculty of Mathematics and Computer Science, Nicolaus
Copernicus University, ul. Chopina 12/18, 87-100 Toru\'n, Poland}
\email{fraczek@mat.umk.pl}

\author[T.~de~la~Rue]{Thierry de la Rue}
\address{Laboratoire de Math\'ematiques Rapha\"el Salem, Normandie Universit\'e,
Universit\'e de Rouen, CNRS, 
Avenue de l'Universit\'e, 76801 Saint Etienne du Rouvray, France}
\email{Thierry.de-la-Rue@univ-rouen.fr}

\subjclass[2000]{37A10, 37E35, 37C80}

\date{\today}

\keywords{measure-preserving flows, translation  surfaces, reversibility of dynamical systems, joinings methods in ergodic theory}

\begin{document}
	\begin{abstract}
	In this paper we prove that translation structures for which the corresponding vertical translation flows is weakly mixing and disjoint with its inverse, form a $G_\de$-dense set in every non-hyperelliptic connected component of the moduli space $\mathcal M$. This is in contrast to hyperelliptic case, where for every translation structure the associated vertical flow is isomorphic to its inverse. To prove the main result, we study  limits of the off-diagonal 3-joinings of special representations of vertical translation flows. Moreover, we construct a locally defined continuous embedding of the moduli space into the space of measure-preserving flows to obtain the $G_\de$-condition.
	\end{abstract}
	\maketitle

\section{Introduction}\label{sec:intr}
Let $M$ be an orientable compact connected topological surface, and $\Sigma$ be a finite set of singular points. On $M$ we can consider a \emph{translation structure} $\zeta$, i.e.\ an atlas on $M\setminus\Sigma$ such that every transition transformation  is a translation. Every translation surface can be viewed as a polygon with pairwise parallel sides of the same length which are glued together (gluing is made by a translation). Parameters given by the sides of such polygons establish a parametrization of the so-called \emph{moduli space} $\mathcal M$ and yield a topology on $\mathcal M$. To each translation structure $\zeta$ we  associate the corresponding Lebesgue measure $\la_\zeta$ on $M$. Moreover, for every direction we consider the flow which acts by translation in that direction with unit speed.  Such translation flows preserve $\la_\zeta$. In this paper we are interested in vertical translation flows. It is worth to mention that the study of directional flows on translation surfaces originates
from  problems concerning billiard flows on rational polygons (see \cite{Fox},\cite{Katok}).

In \cite{KonZo} the authors give a complete characterisation of connected components of the moduli space; all of them are orbifolds. On each connected component $C$ we consider an action of $SL_2(\re)$ which is derived form the linear action of $SL_2(\re)$ on polygons. Moreover, there is a Lebesgue measure $\nu_C$ on $C$ which is invariant under this action.
Let $\mathcal M_1$ be the set of $\zeta\in\mathcal M$ such that $\la_\zeta(M)=1$ and let $C_1:=C\cap \mathcal M_1$. We also consider a measure $\nu_{C_1}$ on $C_1$ which is a projectivization of $\nu_C$. This measure is finite and invariant under the action of $SL_2(\re)$ on $C_1$. In fact, the action of $SL_2(\re)$ is ergodic with respect to this measure (see \cite {Maz} and \cite{Veech}).
This gives an opportunity to use the ergodic theory to study dynamical properties of vertical flows on almost all translations structures. In particular, it was used to prove that the sets of translation structures for which the vertical translation flow is ergodic (see \cite{Maz}), and further is weakly mixing (see \cite{AF}) are of full measure in both $C$ and $C_1$. At the same time, there are no mixing translation flows (see \cite{Ka}).
In this paper we are interested in translation structures for which the corresponding vertical translation flow is disjoint with its inverse, which is a stronger notion than being not reversible. Recall that a measure preserving flow $\{T_t\}_{t\in\re}$ on $(X,\mu)$ is reversible, if there exists an involution $\theta:X\to X$ which preserves $\mu$ and
\[
\theta\circ T_{-t}=T_t\circ\theta\ \text{for all}\ t\in\re.
\]
Our result concerns the topological typicality of the desired property rather than measure-theoretical.
As a by-product, we give a method to show that the set of translation structures for which the associated vertical translation flows satisfy any property which is $G_\de$ in the space of measure preserving flows, is also a $G_\de$-set. Among these properties are for instance weak mixing, ergodicity and rigidity (see \cite{Halmos}).

In the classification of connected components given in \cite{KonZo} we distinguish so called hyperelliptic components. For every hyperelliptic component $C$ there exists an involution $\theta:M\to M$ such that for every $\zeta\in C$ it is given in local coordinates by the formula $z\mapsto -z+c$ for some $c\in\mathbb C$.
In particular, the vertical flow on $(M,\zeta)$ is reversible; it is isomorphic with its inverse by the involution $\theta$ (see remark \ref{hypisom}). In contrast, in this paper we show that on non-hyperelliptic components of the moduli space the set of translation structures for which the vertical flow is disjoint with its inverse is topologically large. It is expressed by the following theorem.
\begin{tw}\label{main}
	Let $C$ be a non-hyperelliptic connected component of the moduli space of translation structures. Then the set of translation structures whose vertical flow is weakly mixing and disjoint with its inverse is a $G_\de$-dense set in $C$.
\end{tw}
It is also worth to mention that on non-hyperelliptic components we can also find a non-trivial set of translation structures for which 
the vertical flow is reversible.
\begin{prop}\label{odwr}
	Let $C$ be a non-hyperelliptic connected component of the moduli space of translation structures. Then the set of translation structures whose vertical flow is reversible is  dense in $C$.
	\end{prop}

Recall that a measure-preserving flow $\{T_t\}_{t\in\re}$ on a standard Borel probability space $(X,\mathcal{B},\mu)$ is disjoint with its inverse if the only $(T_t\times T_{-t})$ - invariant probability measure on $X\times X$, which projects on each coordinate as $\mu$ is 
 the product measure $\mu\otimes\mu$.
In \cite{FrKuLem} the authors developed techniques to prove non-isomorphism of a flow $T^f$ to its inverse that are 
based on studying the weak closure of off-diagonal $3$-self-joinings. Moreover, in \cite{BerkFr} the authors improved those techniques to show that a large class of special flows over interval exchange transformations and under piecewise absolutely continuous functions 
have the property of being non-isomorphic with 
their inverse. The idea of detecting non-isomorphism of a dynamical system and its inverse by studying the weak closure of off-diagonal $3$-self-joinings was introduced by Ryzhikov in \cite{Ryz1}. In this paper we prove that techniques mentioned earlier can be used to detect disjointness of a vertical flow with its inverse.

 To prove the $G_\de$ condition, we use the 
 result of Danilenko and Rhyzhikov from \cite{DaRy} (which derives from a version for automorphisms given in \cite{dJun}), where they proved that the flows with the property of being disjoint with 
 their inverse form a $G_\delta$-dense set in the space of measure preserving flows. To use their result we construct a locally defined continuous embedding of 
 the moduli space into the space of measure preserving flows. To show the density condition, we largely rely on the proof of Lemma 14 in \cite{Fr}.

In Section \ref{sec:prem} we give 
a general background concerning joinings, interval exchange transformations, space of measure preserving flows, translation flows and moduli spaces. In particular, we give some tools needed to prove the continuity of a map with values in the space of flows and we state some connections between the moduli space and interval exchange transformations.

In Section \ref{sec:join} we introduce a criterion of disjointness of two flows by researching the weak limits of certain $3$-self-joinings. This is a direct improvement of 
the criteria stated in \cite{FrKuLem} and \cite{BerkFr} as we show that these are actually criteria of two weakly mixing flows being disjoint. Furthermore, we state a criterion of a flow being weakly mixing, which also uses weak limits of $3$-self-joinings as a tool.

In Section \ref{sec:perm} we state combinatorial conditions on translation structures which are later used in proving the density condition in our main theorem. To be precise, we show that our results apply to every non-hyperelliptic component of 
the moduli space.

In Section \ref{sec:meas} we show that on a given translation surface $(M,\zeta)$ there exists $\ep_\zeta>0$ such that for every absolutely continuous measure $\mu$ on $M$ which has density $f$ satisfying $\int_{M}|f(x)-1|dx<\ep_\zeta$ there exists a homeomorphism $H:M\to M$ such that $H_*\mu$ is the Lebesgue measure with $H$ depending continuously on $f$. The results of section \ref{sec:meas} were inspired by the works of Moser in \cite{Mos} and Goffman, Pedrick in \cite{GP}.

In Section \ref{sec:emb} we use the results presented in 
the previous section to construct a continuous embedding of each connected component of the moduli space into the space of measure-preserving flows. The embedding is defined locally, but we also show that this is enough to transfer the $G_\delta$ condition.

Finally, in Section \ref{sec:main} we first state a result which is a conclusion from the previous 
sections, that in every connected component of 
the moduli space the set of translation structures whose associated vertical translation flow is disjoint with its inverse, is a $G_\de$ set. In the remainder of this section we use the results from \cite{BerkFr} to show that the criteria introduced in Section \ref{sec:join} can be used for a dense set of translation structures in every non-hyperelliptic component,
which leads to the 
proof of Theorem~\ref{main}.
As a by-product we get the proof of Proposition \ref{odwr}.
	
\section{Preliminaries}\label{sec:prem}
We will now give some details regarding interval exchange transformations, joinings of dynamical systems and some basic information about moduli space.

\subsection{Joinings}
In this subsection we give some definitions which are stated for standard Borel probability spaces.  
However these definitions can be easily extended to standard Borel spaces with finite measures. Though here we state definitions for probability spaces, in the remainder of this paper we will freely use them in case when 
the measure is finite and not necessarily a probability. In particular we say that the measure preserving flows $\mathcal T=\{T_t\}_{t\in\re}$ on $(X,\mathcal B,\mu)$ and $\mathcal S=\{S_t\}_{t\in\re}$ on $(Y,\mathcal C,\nu)$ are \emph{isomorphic}, if there exists a measurable $F:(X,\mathcal B)\to(Y,\mathcal C)$ such that
\[
T_t=F^{-1}\circ S_t\circ F\ \text{for}\ t\in\re\quad\text{and}\quad
F_*\mu=\frac{\mu(X)}{\nu(Y)}\nu.
\]

Let $K>0$ be a natural number and for $1\le i\le K$ let $\mathcal T^i=\{T^i_t\}_{t\in\re}$ be a measure preserving flow acting on a standard Borel probability space $(X^i,\mathcal B^i,\mu^i)$ . We say that a measure $\la$ on $(X^1\times\ldots\times X^K,\mathcal B^1\otimes\ldots\otimes\mathcal B^K)$ is a {\it K-~joining} if it is 
$\mathcal T^1\times\ldots\times\mathcal T^K$-invariant and
it projects on $X^i$ as $\mu^i$ for each $i=1,\ldots,K$. We denote by $J(\mathcal T^1,\ldots,\mathcal T^K)$ the set of all joinings of $\mathcal T^i$ for $i=1,\ldots,K$ and by $J^e(\mathcal T^1,\ldots,\mathcal T^K)$ the subset of ergodic joinings. If for $i=1,\ldots,K$, $(X^i,\mathcal B^i,\mu^i,\mathcal T^i)$ are copies of the same flow, then we say that $\la$ is a {\it K-self joining}. We denote the set of $K$-self joinings of a flow $\mathcal T$ by $J_K(\mathcal T)$ and ergodic $K$-self joinings by $J^e_K(\mathcal T)$. If $\mathcal T^1,\ldots,\mathcal T^K$ are ergodic, then the following remarks hold.
\begin{uw}
	$J(\mathcal T^1,\ldots,\mathcal T^K)$ is a compact simplex in the space of all $\mathcal T^1\times\ldots\times\mathcal T^K$-invariant measures and its set of extremal points $\big(J(\mathcal T^1,\ldots,\mathcal T^K)\big)$ equals $J^e(\mathcal T^1,\ldots,\mathcal T^K)$.
\end{uw}
\begin{uw}[Ergodic decomposition]
	For each $\la\in J(\mathcal T^1,\ldots,\mathcal T^K)$ there exists a unique probability measure $\kappa$ on $J^e(\mathcal T^1,\ldots,\mathcal T^K)$ such that
	\[
	\la=\int_{J^e(\mathcal T^1,\ldots,\mathcal T^K)}\rho d\kappa(\rho).
	\]
\end{uw}

Assume that $K=2$. Note that $\mu^1\otimes\mu^2\in J(\mathcal T^1,\mathcal T^2)$. We say that 
the flows $\mathcal T^1$ and $\mathcal T^2$ are {\it disjoint in the sense of Furstenberg} (or simply disjoint) if the product measure is the only joining between them.
\begin{uw}
	If two flows are disjoint, then they have no common factor. In particular, they are not isomorphic.
\end{uw}
The 
notions of joinings and disjointness can be rewritten for automorphisms instead of flows. Then we also have the following well-known observation.
\begin{uw}\label{ident}
	If $(X,\mathcal B,\mu,T)$ is an ergodic automorphism and $(Y,\mathcal C,\nu,Id)$ is the identity then $T$ and $Id$ are disjoint.
	\end{uw}

Let $\phi:(X^1,\mathcal B^1,\mu^1,\mathcal T^1)\to(X^2,\mathcal B^2,\mu^2,\mathcal T^2)$ be an isomorphism. It is easy to see that $\mu_\phi^1:=(Id\times \phi)_*\mu^1$ is a joining of $\mathcal T^1$ and $\mathcal T^2$. We say that $\mu_\phi^1$ is a {\it graph joining}. We have the following remark.
\begin{uw}\label{graph}
	Let $\la\in J(\mathcal T^1,\mathcal T^2)$ and let $\Pi\subseteq \mathcal B^1$ be a family of measurable sets. Let $\phi:(X^1,\mathcal B^1,\mu^1,\mathcal T^1)\to(X^2,\mathcal B^2,\mu^2,\mathcal T^2)$ be an isomorphism. Then the following are equivalent:
	\begin{itemize}
	\item[(1)] $\la(A\times B)=\mu^1(A\cap\phi^{-1}(B))$ for all $A\in\Pi$ and $B\in\phi(\Pi)$;
	\item[(2)] $\la(A\times X\ \triangle\ X\times \phi A)=0$ for every $A\in\Pi$;
	\item[(3)] $\la(A\times \phi A^c)=\la(A^c\times\phi A)=0$ for every $A\in\Pi$.
	\end{itemize}
	\end{uw}
Consider graph joinings between two identical flows $(X,\mathcal B,\mu,\mathcal T)$. If $\phi=T_{-t}$ for some $t\in\re$ then we say that $\mu_\phi$ is a {\it 2-off-diagonal joining} and we denote it by $\mu_t$. In other words for $A,B\in\mathcal B$ we have
\[
\mu_t(A\times B)=\mu(A\cap T_t B)=\mu(T_{-t}A\cap B).
\]
This definition is easily extended to higher dimensions, namely a {\it K-off-diagonal joining} $\mu_{t_1,\ldots,t_{K-1}}$ is a $K$-joining given by the formula
\begin{equation}\label{defjoin}
\mu_{t_1,\ldots,t_{K-1}}(A_1\times\ldots\times A_K)=\mu(T_{-t_1}A_1\cap\ldots\cap T_{-t_{K-1}}A_{K-1}\cap A_K),
\end{equation}
for all $A_1,\ldots,A_K\in\mathcal B$.

Let $\mathcal P(\re^{K-1})$ be the set of Borel probability measures on 
$\re^{K-1}$. For every $P\in\mathcal P(\re^{K-1})$ we consider 
the \emph{$K$-integral joining} given by
\[
\begin{split}
\Bigg(\int_{\re^{K-1}}\mu_{t_1,\ldots,t_{K-1}}&dP(t_1,\ldots,t_{K-1})\Bigg)(A_1\times\ldots\times A_K)\\
&:=\int_{\re^{K-1}}\mu_{t_1,\ldots,t_{K-1}}(A_1\times\ldots\times A_K)dP(t_1,\ldots,t_{K-1}),
\end{split}
\]
where $A_1,\ldots,A_K\in\mathcal B$.

A \emph{Markov operator} $\Psi:L^2(X,\mathcal B,\mu)\to L^2(X,\mathcal B,\mu)$ is a linear operator which satisfies
\begin{enumerate}
	\item it is a contraction, that is $\|\Psi f\|_2\le\|f\|_2$ for every $f\in L^2(X,\mathcal B,\mu)$;
	\item $f\ge 0\Rightarrow \Psi(f)\ge 0$;
	\item $\Psi(\mathbf 1)=\mathbf 1=\Psi^*(\mathbf 1)$.
	\end{enumerate}
With every 2-self joining $\la\in J_2(\mathcal T)$, we can associate a Markov operator $\Psi(\la):L^2(X,\mathcal B,\mu)\to L^2(X,\mathcal B,\mu)$ such that
\begin{equation}\label{operator}
\la(A\times B)=\int_X\Psi(\la)(\chi_A)\chi_Bd\mu
\text{ for any }A,B\in\mathcal B.
\end{equation}
Denote by $\mathcal J(\mathcal T)$ the set of all Markov operators which commute with the Koopman 
operator associated with $\mathcal T$. It appears that if we consider weak-{\textasteriskcentered } topology on $\mathcal J(\mathcal T)$, then \eqref{operator} defines an affine homeomorphism $\Psi:J_2(\mathcal T)\to\mathcal J(\mathcal T)$.  For more information about joinings and Markov operators we refer to \cite{Glas}.

Consider 
the affine continuous map $\Pi_{1,3}:J_3(\mathcal T)\to J_2(\mathcal T)$ given by
\begin{equation}\label{rzutpi}
\Pi_{1,3}(\la)(A\times B):=\la(A\times X\times B)\ \text{for any}\ A,B\in\mathcal B.
\end{equation}
In other words $\Pi_{1,3}(\la)$ is the projection of  
the joining $\la$ on the first and third 
coordinates. Analogously, we define $\Pi_{2,3}$, the projection on the second and third 
coordinates. Since $J_2(\mathcal T)$ and $\mathcal J(\mathcal T)$ are affinely homeomorphic, we can consider 
the affine continuous maps $\Psi\circ\Pi_{i,3}:J_3(\mathcal T)\to\mathcal J(\mathcal T)$ for $i=1,2$. Note that for any $t,s\in\re$ we have
\begin{equation}\label{doop}
\Psi\circ\Pi_{1,3}(\mu_{t,s})=T_{-t},\ \text{and}\ \Psi\circ\Pi_{2,3}(\mu_{t,s})=T_{-s}.
\end{equation}
For $i\in\{1,2\}$ let $\sigma_i:\re^2\to\re$ be the projection on 
the $i$-th coordinate. Then for every $P\in\mathcal P(\re^2)$, we also have
\begin{equation}\label{rzutop}
\begin{split}
&\Pi_{i,3}\Big(\int_{\re^2}\mu_{-t,-s}\,dP(t,s)\Big)=\int_{\re^2}\Pi_{i,3}(\mu_{-t,-s})\,dP(t,s)\text{ and }\\
\Psi\circ&\Pi_{i,3}\Big(\int_{\re^2}\mu_{-t,-s}\,dP(t,s)\Big)=\int_{\re}T_{t}d\big((\sigma_i)_*P\big)(t),
\end{split}
\end{equation}
for $i=1,2$.
\subsection{Special flows.}
\label{sec:specialflows}
Let $(X,\mathcal B,\mu)$ be a standard Borel probability space.
Let $T:X\to X$ be an ergodic 
$\mu$-preserving automorphism. Let $f\in L^1([0,1))$ be positive and for any $n\in\z$ consider
\[
f^{(n)}(x) 
:= \begin{cases}
\sum_{i=0}^{n-1}f(T^ix)&\text{ if }n\ge 1\\
0&\text{ if }n=0\\
-\sum_{i=n}^{-1}f(T^ix)&\text{ if }n\le -1.
\end{cases}
\]
Define $X^f:=((x,r);x\in X,\ 0\le r<f(x))$ and on $X^f$ consider 
the measure $\mu\otimes \operatorname{Leb}(|_{X^f}$. 
The special flow $T^f=\{T^f_t\}_{t\in\re}$ is 
the measure preserving flow acting on $X^f$ by the formula
\[
T^f_t(x,r):=(T^nx,r+t-f^{(n)}(x)),
\]
where $n\in\z$ is unique, such that $f^{(n)}(x)\le r+t<f^{(n+1)}(x)$. We say that $f$ is the \emph{roof function} and $T$ is the \emph{base} of 
the special flow. In view of Ambrose Representation Theorem (see \cite{Ambr}), every ergodic flow is measure theoretically isomorphic to a special flow. Such a special flow is called \emph{a special representation} of the flow.
In this paper we deal with special flows whose roof functions are piecewise continuous and 
whose bases are interval exchange transformations. We always assume that roof functions are right-continuous and that the left limits exist. If a piecewise continuous bounded function $f$ has a discontinuity at $x$, then 
the \emph{jump at }$x$ is the number $d:=f(x)-\lim_{y\to x^-}f(x)$.

\subsection{Space of flows}
Let $(X,\mathcal B(X),\mu)$ be a standard Borel probability space. By $\operatorname{Flow}(X)$ we denote the set of all measure preserving flows on $X$. Let $\mathcal T=\{T_t\}_{t\in\re}\in \operatorname{Flow}(X)$, $A\in\mathcal B(X)$ and $\ep>0$. Let
\[
U(\mathcal T,A,\ep):=\{\mathcal S=\{S_t\}_{t\in\re}\in \operatorname{Flow}(X);\sup_{t\in[-1,1]}\mu(T_tA\triangle S_tA)<\ep\}.
\]
It appears that the family of sets of the above form gives a subbase of a topology, and $\operatorname{Flow}(X)$ endowed with this topology is a Polish space.

Let $(Y,d)$ be a metric space. It follows that a map $F:Y\to \operatorname{Flow}(X)$ is continuous if for any $y\in Y$ and $A\in\mathcal B(X)$ we have
\begin{equation}\label{ciagdoflow}
\text{for any }\ep>0\text{ there exists }\de>0\text{ such that } d(y,z)<\de\Rightarrow F(z)\in U(F(y),A,\ep).
\end{equation}
By using the fact that for any $A_1,B_1,A_2,B_2\in\mathcal B(X)$ we have
\[
A_1\triangle B_1=A_1^c\triangle B_1^c\text{ and }(A_1\cup B_1)\triangle(A_2\cup B_2)\subseteq(A_1\triangle A_2)\cup(B_1\triangle B_2),
\]
we can prove that 
the set of all $A\in \mathcal B(X)$, for which for every $\ep>0$ there exists $\de_A$ such that \eqref{ciagdoflow} is satisfied, form an algebra. 
By using the triangle inequality
\[
\mu(A\triangle B)\le\mu(A\triangle C)+\mu(B\triangle C)\text{ for }A,B,C\in\mathcal B(X),
\]
 we can prove that this algebra is closed under taking the countable union of increasing family of sets and thus, it is a $\sigma$-algebra.
Hence it is enough to check \eqref{ciagdoflow} for a family of sets which generates $\mathcal B(X)$.

All non-atomic standard Borel probability spaces are measure theoretically isomorphic (see Theorem 3.4.23 in \cite{Borel}). Let $(X_1,\mathcal B(X_1),\mu_1)$ and $(X_2,\mathcal B(X_2),\mu_2)$ be  standard Borel non-atomic probability spaces and let 
$H:X_1\to X_2$ be 
some isomorphism. Then $\operatorname{Flow}(X_1)$ and $\operatorname{Flow}(X_2)$ can be identified by a homeomorphism $\phi:\operatorname{Flow}(X_1)\to \operatorname{Flow}(X_2)$ given by the formula
\[
\phi(\mathcal T):=H\circ\mathcal T\circ H^{-1}.
\]
\begin{uw}\label{transport}
	To prove that $F:(Y,d)\to \operatorname{Flow}(X_1)$ is continuous, we can instead prove that $\phi\circ F:(Y,d)\to \operatorname{Flow}(X_2)$ is continuous. In other words, we need to prove that for every $y\in Y$ and $A\in\mathcal D\subset\mathcal B(X_2)$, where $\mathcal D$ generates $\mathcal B(X_2)$, we have
	\[
	\text{for any }\ep>0\text{ there exists }\de>0\text{ such that } d(y,z)<\de\Rightarrow \phi\circ F(z)\in U(\phi\circ F(y),A,\ep).
	\]
\end{uw}

\subsection{Interval exchange transformations}
Let $\mathcal A$ be an alphabet of $d$ elements. Let now $\epsilon\in\{0,1\}$ and let $\pi_\epsilon:\mathcal A\to\{1,\ldots,d\}$ be bijections. We will now consider 
a permutation $\pi$ as a pair $\{\pi_0,\pi_1\}$ where $\pi_0(\al)$ corresponds to the position of letter $\al$ before permutation, while $\pi_1(\al)$ defines the position of $\al$ after permutation. We say that a  permutation $\pi$ is {\it irreducible} if there is no $1\le k<d$ such that
\[
\pi_{1}\circ\pi_0^{-1}\big(\{1,\ldots,k\}\big)=\{1,\ldots,k\}.
\]
In this paper we will only deal with irreducible permutations, so this assumption will usually be omitted. We say that the permutation is \emph{symmetric} if
\[
\pi_1(\al)=d+1-\pi_0(\al)\ \text{for every}\ \al\in\mathcal A.
\]
Note that a symmetric permutation is always irreducible.

The 
intervals that we will now consider are always left-side closed and right-side open unless told otherwise. Let $I$ be an interval equipped with 
its Borel $\sigma$-algebra and Lebesgue measure $\operatorname{Leb}($.  Without losing generality, we can assume that the left endpoint of $I$ is $0$. Let $\{I_\al\}_{\al\in\mathcal A}$ be a partition of $I$ into $d$ intervals, where $I_\al$ has length $\la_\al\geq 0$. We will denote $\la:=\{\la_\al\}_{\al\in\mathcal A}$ {\it the length vector} and obviously we have $|\la|:=\sum_{\al\in\mathcal A}\la=\operatorname{Leb}(I)$.
The {\it interval exchange transformation (IET)} $T_{\pi,\la}:I\to I$ is 
the automorphism which permutes intervals $I_\al$ according to the permutation $\pi$.
Let now $\Omega_\pi:=[(\Omega_\pi)_{\al\beta}]_{\al,\beta\in\mathcal A}$ be 
the $d\times d$ matrix given by the following formula
\begin{equation}
(\Omega_\pi)_{\al\beta}:=
\begin{cases}
+1&\text{ if }\pi_0(\al)<\pi_0(\beta)\text{ and }\pi_1(\al)>\pi_1(\beta);\\
-1&\text{ if }\pi_0(\al)>\pi_0(\beta)\text{ and }\pi_1(\al)<\pi_1(\beta);\\
0&\text{  otherwise}.
\end{cases}
\end{equation}
We will say that $\Omega_\pi$ is 
the {\it translation matrix} of $T^{\pi,\la}$. The name of the matrix is derived from the fact that $T_{\pi,\la}$ acts on an interval $I_{\al}$ as a translation by number $\sum_{\beta\in\mathcal A}(\Omega_\pi)_{\al\beta}\la_\beta$.

Let $\partial I_\al$ be the left endpoint of $I_\al$. We say that 
the IET $T_{\pi,\la}$ satisfies {\it Keane's condition} if
\[
T^m_{\pi,\la}(\partial I_\al)=\partial I_\beta\text{ for }m>0\text{ implies }\al=\pi_1^{-1}(1),\beta=\pi_0^{-1}(1)\text { and }m=1.
\]
It is easy to see that $T_{\pi,\la}$ satisfies Keane's condition 
whenever $\la$ is a rationally independent vector
(that is there is no nontrivial integer linear combination of numbers $\la_\al$, which will give 
a rational number).

Denote by $S_{\mathcal A}^0$ the set of all irreducible permutations of 
$\mathcal A$. We may consider the space $S_{\mathcal A}^0\times \re_{\ge0}^{\mathcal A}$ of all IETs of $d$ intervals. Define 
the operator $R:S_{\mathcal A}^0\times \re_{\ge0}^{\mathcal A}\to S_{\mathcal A}^0\times \re_{\ge0}^{\mathcal A}$, such that $R(\pi,\la)=R(T_{\pi,\la})$ is 
the first return map of $T_{\pi,\la}$ to the interval $[0,|\la|-\min\{\la_{\pi_0^{-1}(d)},\la_{\pi_1^{-1}(d)}\})$. The operator $R$ is called the {\it Rauzy-Veech induction} (or {\it righthand side Rauzy-Veech induction}). The {\it Rauzy classes} are the minimal subsets of $S_{\mathcal A}^0$ which are invariant under the induced action of $R$ on $S_{\mathcal A}^0$.

Let
\[
l:\{1,\ldots,d\}\to\{1,\ldots,d\}\text{ be given by }l(i)=d+1-i.
\]
The function $l$ acts on $S_{\mathcal A}^0$ by mapping $\{\pi_0,\pi_1\}$ onto $\{l\circ\pi_0,l\circ\pi_1\}$. The {\it extended Rauzy classes} are minimal subsets of $S_{\mathcal A}^0$ which are invariant under $R$ and action of $l$. We have the following result.
\begin{tw}[Rauzy]\label{twpierost}
	Any Rauzy class of permutations of $d\ge 2$ elements contains at least one permutation $\pi$ such that
	\[
	\pi_1\circ\pi_0^{-1}(1)=d\quad \text{and}\quad\pi_1\circ\pi_0^{-1}(d)=1.
	\]
\end{tw}
\subsection{Translation surfaces and 
moduli space.}
Let $M$ be an orientable compact and connected topological surface of genus $g\ge 1$. Let $\Sigma:=\{A_1,\ldots,A_s\}$ be a finite subset of singular points in $M$. Let $\kappa:=(\kappa_1,\ldots,\kappa_s)$ be a
vector of positive integers satisfying $\sum_{i=1}^s\kappa_i=2g-2$.
A {\it translation structure} on $M$ is a maximal atlas $\zeta$ on $M\setminus\Sigma$ of charts by open sets of $\mathbb C$ such that any coordinate change between charts is a translation
and for each $1\le i\le s$ there exists a neighbourhood $V_i\subset M$ of $A_i$, a neigbourhood $W_i\subset \mathbb C$ of $0$ and a ramified covering $\pi:(V_i,A_i)\to(W_i,0)$ of degree $\kappa_i+1$ such that each injective restriction of $\pi$ is a chart of $\zeta$.
On $(M,\zeta)$ we can consider a holomorphic 1-form which in the local coordinates can be written as $dz$. We will denote this form also by $\zeta$.
We identify the associated 2-form $\frac{i}{2}\zeta\wedge \bar{\zeta}$ with the Lebesgue measure $\la_\zeta$ on $M$.
Moreover, the quadratic form $|\zeta|^2$ yields a Riemannian metric $(M,\zeta)$. By $d_\zeta$ we denote the distance given by this metric.
We also consider on $(M,\zeta)$ a vertical translation flow, denoted by $\mathcal T^\zeta=\{\mathcal T^\zeta_t\}_{t\in\re}$, which in 
local coordinates is a unit speed flow in the vertical direction. The flow $\mathcal T^\zeta$ preserves $\la_\zeta$ and thus can be viewed as an element of $\operatorname{Flow}(M,\la_\zeta)$.

In the set of all translation structures on $M$ we identify 
the structures $\zeta_1$ and $\zeta_2$ if there exists a homeomorphism $H:M\to M$ which sends singular points of $\zeta_1$ onto singular points of $\zeta_2$ and $\zeta_1=H^*\zeta_2$.
In terms of local coordinates, $H$ is locally a translation.
This is an equivalence relation and its equivalence classes form  
the {\it moduli space} denoted by $Mod(M)$. The moduli space can be divided into subsets called {\it strata} $\mathcal M(M,\Sigma,\kappa)=\mathcal M(\kappa)$, for which the vector of degrees of singularities is given by $\kappa$.
It appears that each such stratum $\mathcal M(M,\Sigma,\kappa)$ is a complex orbifold (see \cite{Veech}) and has a finite number of connected components (see \cite{KonZo}).
On $\mathcal M$ we can consider an action of $SL(2,\re)$. It is given by composing the charts of a translation surface with 
a given linear map.
The strata are invariant under the action of $SL(2,\re)$.
It is worth noting that in particular for every $\theta\in\re/\z$ we can apply the rotation $r_\theta$ by $\theta$ to the translation structure and almost every angle $\theta$ yields no saddle connection of a vertical flow.

Let $\pi=(\pi_0,\pi_1)$ be a permutation of the alphabet $\mathcal A$ of $d>1$ elements and let $\la\in\re^{\mathcal A}_{\ge0}$. Consider also 
$\tau\in\re^{\mathcal A}$ such that for each $1\le k<d$ we have
\[
\sum_{\{\al;\pi_0(\al)\le k\}}\tau_\al>0\text{ and }\sum_{\{\al;\pi_1(\al)\le k\}}\tau_\al<0.
\]
Moreover we require that
\[
\pi_i(\al)=\pi_i(\beta)+1\ \wedge\ \la_\al=\la_\beta=0\ \Rightarrow\ \tau_\al\cdot\tau_\beta>0\ \text{for all}\ i=0,1\text{ and }\al,\beta\in\mathcal A.
\]
For a fixed permutation $\pi$, we denote by $\Theta_\pi$ the set of triples $(\pi,\la,\tau)$ satisfying the above conditions.

Consider the 
polygonal curve in $\mathbb C$ obtained by connecting the points $0$ and $\sum_{i\le k}(\la_{\pi^{-1}_0(i)}+i\tau_{\pi^{-1}_0(i)})$ for $k=1,\ldots,d$, using the 
line segments.
Analogously we can consider the 
polygonal curve obtained by connecting the points $0$ and $\sum_{i\le k}(\la_{\pi^{-1}_1(i)}+i\tau_{\pi^{-1}_1(i)})$ for $k=1,\ldots,d$.
These two 
polygonal curves define a polygon with $d$ pairs of parallel sides. By identifying those sides we obtain a translation surface $M$, with $\Sigma$ being the set of vertices of the polygon (some of them may be identified).
We denote by $M(\pi,\la,\tau)$ the translation structure given by $(\pi,\la,\tau)$.

It appears that whenever $\mathcal T^\zeta$ admits no saddle-connections, $\zeta$ can be viewed as $M(\pi,\la,\tau)$ for some $(\pi,\la,\tau)\in\Theta_\pi$,
with $\pi$ being some permutation (see 
\textit{e.g.} \cite{Yoccoz}). Moreover we can consider $(\pi,\la,\tau)\in\Theta_\pi$ as local coordinates in the neighbourhood of such $\zeta$ in the corresponding stratum. Since almost every rotation yields no saddle-connections, to obtain local coordinates in the neighbourhood of $\zeta$ for which $\mathcal T^\zeta$ has a saddle connection, we can use the rotation to obtain local coordinates around rotated form and then rotate it back.

 Kontsevich and Zorich in \cite{KonZo} gave a complete characterization of connected components of strata in 
 the moduli space.
 In particular, they showed that each stratum $\mathcal M(2g-2)$ and $\mathcal M(g-1,g-1)$, where $g$ is the genus of 
 the surface, contains exactly one so-called \emph{hyperelliptic} connected component, which we denote by $\mathcal M^{hyp}(2g-2)$ and $\mathcal M^{hyp}(g-1,g-1)$ respectively.
 For every hyperelliptic component $C\subset \mathcal M$, there exists an involution $\phi:M\to M$ such that for every $\zeta\in C$ we have $\phi^*\zeta=-\zeta$. In particular we have the following remark.
 \begin{uw}\label{hypisom}
 For every hyperelliptic connected component $C\subset\mathcal M$ and for every $\zeta\in G$, the vertical flow on $(M,\zeta)$ is isomorphic with its inverse.
 	\end{uw}

 It appears that 
 the connected components of the moduli space can be described by 
 the Rauzy classes of permutations.  Let us recall first  the notion of non-degenericity, as introduced by Veech.
 We say that a permutation $\pi=\{\pi_0,\pi_1\}$ of $\mathcal A$ is \emph{degenerate} if one of the following conditions is satisfied:
 \begin{align}
 \quad &\pi_1\circ\pi_0^{-1}(j+1)=\pi_1\circ\pi_0^{-1}(j)+1\text{ for some } 1\le j< d;\label{deg1}\\
 \quad  &\pi_1\circ\pi_0^{-1}(\pi_0\circ\pi_1^{-1}(d)+1)=\pi_1\circ\pi_0^{-1}(d)+1\label{deg2}\\
 \quad &\pi_0\circ\pi_1^{-1}(1)-1=\pi_0\circ\pi_1^{-1}(\pi_1\circ\pi_0^{-1}(1)-1)\label{deg3}\\
 \quad&\pi_0\circ\pi_1^{-1}(d)=\pi_0\circ\pi_1^{-1}(1)-1\text{ and } \pi_1\circ\pi_0^{-1}(d)=\pi_1\circ\pi_0^{-1}(1)-1\label{deg4}.
 \end{align}
 Otherwise the permutation is called \emph{non-degenerate}. The property of non-degenericity is invariant under the action of 
 the Rauzy-Veech induction.
 The importance of this notion is given by the following theorem.
 \begin{tw}[Veech]
 	The extended Rauzy classes of nondegenerate
 	permutations are in 
 	one-to-one correspondence with the connected
 	components of the strata in the moduli space.
 \end{tw}
In view of the above theorem, for each genus $g\ge 2$, the hyperelliptic components $\mathcal M^{hyp}(2g-2)$ and $\mathcal M^{hyp}(g-1,g-1)$ correspond to the extended Rauzy classes of symmetric permutations of $2g$ and $2g-1$ elements respectively.
\begin{uw}\label{less5}
Connected components which are associated with extended Rauzy graphs of permutations of $d\le 5$ elements are hyperelliptic.
	\end{uw}
For 
a given extended Rauzy class $\mathcal R$, let $C_{\mathcal R}$ be its associated connected component of the moduli space.
 Then for any $\pi\in\mathcal R$ the map $M:\Theta_\pi\to C_{\mathcal R}$ given by $(\pi,\la,\tau)\mapsto M(\pi,\la,\tau)$ is continuous and the range of the map $M$ is dense in $C_{\mathcal R}$.
  Moreover, recall that, due to Theorem \ref{twpierost}, for every connected component of the moduli space we can find a permutation $\bar\pi$ belonging to the corresponding extended Rauzy class,
  satisfying 
  $\bar\pi_1\circ\bar\pi_0^{-1}(d)=1$ and 
  $\bar\pi_0\circ\bar\pi_1^{-1}(d)=1$. Hence, to prove that some condition is satisfied for a dense set of translation structures in $C_{\mathcal R}$, it is enough to prove that it holds for translation structures, whose associated polygonal parameters belong to a dense subset  of $\Theta_{\bar\pi}$.

Let $\mathcal{R}$ be any extended Rauzy class.  Let us  consider a transformation
$\tilde R:\bigcup_{\pi\in \mathcal{R}}\Theta_\pi\mapsto \bigcup_{\pi\in \mathcal{R}}\Theta_\pi$ called a  \emph{polygonal Rauzy Veech induction} (or \emph{righthand side polygonal Rauzy Veech induction}) which
yields different parameters of a translation surface.

Let $\pi\in \mathcal{R}$ and let $(\pi,\la,\tau)\in\Theta_\pi$. Assume that $\la_{\pi_0^{-1}(d)}\neq\la_{\pi_1^{-1}(d)}$. If $\la_{\pi_0^{-1}(d)}<\la_{\pi_1^{-1}(d)}$, then for any $a\in\mathcal A$ define
\begin{align*}
\tilde\pi_0(a)&:=\begin{cases}
\pi_0(a)&\text{ if }\quad\pi_0(a)\le\pi_0(\pi_1^{-1}(d));\\
\pi_0(\pi_1^{-1}(d))+1&\text{ if }\quad \pi_0(a)=d;\\
\pi_0(a)+1&\text{ if }\quad\pi_0(\pi_1^{-1}(d))<\pi_0(a)\le d-1,
\end{cases}
\\
\tilde\pi_1(a)&:=\pi_1(a),\\
\tilde{\la}_a&:=\begin{cases}
\la_{\pi_1^{-1}(d)}-\la_{\pi_0^{-1}(d)}&\text{ if }\quad\pi_1(a)=d;\\
\la_a&\text{ otherwise},
\end{cases}
\\
\tilde\tau_a&:=\begin{cases}
\tau_{\pi_1^{-1}(d)}-\tau_{\pi_0^{-1}(d)}&\text{ if }\quad\pi_1(a)=d;\\
\tau_a&\text{ otherwise}.
\end{cases}
\end{align*}
Analogously, if $\la_{\pi_0^{-1}(d)}>\la_{\pi_1^{-1}(d)}$, we define
\begin{align*}
\tilde\pi_0(a)&:=\pi_0(a),
\\
\tilde\pi_1(a)&:=\begin{cases}
\pi_1(a)&\text{ if }\quad\pi_1(a)\le\pi_1(\pi_0^{-1}(d));\\
\pi_1(\pi_0^{-1}(d))+1&\text{ if }\quad \pi_1(a)=d;\\
\pi_1(a)+1&\text{ if }\quad\pi_1(\pi_0^{-1}(d))<\pi_1(a)\le d-1,
\end{cases}
\\
\tilde{\la}_a&:=\begin{cases}
\la_{\pi_0^{-1}(d)}-\la_{\pi_1^{-1}(d)}&\text{ if }\quad\pi_0(a)=d;\\
\la_a&\text{ otherwise},
\end{cases}
\\
\tilde\tau_a&:=\begin{cases}
\tau_{\pi_0^{-1}(d)}-\tau_{\pi_1^{-1}(d)}&\text{ if }\quad\pi_0(a)=d;\\
\tau_a&\text{ otherwise}.
\end{cases}
\end{align*}
We define $\tilde R$ by setting $\tilde R(\pi,\la,\tau):=(\tilde\pi,\tilde\la,\tilde\tau)$. It is defined almost everywhere on $\bigcup_{\pi\in \mathcal{R}}\Theta_\pi$ and if $M(\pi,\la,\tau)$ admits no saddle connection, it can be iterated indefinitely.
Similarly, we can also define a left hand side polygonal Rauzy Veech induction. Note that the polygons derived from $(\pi,\la,\tau)$ and $(\tilde{\pi},\tilde{\la},\tilde{\tau})$ represent the same translation surface, i.e. $M(\pi,\la,\tau)=M(\tilde\pi,\tilde\la,\tilde\tau)$.
Indeed, the latter is obtained from $(\pi,\la,\tau)$ by cutting out the triangle formed by the last top side and the last bottom side and gluing it to a side of a polygon which is identified with one of the two sides forming the triangle.

 Every $\zeta\in C_\pi$ which does not have vertical saddle-connections can be represented as $M(\pi,\la,\tau)$, for some $(\pi,\la,\tau)\in\Theta_\pi$. We can consider 
 the metric on the neighbourhood of $M(\pi,\la,\tau)$ on $\mathcal M(M,\Sigma,\kappa)$ given by
 \[
 d_{Mod}\big((\pi,\la',\tau'),(\pi,\la'',\tau'')):=\sum_{a\in\mathcal A}(|\la_a'-\la_a''|+|\tau_a'-\tau_a''|).
 \]
 If $\zeta$ admits vertical saddle-connections, we can apply $r_\theta$ for some $\theta\in\re/\z$, so that $r_\theta^*\zeta$ does not have vertical saddle-connections and then define a metric in the neighbourhood of $\zeta$.

For any $\zeta=M(\pi,\la,\tau)\in\mathcal M(M,\Sigma,\kappa)$ we can consider a special representation of the vertical flow on $(M,\zeta)$. The basis of this special flow is the IET $T_{\pi,\la}$ and the roof function $h$ is positive and constant over exchanged intervals. Hence $h$ can be considered as a vector $(h_a)_{a\in\mathcal A}\in\re_{>0}^{\mathcal A}$, where $h_a$ is the value of $h$ over the exchanged interval labelled by $a$. The vector $h$ is given by the formula
\begin{equation}\label{roof}
h=-\Omega_\pi\tau,
\end{equation}
where $\Omega_\pi$ is 
the translation matrix of $(\pi,\la)$.
This gives rise to  new local coordinates of the moduli space. In particular, the polygonal Rauzy-Veech induction receives a new form.
Namely, if $\la_{\pi_0^{-1}(d)}\neq \la_{\pi_1^{-1}(d)}$ then $\tilde R(\pi,\la,h)=(\tilde\pi,\tilde\la,\tilde h)$, where the formulas for $\tilde\pi$ and $\tilde\la$ remain unchanged  and for any $a\in\mathcal{A}$ we take
\[
\tilde h_a:=\begin{cases}
h_{\pi_0^{-1}(d)}+h_{\pi_1^{-1}(d)}&\text{if }a=\pi_0^{-1}(d)\text{ and }\la_{\pi_0^{-1}(d)}<\la_{\pi_1^{-1}(d)};\\
h_{\pi_0^{-1}(d)}+h_{\pi_1^{-1}(d)}&\text{if }a=\pi_1^{-1}(d)\text{ and }\la_{\pi_0^{-1}(d)}>\la_{\pi_1^{-1}(d)};\\
h_a&\text{otherwise.}
\end{cases}
\]

\section{Consequences of limit joinings}\label{sec:join}
In this section, we formulate a criterion for two flows to be disjoint, and a criterion for a flow to be weakly mixing.
Both criteria rely on the properties of the weak limit of 
some sequence of 3-off diagonal joinings.

For every measure $\la\in\mathcal P(X\times Y)$, we denote by $\la|_X$ and $\la|_Y$ the projections of $\la$ on $X$ and $Y$ respectively, that is for every measurable subsets $A\subseteq X$ and $B\subseteq Y$ we have
\[
\la|_X(A)=\la(A\times Y)\quad\text{and}\quad\la|_Y(B)=\la(X\times B).
\]

Let $\mathcal T=\{T_t\}_{t\in\re}$ and $\mathcal S=\{S_t\}_{t\in\re}$ be weakly mixing flows acting on standard Borel spaces $(X,\mathcal B,\mu)$ and $(Y,\mathcal C,\nu)$ respectively.
\begin{lm}\label{podst}
Let $\la\in J^e(\mathcal T, \mathcal S)$. Let $\rho\in J^e_2(\mathcal T\times \mathcal S,\la)$, which is defined on $X_1\times Y_1\times X_2\times Y_2$ with $X_1=X_2=X$ and $Y_1=Y_2=Y$.
Assume that for some $r,r'\in\re$ we have $\rho|_{X_1\times X_2}=\mu_{T_r}$ and $\rho|_{Y_1\times Y_2}=\nu_{S_{r'}}$. If $r\neq r'$ then $\la=\mu\otimes\nu$.
\end{lm}
\begin{proof}
First we prove that $\la=(T_r\times S_{r'})_*\la$. We show that (3) in  Remark \ref{graph} is satisfied for the $\pi$-system of product sets 
and the isomorphism $\phi:=T_{-r}\times S_{-r'}$ between $(X_1\times Y_1,\la)$ and $(X_2\times Y_2,\la)$. In other words, for every $A\in\mathcal B$ and $B\in\mathcal C$ we have 
\[
\rho(A\times B\times (T_{-r}\times S_{-r'})(A\times B)^c)=\rho((A\times B)^c\times (T_{-r}\times S_{-r'})(A\times B))=0.
\]
Indeed, recall that $\mu_r$ and $\nu_{r'}$ are graph joinings of $\mathcal T$ and $\mathcal S$ given by $T_{-r}$ and $S_{-r'}$ respectively. By Remark \ref{graph} this implies that for every $A\in\mathcal B$ and $B\in\mathcal C$ we have
\[
\mu_{r}(A\times T_{-r}A^c)=0\quad\text{and}\quad\nu_{r'}(B\times T_{-r}B^c)=0.
\]
Thus we obtain
\[
\begin{split}
\rho(A\times B&\times (T_{-r}\times S_{-r'})(A\times B)^c)=\rho(A\times B\times T_{-r}A^c\times S_{-r'}B)\\&+\rho(A\times B\times T_{-r}A^c\times S_{-r'}B^c)
+\rho(A\times B\times T_{-r}A\times S_{-r'}B^c)\\
&\le2\rho(A\times Y\times T_{-r}A^c\times Y)+\rho(X\times B\times X\times S_{-r'}B^c)\\
&=2\mu_r(A\times T_{-r}A^c)+\nu_{r'}(B\times S_{-r'}B^c)=0
\end{split}
\]
and
\[
\begin{split}
\rho((A\times B)^c&\times T_{-r}A\times S_{-r'}B)=\rho(A^c\times B\times T_{-r}A\times S_{-r'}B)\\&+\rho(A^c\times B^c\times T_{-r}A\times S_{-r'}B)
+\rho(A\times B^c\times T_{-r}A\times S_{-r'}B)\\
&\le2\rho(A^c\times Y\times T_{-r}A\times Y)+\rho(X\times B^c\times X\times S_{-r'}B)\\
&=2\mu_r(A^c\times T_{-r}A)+\nu_{r'}(B^c\times S_{-r'}B)=0.
\end{split}
\]
Hence we have proved that (3) in Remark \ref{graph} is satisfied for 
the $\pi$-system of product sets. Since $\rho\in J^e_2(\mathcal T\times \mathcal S,\la)$, in view of (2) in Remark \ref{graph} we get 
\[
\begin{split}
\la(A\times B)&=\rho(A\times B\times X\times Y)=\rho(X\times\ Y\times T_{-r}A\times S_{-r'}B)
\\&=\la(T_{-r}A\times S_{-r'}B)=(T_r\times S_{r'})_*\la(A\times B),
\end{split}
\]
for all $A\in\mathcal B$ and $B\in\mathcal C$. Since the $\pi$-system of product sets generates $\mathcal B\otimes\mathcal C$, we get that the measures $\la$ and $(T_r\times S_{r'})_*\la$ are equal.
By the $(\mathcal T\times \mathcal S)\,$-invariance of $\la$, we have that $\la$ is $(Id\times S_{r-r'})\,$-invariant.
By weak mixing of $\mathcal S$, $S_{r-r'}$ is ergodic whenever $r\neq r'$. Since $Id$ is disjoint with every ergodic transformation (see Remark \ref{ident}), we get $\la=\mu\otimes\nu$.
\end{proof}
\begin{prop}\label{kryterium}
Assume that for some real sequences $(a_n)_{n\in\n}$ and $(b_n)_{n\in\n}$ we have
\[
\mu_{a_n,b_n}\to(1-\al)\int_{\re^2}\mu_{-t,-u}dP(t,u)+\al\xi_1,
\]
and
\[
\nu_{a_n,b_n}\to(1-\al')\int_{\re^2}\nu_{-t,-u}dQ(t,u)+\al'\xi_2,
\]
for some $0\leq\al,\al'< 1$, measures $P,Q\in\mathcal P(\re^2)$ and $\xi_1\in J_3(\mathcal T)$, $\xi_2\in J_3(\mathcal S)$. Assume moreover, that there exists a set $B\in\mathcal B(\re^2)$, such that
\begin{equation}\label{setB}
(1-\al)P(B)-(1-\al')Q(B)>\al'.
\end{equation}
Then $\mathcal T$ and $\mathcal S$ are disjoint.

\end{prop}
\begin{uw}
The above proposition can be also proven in higher dimensional case, that is when we consider limits of joinings of higher rank.
\end{uw}
\begin{proof}[Proof of Proposition \ref{kryterium}]
Let $\xi_1=\int_{J_3^e(\mathcal T)}\rho^{\mathcal T} d \kappa_1(\rho^{\mathcal T})$ and $\xi_2=
\int_{J_3^e(\mathcal S)}\rho^{\mathcal S} d\kappa_2(\rho^{\mathcal S})$ be the ergodic decompositions of $\xi_1$ and $\xi_2$ respectively. Let also $\mathcal A_1$ be the set of $3$-off-diagonal joinings in $J_3^e(\mathcal T)$ and $\mathcal A_2$ be the set of $3$-off-diagonal joinings in $J_3^e(\mathcal S)$. In view of Souslin theorem the sets $\mathcal A_1$ and $\mathcal A_2$ are measurable. We can assume that $\kappa_1(\mathcal A_1)=\kappa_2(\mathcal A_2)=0$. Indeed, let $\beta:=1-\kappa_1(\mathcal A_1)\ge 0$ and $\beta':=1-\kappa_2(\mathcal A_2)\ge 0$. Then
\[
\xi_1=(1-\beta)\int_{\re^2}\mu_{-t,-u}dP'(t,u)+\beta\xi_1'
\]
and
\[
\xi_2=(1-\beta')\int_{\re^2}\nu_{-t,-u}dQ'(t,u)+\beta'\xi_2',
\]
where $\xi_1'\in J_3(\mathcal T)$ and $\xi_2'\in J_3(\mathcal S)$ do not have 3-off-diagonal joinings in their ergodic decomposition.
Then
\[
\begin{split}
\mu_{a_n,b_n}&\to(1-\al)\int_{\re^2}\mu_{-t,-u}dP(t,u)+\al\big((1-\beta)\int_{\re^2}\mu_{-t,-u}dP'(t,u)+\beta\xi_1'\big)\\
&=(1-\al\beta)\int_{\re^2}\mu_{-t,-u}d(\frac{1-\al}{1-\al\beta}P+\frac{\al(1-\beta)}{1-\al\beta}P')+\al\beta\xi_1'\\
&=(1-\al\beta)\int_{\re^2}\mu_{-t,-u}d\bar P+\al\beta\xi_1',
\end{split}
\]
where $\bar P=\frac{1-\al}{1-\al\beta}P+\frac{\al(1-\beta)}{1-\al\beta}P'$. Analogously
\[
\nu_{a_n,b_n}\to(1-\al'\beta')\int_{\re^2}\nu_{-t,-u}d\bar Q+\al'\beta'\xi_2',
\]
where $\bar Q=\frac{1-\al'}{1-\al'\beta'}Q+\frac{\al'(1-\beta')}{1-\al'\beta'}Q'$.
Then for the set $B$ satisfying \eqref{setB} we have
\[
\begin{split}
(1&-\al\beta)\bar P(B)-(1-\al'\beta')\bar Q(B)\\
&=(1-\al)P(B)+\al(1-\beta)P'(B)-(1-\al')Q(B)-\al'(1-\beta')Q'(B)\\
&>\al'+\al(1-\beta)P'(B)-\al'(1-\beta')Q'(B)\ge\al'-\al'(1-\beta')=\al'\beta'.
\end{split}
\]
It is enough then, to replace $P,Q$ by $\bar P,\bar Q$ and $\al,\al'$ by $\al\beta,\al'\beta'$ respectively.

Let $\la\in J^e(\mathcal T,\mathcal S)$. We show that $\la=\mu\otimes\nu$. Consider the sequence $\{\la_{a_n,b_n}\}_{n\in\n}$ in $J_3^e(\mathcal T\times \mathcal S, \la)$. By the compactness of $J_3(\mathcal T\times \mathcal S, \la)$ we have that $\la_{a_n,b_n}\to \eta$ weakly in $J_3(\mathcal T\times \mathcal S, \la)$, up to taking a subsequence. Moreover, by assumptions we have
\[
\eta|_{X_1\times X_2\times X_3}=(1-\al)\int_{\re^2}\mu_{-t,-u}dP(t,u)+\al\xi_1
\]
and
\[
\eta|_{Y_1\times Y_2\times Y_3}=(1-\al')\int_{\re^2}\nu_{-t,-u}dQ(t,u)+\al'\xi_2.
\]
Let $h^\mathcal T:\re^2\to\mathcal A_1$ and $h^\mathcal S:\re^2\to\mathcal A_2$ be given by  $h^\mathcal T(t,u) :=\mu_{-t,-u}$ and $h^\mathcal S(t,u):=\nu_{-t,-u}$. Then
\begin{equation}\label{ergo1}
\eta|_{X_1\times X_2\times X_3}=\int_{J_3^e(\mathcal T)}
\rho^{\mathcal T}\, d((1-\al)h^{\mathcal T}_*P+\al\kappa_1)(\rho^{\mathcal T}),
\end{equation}
and
\begin{equation}\label{ergo2}
\eta|_{Y_1\times Y_2\times Y_3}=\int_{J_3^e(\mathcal S)}
 \rho^{\mathcal S}\, d((1-\al') h^{\mathcal S}_*Q+\al'\kappa_2)(\rho^{\mathcal S}).
\end{equation}

Let now $\eta=\int_{J_3^e(\mathcal T\times \mathcal S,\la)}\psi d\kappa(\psi)$ be the ergodic decomposition of $\eta$. Then we have
\[
\eta|_{X_1\times X_2\times X_3}=\int_{J_3^e(\mathcal T\times \mathcal S,\la)}\psi|_{X_1\times X_2\times X_3} d\kappa(\psi),
\]
and
\[
\eta|_{Y_1\times Y_2\times Y_3}=\int_{J_3^e(\mathcal T\times \mathcal S,\la)}\psi|_{Y_1\times Y_2\times Y_3} d\kappa(\psi).
\]
Since $\psi\in J_3^e(\mathcal T\times\mathcal S)$, we have $\psi|_{X_1\times X_2\times X_3}\in J_3^e(\mathcal T)$ and $\psi|_{Y_1\times Y_2\times Y_3}\in J_3^e(\mathcal S)$. Consider $\Omega^\mathcal T:J_3^e(\mathcal T\times \mathcal S,\la)\to J_3^e(\mathcal T)$ and $\Omega^\mathcal S:J_3^e(\mathcal T\times \mathcal S,\la)\to J_3^e(\mathcal S)$ given by
\[
\Omega^\mathcal T(\psi)=\psi|_{X_1\times X_2\times X_3}\quad\text{ and }\quad\Omega^\mathcal S(\psi)=\psi|_{Y_1\times Y_2\times Y_3}.
\]
We have
\[
\eta|_{X_1\times X_2\times X_3}=\int_{J_3^e(\mathcal T)}\rho^\mathcal T d(\Omega_*^\mathcal T\kappa)(\rho^\mathcal T),
\]
and
\[
\eta|_{Y_1\times Y_2\times Y_3}=\int_{J_3^e(\mathcal S)}\rho^\mathcal S d(\Omega_*^\mathcal S\kappa)(\rho^\mathcal S).
\]
By comparing this with \eqref{ergo1} and \eqref{ergo2} and using the uniqueness of ergodic decomposition we obtain that
\begin{equation}\label{rown}
\Omega_*^\mathcal T\kappa=(1-\al)h_*^\mathcal T P+\al\kappa_1\quad\text{ and }\quad\Omega_*^\mathcal S\kappa=(1-\al')h_*^\mathcal S Q+\al'\kappa_2.
\end{equation}

Let now
\begin{multline*}
\mathcal A:=\{\psi\in J_3^e(\mathcal T\times \mathcal S,\la): \exists t,u,t',u'\in\re, (t,u)\neq(t',u'), \\
\psi|_{X_1\times X_2\times X_3}=\mu_{-t,-u},\psi|_{Y_1\times Y_2\times Y_3}=\nu_{-t',-u'}\}.
\end{multline*}

We now show that $\kappa(\mathcal A)>0$. For any measurable subsets $C\subset J_3^e(\mathcal T)$ and $D\subset J_3^e(\mathcal S)$ denote by $C\bar\times D$ the set of all $\psi\in J_3^e(\mathcal T\times \mathcal S,\la)$ such that $\psi|_{X_1\times X_2\times X_3}\in C$ and $\psi|_{Y_1\times Y_2\times Y_3}\in D$.

Assume that $\kappa(\mathcal A)=0$. Let $B$ be the set satisfying \eqref{setB}. If $(t,u)\in B$ then by the definition of $h^\mathcal T$ and $h^\mathcal S$ we have $\mu_{-t,-u}\in h^\mathcal T(B)$ and $\nu_{-t,-u}\in h^\mathcal S(B)$. Moreover $\kappa(\mathcal A)=0$ and $\ h^\mathcal T(B)\bar\times(\mathcal A_2\setminus h^\mathcal S(B))
\subset \mathcal A$ yield
\begin{equation}\label{enough}
\kappa(h^\mathcal T(B)\bar\times\mathcal A_2)=\kappa(h^\mathcal T(B)\bar\times h^\mathcal S(B)).
\end{equation}
Note that $\kappa_1(h^{\mathcal T}(B))\le \kappa_1(\mathcal A_1)=0$.
Hence, \eqref{rown} and \eqref{enough} implies
\begin{equation}\label{end1}
\begin{split}
(1-\al)P(B)&=(1-\al)h_*^\mathcal T P(h^\mathcal T(B))=[(1-\al)h_*^\mathcal TP+\al\kappa_1](h^\mathcal T(B))\\&=\Omega_*^\mathcal T\kappa(h^\mathcal T(B))=\kappa(h^\mathcal T(B)\bar\times J_3^e(\mathcal S))\\&=\kappa(h^\mathcal T(B)\bar\times \mathcal A_2)+\kappa(h^\mathcal T(B)\bar\times \mathcal (J_3^e(\mathcal S)\setminus\mathcal A_2))\\&=\kappa(h^\mathcal T(B)\bar\times h^\mathcal S(B))+\kappa(h^\mathcal T(B)\bar\times(J_3^e(\mathcal S)\setminus\mathcal A_2)).
\end{split}
\end{equation}
Analogously we 
also obtain
\begin{equation}\label{end2}
(1-\al')
Q(B)=(1-\al')h_*^\mathcal S Q(h^\mathcal S(B))=\kappa(h^\mathcal T(B)\bar\times h^\mathcal S(B))+\kappa((J_3^e(\mathcal T)\setminus\mathcal A_1)\bar\times h^\mathcal S(B)).
\end{equation}
Moreover, in view of \eqref{rown} we get
\[
\begin{split}
\kappa(h^\mathcal T(B)\bar\times(J_3^e(\mathcal S)\setminus\mathcal A_2))&\le\kappa(J_3^e(\mathcal T)\bar\times(J_3^e(\mathcal S)\setminus\mathcal A_2))\\&=\Omega_*^\mathcal S\kappa(J_3^e(\mathcal S)\setminus\mathcal A_2)=\al'\kappa_2(J_3^e(\mathcal S)\setminus\mathcal A_2)= \al'.
\end{split}
\]
Since $(1-\al)P(B)-(1-\al')Q(B)>\al'$, by substracting \eqref{end1} and \eqref{end2} we obtain
\[
\begin{split}
\al'&<(1-\al)h_*^\mathcal T P(h^\mathcal T(B))-(1-\al')h_*^\mathcal T Q(h^\mathcal S(B))\\&=(\kappa(h^\mathcal T(B)\bar\times h^\mathcal S(B))+\kappa(h^\mathcal T(B)\bar\times\mathcal (J_3^e(\mathcal S)\setminus\mathcal A_2)))\\&
-(\kappa(h^\mathcal T(B)\bar\times h^\mathcal S(B))+\kappa((J_3^e(\mathcal T)\setminus\mathcal A_1)\bar\times h^\mathcal S(B)))\\&=\kappa(h^\mathcal T(B)\bar\times(J_3^e(\mathcal S)\setminus\mathcal A_2))-\kappa((J_3^e(\mathcal T)\setminus\mathcal A_1)\bar\times h^\mathcal S(B))\le\al',
\end{split}
\]
which is a contradiction. This yields $\kappa(\mathcal A)>0$ and hence $\mathcal A$ is non-empty.
Therefore, there exists $\psi\in\mathcal A\subset J_3^e(\mathcal T\times \mathcal S,\la)$ such that
$\psi|_{X_1\times X_2\times X_3}=\mu_{t,u}$ and $\psi|_{Y_1\times Y_2\times Y_3}=\nu_{t',u'}$
with $(t,u)\neq(t',u')$. Assume that $t\neq t'$ (the case when $u\neq u'$ is analogous). 
Then $\phi:=\Pi_{1,3}(\psi)\in J_2^e(\mathcal T\times \mathcal S,\la)$ satisfies
\[
\phi|_{X_1\times X_3}=\mu_t\quad\text{ and }\quad\phi|_{Y_1\times Y_3}=\nu_{t'}.
\]
Thus, by Lemma \ref{podst}, $\la=\mu\otimes \nu$.

\end{proof}
The above criterion strengthens the results obtained in \cite{BerkFr}, that is the flows described in this paper are not only non-isomorphic with their inverses, but also disjoint. To prove the main result of this paper, we use the following simplified version of 
Proposition~\ref{kryterium}.
\begin{wn}\label{kryterium1}
	Let $\mathcal T=\{T_t\}_{t\in\re}$ and $\mathcal S=\{S_t\}_{t\in\re}$ be weakly mixing flows acting on the standard Borel spaces $(X,\mathcal B,\mu)$ and $(Y,\mathcal C,\nu)$ respectively.
	Assume that for some real sequences $(a_n)_{n\in\n}$ and $(b_n)_{n\in\n}$ we have
	\[
	\mu_{a_n,b_n}\to\int_{\re^2}\mu_{-t,-u}dP(t,u)\
	\text{and}\
	\nu_{a_n,b_n}\to\int_{\re^2}\nu_{-t,-u}dQ(t,u),
	\]
	for some measures $P,Q\in\mathcal P(\re^2)$. If $P\neq Q$, then $\mathcal T$ and $\mathcal S$ are disjoint.
\end{wn}

Let $\xi:\re^2\to\re$ be given by $\xi(t,u):=t-2u$. The following result gives 
a condition on limit joinings which imply weak mixing of a flow.

		\begin{prop}\label{weakmix}
			Let $\mathcal T=\{T_t\}_{t\in\re}$ be an ergodic flow on $(X,\mathcal B,\mu)$ and assume that there exists a real increasing sequence $\{b_n\}_{n\in\n}$, a real number $\rho\in[0,1)$ and a probability measure $P\in\mathcal P(\re^2)$ such that
			\begin{equation}\label{glownazbiez}
			\mu_{2b_n,b_n}\to(1-\rho)\int_{\re^2}\mu_{-t,-u}dP(t,u)+\rho\psi,
			\end{equation}
			for some $\psi\in J_3(\mathcal T)$. If $P$ is not supported on an affine lattice in $\re^2$ then $\mathcal T$ is weakly mixing.
			In particular, if 
			there exist two rationally independent real numbers $d_1$ and $d_2$ such that $d_1,d_2$ and $0$ are atoms of $\xi_*P$,
			then the flow $\mathcal T$ is weakly mixing.
		\end{prop}
		\begin{proof}
			Assume that $P$ is not supported on an affine lattice and the flow $\mathcal T$ is not weakly mixing. Then there exists a non-zero function $f\in L^2(X,\mu)$ and $a\in\re\setminus\{0\}$ such that
			\begin{equation}\label{wartwl}
			\forall t\in\re,\
			f\circ T_t=e^{-2\pi iat}f.
			\end{equation}
			Recall that 
			$\sigma_1:\re^2\to\re$ denotes the projection on the first coordinate. By applying $\Psi\circ\Pi_{1,3}$ (see \eqref{rzutpi}) to \eqref{glownazbiez} and using \eqref{doop} and \eqref{rzutop}, we obtain
			\[
			T_{2b_n}\to (1-\rho)\int_\re T_t dP_1(t)+\rho\Psi_1,
			\]
			where $P_1:=(\sigma_{1})_*P$ and $\Psi_1$ is a Markov operator. Let $\langle\cdot,\cdot\rangle$ be 
			the scalar product on $L^2(X,\mu)$. By \eqref{wartwl}, we get
			\[
			\|f\|^2=|\langle f,f\rangle|=|\langle f,e^{-2\pi iat}f\rangle|=|\langle f,f\circ T_t \rangle|=|\langle f,f\circ T_{2b_n} \rangle|	
			\]
			for every $n\in\n$. As $n\to\infty$, we get
			\[
			\|f\|^2=|\langle f,f\circ T_{2b_n} \rangle|=\Big|\Big\langle f,(1-\rho)\int_\re f\circ T_t dP_1(t)+\rho\Psi_1(f) \Big\rangle\Big|.		
			\]
			On the other hand by the fact that Markov operator is a contraction, we get
			\begin{align*}
			\Big|\Big\langle & f,(1-\rho)\int_\re f\circ T_t dP_1(t)+\rho\Psi_1(f) \Big\rangle \Big|\\ &\le(1-\rho)\Big|\Big\langle f,\int_\re f\circ T_t dP_1(t)\Big\rangle\Big|+\rho|\langle f,\Psi_1(f) \rangle|\\ &\le
			(1-\rho)\Big|\int_\re\langle f,f\circ T_t\rangle dP_1(t)\Big|+\rho\|f\|^2\\
			&=(1-\rho)\|f\|^2\Big|\int_\re e^{-2\pi iat} dP_1(t)\Big|+\rho\|f\|^2
			\end{align*}
			Thus we get
			\[
			\Big|\int_\re e^{-2\pi iat}dP_1(t)\Big|=1
			\]
			that is
			\[
			\int_\re e^{-2\pi iat}dP_1(t)=e^{-2\pi ib}\ \text{ for some }b\in\re.
			\]
			It follows that
			\[
			\int_\re e^{-2\pi i(at-b)}dP_1(t)=1.
			\]
			This implies
			\[
			P_1(\{t\in\re;\ at-b\in\z \})=1.
			\]
			Consider now $P_2:=(\sigma_2)_*P$. Analogously, by applying $\Psi\circ\Pi_{2,3}$ to \eqref{glownazbiez}, we get 
			\[
			P_2(\{u\in\re;\ au-c\in\z \})=1\ \text{ for some }c\in\re.
			\]
			Combining the two above results, we finally obtain
			\begin{equation}\label{Pkrata}
			P\big(\{(t,u)\in\re^2;\ a(t,u)-(b,c)\in\z^2 \}\big)=1,
			\end{equation}
			which is a contradiction with our assumption. Thus if $P$ is not supported on an affine lattice then the flow $\mathcal T$ is weakly mixing.
			
			Suppose now that $\xi_*P$ has atoms at points $0,d_1$ and $d_2$. Assume again that $\mathcal T$ is not weakly mixing 
			and that $e^{2\pi i a}$, $a\neq 0$, is an eigenvalue. By the definition of $\xi$, 
			the lines $(x,\frac{1}{2}(x-d_i))$ for $i=1,2$ and $(x,\frac{1}{2}x)$ have positive measure $P$. This together with \eqref{Pkrata} yields $x_0,x_1,x_2\in\re$, such that
			\[
			\begin{split}
			&a(x_0,\tfrac{1}{2}x_0)-(b,c)\in\z^2,\\
			&a(x_1,\tfrac{1}{2}(x_1-d_1))-(b,c)\in\z^2,\\
			&a(x_2,\tfrac{1}{2}(x_2-d_2))-(b,c)\in\z^2.
			\end{split}
			\]
			This implies
			\[
			\begin{split}
			&a(x_1-x_0,\tfrac{1}{2}(x_1-x_0)-\tfrac{1}{2}d_1)\in\z^2,\\
			&a(x_2-x_0,\tfrac{1}{2}(x_2-x_0)-\tfrac{1}{2}d_2)\in\z^2.
			\end{split}
			\]		
			By applying $\xi$ to the above, we get that $ad_1\in\z$ and $ad_2\in\z$. Since $a,d_1,d_2\neq0$, we get that here
			$(ad_1)d_2-(ad_2)d_1=0$ is a non-trivial integer combination of $d_1$ and $d_2$. By the rational independence of $d_1$ and $d_2$ this yields $a=0$. This is a contradiction, hence $\mathcal T$ is weakly mixing.
			\end{proof}

\section{Acceptable permutations}\label{sec:perm}

In this section, we establish a technical result concerning a particular non-degenerate permutation, which plays a key role in proving that our main result applies to all non-hyperelliptic connected components of 
the moduli space. In particular, in view of the Remark \ref{less5}, we assume that the alphabet we consider has $d\ge 6$ elements. Recall that in every Rauzy class we can fix a non-degenerate permutation  $\pi=\{\pi_0,\pi_1\}$ satisfying
	\begin{equation}\label{pierost}
		\pi_1\circ\pi_0^{-1}(1)=d\quad \text{and}\quad\pi_1\circ\pi_0^{-1}(d)=1.
	\end{equation}
	We have the following theorem.
\begin{prop}\label{acperm}
In every Rauzy class corresponding to a non-hyperelliptic connected component of 
the moduli space $\mathcal M$, there exists a permutation $\pi=\{\pi_0,\pi_1\}$ satisfying \eqref{pierost} such that there exist distinct symbols $\al_1,\al_2,\g_1,\g_2\in\mathcal A\setminus\{\pi_0^{-1}(1),\pi_0^{-1}(d)\}$ satisfying 
the three following properties
\begin{align*}
	&\Omega_{\al_1\al_2}=\Omega_{\al_2\al_1}=0,\\
	&\Omega_{\al_1\g_2}\Omega_{\al_2\g_1}= 0 \numberthis\label{tezperm}\\
	&\Omega_{\al_1\g_1}\Omega_{\al_2\g_2}\neq0,
	\end{align*}
where $\Omega:=\Omega_\pi$ is 
the associated translation matrix.
\end{prop}
\begin{proof}
Let $\pi=\{\pi_0,\pi_1\}$ be a non-degenerate permutation satisfying \eqref{pierost}
that belongs to a Rauzy class associated with a non-hyperelliptic connected component.
Then it is not symmetric, 
hence 
its translation matrix $\Omega$ contains zero entries outside the diagonal.
Indeed, assume contrary to our claim that
\[
\pi_0(\al)<\pi_0(\beta)\Leftrightarrow\pi_1(\al)>\pi_1(\beta)\text{ for all }\al,\beta\in\mathcal A.
\]
Then for every $\al\in\mathcal A$
\[
\pi_1(\al)=\#\{\beta\in\mathcal A; \pi_1(\beta)<\pi_1(\al)\}+1=\#\{\beta\in\mathcal A; \pi_0(\beta)>\pi_0(\al)\}+1=d-\pi_0(\al)+1.
\]
Hence $\pi$ is a symmetric permutation.

 We need to consider two cases separately.

 \noindent\textbf{Case 1.} Assume first that there exists a symbol $\al\in\mathcal{A}$ such that for all symbols $\beta\in\mathcal{A}$ with $1<\pi_0(\beta)<d$ we have
\[
\pi_0(\de)< \pi_0(\al)\Leftrightarrow \pi_1(\de)< \pi_1(\al)
\]
that is
\begin{equation}\label{0wiersz}
\Omega_{\al\beta}=0\quad\text{for all}\quad\beta\in\mathcal{A}\setminus\{\pi_0^{-1}(1),\pi_0^{-1}(d)\}.
\end{equation}
Since $\pi$ is non-degenerate, there exist symbols $\al_1,\g_1$ such that
\[
1<\pi_0(\al_1)<\pi_0(\g_1)<\pi_0(\al)\quad\text{and}\quad 1<\pi_1(\g_1)<\pi_1(\al_1)<\pi_1(\al).
\]
Otherwise, $\pi$ satisfies \eqref{deg2} and hence, it is degenerate. Similarly, there exist symbols $\al_2,\g_2$ such that
\[
d>\pi_0(\al_2)>\pi_0(\g_2)>\pi_0(\al)\quad\text{and}\quad d>\pi_1(\g_2)>\pi_1(\al_2)>\pi_1(\al).
\]
Otherwise, $\pi$ satisfies \eqref{deg3} and it is again degenerate. Thus we have
\[
\Omega_{\al_1\al_2}=\Omega_{\al_2\al_1}=\Omega_{\al_1\g_2}=\Omega_{\al_2\g_1}=0\ \text{ and }\ \Omega_{\al_1\g_1}=1\ \text{ and }\ \Omega_{\al_2\g_2}=-1,
\]
which is the desired property. Hence $\al_1,\al_2,\g_1,\g_2$ are the desired symbols.

\noindent\textbf{Case 2. }Assume now, that there are no symbols satisfying $\eqref{0wiersz}$. Since there are zeroes outside the diagonal in $\Omega_\pi$, there exist two distinct symbols $\al_1,\al_2\in\mathcal{A}$ such that $\Omega_{\al_1\al_2}=\Omega_{\al_2\al_1}=0$.

\noindent\textbf{Case 2a. }Suppose first that the rows of $\Omega_\pi$ corresponding to $\al_1$ and $\al_2$ are not identical. Then there exists 
a symbol $\g$ such that $\Omega_{\al_1\g}\neq 0$ and $\Omega_{\al_2\g}=0$ or $\Omega_{\al_2\g}\neq 0$ and $\Omega_{\al_1\g}=0$. Assume that the first case holds (the second is done analogously) and set $\g_1:=\g$.
Note that $\g_1\in\mathcal A\setminus\{\pi_0^{-1}(1),\pi_0^{-1}(d)\}$.
Since $\al_2$ does not satisfy \eqref{0wiersz}, there exist two $\g_2\in\mathcal A\setminus\{\pi_0^{-1}(1),\pi_0^{-1}(d)\}$, $\g_2\neq\g_1$, such that $\Omega_{\al_2\g_2}\neq 0$. Thus we obtain \eqref{tezperm}.

 \noindent\textbf{Case 2b. }Suppose that the rows of 
the matrix $\Omega$ corresponding to indices $\al_1,\al_2$ are identical. Then there are no indices $\g_1,\g_2$ such that $\al_1,\al_2,\g_1,\g_2$ satisfy \eqref{tezperm}. We show that there is a different set of symbols satisfying \eqref{tezperm}.  Note that all symbols $\beta$ such that
\begin{equation}\label{miedzy1}
\pi_0(\al_1)<\pi_0(\beta)<\pi_0(\al_2)
\end{equation}
satisfy
\begin{equation}\label{miedzy2}
\pi_1(\al_1)<\pi_1(\beta)<\pi_1(\al_2).
\end{equation}
Otherwise only one of the entries $\Omega_{\al_1\beta}$ and $\Omega_{\al_2\beta}$ would be non-zero. In other words all symbols $\beta\in\mathcal A$ satisfying \eqref{miedzy1} satisfy $\Omega_{\al_1\beta}=\Omega_{\al_2\beta}=0$. Observe that there exist two different symbols $\hat\al_1,\g_1\in\mathcal{A}$ satisfying  \eqref{miedzy1} such that $\Omega_{\hat\al_1\g_1}\neq 0$. Otherwise the permutation $\pi$ satisfies 
\eqref{deg1} and it is degenerate. Since $\al_1,\al_2$ do not satisfy \eqref{0wiersz} and the corresponding rows are identical, there exists 
a symbol $\g_2\in\mathcal{A}$ such that $\Omega_{\al_1\g_2}=\Omega_{\al_2\g_2}\neq 0$. It follows that $\Omega_{\hat\al_1\g_2}=\Omega_{\al_2\g_2}\neq 0$. We have
\[
\Omega_{\hat\al_1\al_2}=\Omega_{\al_2\hat\al_1}=\Omega_{\al_2\g_1}=0\ \text{ and }\ \Omega_{\hat\al_1\g_1}\neq 0\ \text{ and }\ \Omega_{\al_2\g_2}\neq 0.
\]
Hence, $\hat\al_1,\al_2,\g_1,\g_2$ are the desired symbols.
\end{proof}
\begin{wn}\label{nozero}
If $\pi$ is a nonsymmetric and nondegenerate permutation 
satisfying \eqref{pierost}, and $\tau\in\re^\mathcal{A}$ is a rationally independent vector, then there exist $\al_1,\al_2\in\mathcal{A}$
such that $\Omega_{\al_1\al_2}=\Omega_{\al_2\al_1}=0$ and 
for each $i=1,2$ the numbers
\[
(\Omega\tau)_{\al_2}-(\Omega\tau)_{\al_1},\ \text{and}\  (\Omega\tau)_{\al_i}-((\Omega\tau)_{\pi_0^{-1}(1)}+(\Omega\tau)_{\pi_0^{-1}(d)})
\]
are 
rationally independent. 
\end{wn}
\begin{proof}
We prove the case when $i=1$. If $i=2$, the proof goes along the same lines. Consider symbols $\al_1,\al_2,\g_1,\g_2$ given by 
Proposition~\ref{acperm}. We have $\Omega_{\al_1\al_2}=\Omega_{\al_2\al_1}=0$.
Assume that there exist 
integers $p$ and $q$ such that
	\[
	\begin{split}
	0&=p\big((\Omega\tau)_{\al_2}-(\Omega\tau)_{\al_1}\big)+q\big( (\Omega\tau)_{\al_1}-((\Omega\tau)_{\pi_0^{-1}(1)}+(\Omega\tau)_{\pi_0^{-1}(d)})\big)\\
	&=\sum_{\beta\in\mathcal A}(-q\Omega_{\pi_0^{-1}(1)\beta}-q\Omega_{\pi_0^{-1}(d)\beta}+(q-p)\Omega_{\al_1\beta}+p\Omega_{\al_2\beta})\tau_\beta.
	\end{split}
	\]
	By rational independence of $\tau$, this yields
	\[
	-q\Omega_{\pi_0^{-1}(1)\beta}-q\Omega_{\pi_0^{-1}(d)\beta}+(q-p)\Omega_{\al_1\beta}+q\Omega_{\al_2\beta}=0,
	\]
	for every $\beta\in\mathcal A$. Since
	\[
	\Omega_{\pi_0^{-1}(1)\beta}=1\text{ for }\beta\in\mathcal A\setminus\{\pi_0^{-1}(1)\}\text{ and }\Omega_{\pi_0^{-1}(d)\beta}=-1\text{ for }\beta\in\mathcal A\setminus\{\pi_0^{-1}(d)\},
	\]
	we have
	\[
	(q-p)\Omega_{\al_1\beta}+q\Omega_{\al_2\beta}=0\ \text{for}\ \beta\in\mathcal A\setminus\{\pi_0^{-1}(1),\pi_0^{-1}(d) \}
	\]
	Since 
	by~\eqref{tezperm} the matrix
	\[
	\begin{bmatrix}
	\Omega_{\al_1\g_1} & \Omega_{\al_1\g_2}\\
	\Omega_{\al_2\g_1} & \Omega_{\al_2\g_2}
	\end{bmatrix}
	\]
	is triangular and has non-zero entries on the main diagonal, it follows that
	\[
	p=q=0,
	\]
	which proves the rational independence.
\end{proof}

\section{The measures on the surface}\label{sec:meas}
In this section we will deal with measures on a given surface $(M,\Sigma)$ 
which are absolutely continuous 
with respect to the Lebesgue measure. We want to prove that, if the density of such measure is bounded and close enough to 
the constant function $1$ in $L^1$, then there is 
an explicit way to construct a homeomorphism which pushes this measure to 
the Lebesgue measure. The computation given below is partially inspired by the paper of Moser \cite{Mos}. We will need the following auxiliary lemma.
\begin{lm}\label{zlep}
Let $x,y\in\re^2$ be two points on 
the plane and let $\overline{xy}$ be 
the segment with endpoints at $x$ and $y$. Let $H_1,H_2$ be two affine transformations on $\re^2$. If $H_1(x)=H_2(x)$ and $H_1(y)=H_2(y)$, then $H_1|_{\overline{xy}}=H_2|_{\overline{xy}}$. Moreover, for each noncollinear triples $x_1,x_2,x_3\in\re^2$ and $y_1,y_2,y_3\in\re^2$ there exists a unique invertible affine transformation $H$ such that $H(x_i)=y_i$ for $i=1,2,3$.
\end{lm}
%
We will now prove lemmas which give a construction of a homeomorphism of an isosceles right triangle which pushes forward any given absolutely continuous measure whose density satisfies some conditions to the Lebesgue measure.
For any affine transformation $G\colon \re^2\to\re^2$ define $\text{lin}(G)=DG$ as the matrix determining its linear part, $D$ denotes the derivative. Moreover for any real $2\times 2$ matrix
$M=\bigl(\begin{smallmatrix}
a&b \\ c&d
\end{smallmatrix} \bigr)$
we have the following formula for the operator norm
\begin{equation}\label{norma}
\|M\|:=\sqrt{\frac{a^2+b^2+c^2+d^2+\sqrt{(a^2+b^2+c^2+d^2)^2-4(\det(M))^2}}{2}}.
\end{equation}
In particular, $\|\operatorname{lin}(G)\|$ is a Lipschitz constant of $G$.

On the space of homeomorphisms $\text{Hom}(X)$ 
of a compact metric space $X$, we will consider 
the standard metric $d_{\text{Hom}}(H,G):=\max\{\sup_{x\in X}d(H(x),G(x)),\sup_{x\in X}d(H^{-1}(x),G^{-1}(x))\}$, where $d$ is 
the metric on $X$.

Throughout this section we will heavily depend on the following construction. 
Let  $0<a<1$, and let $V$ be the triangle in $\re^2$ with vertices in points $(0,a),(a,0),(0,-a)$.
Let $V_1$ and $V_2$ be the triangles whose vertices are $(0,a),(a,0),(0,0)$ and $(0,0),(a,0),(0,-a)$ respectively. Let $0\le h<1$ and $0<\epsilon<1$. Let $y(h):=(\epsilon a(1-h),ha)$. 	Consider the triangles given by the following  sets of vertices:
\begin{itemize}
	\item $C_1=C_1(h,\epsilon)$ given by $\{(0,0),(0,a),y(h,\epsilon)\}$;
	\item $C_2=C_2(h,\epsilon)$ given by $\{(a,0),(0,a),y(h,\epsilon)\}$;
	\item $C_3=C_3(h,\epsilon)$ given by $\{(0,0),(\epsilon a,0),y(h,\epsilon)\}$;
	\item $C_4=C_4(h,\epsilon)$ given by $\{(a,0),(\epsilon a,0),y(h,\epsilon)\}$;
	\item $C_5=C_5(h,\epsilon)$ given by $\{(0,0),(0,-a),(\epsilon a,0)\}$;
	\item $C_6=C_6(h,\epsilon)$ given by $\{(a,0),(0,-a),(\epsilon a,0)\}$.
\end{itemize}
	
	Let $\hat h:=\frac{h}{h+1}\ge 0$. Consider the point $\hat y(h,\epsilon)=(\epsilon a-\hat h\epsilon a,-\hat ha)$. Consider also the triangles
	\begin{itemize}
		\item $\hat C_1=\hat C_1(h,\epsilon)$ given by $\{(0,0),(0,a),(\epsilon a,0)\}$;
		\item $\hat C_2=\hat C_2(h,\epsilon)$ given by $\{(a,0),(0,a),(\epsilon a,0)\}$;
		\item $\hat C_3=\hat C_3(h,\epsilon)$ given by $\{(0,0),(\epsilon a,0),\hat y(h,\epsilon)\}$;
		\item $\hat C_4=\hat C_4(h,\epsilon)$ given by $\{(a,0),(\epsilon a,0),\hat y(h,\epsilon)\}$;
		\item $\hat C_5=\hat C_5(h,\epsilon)$ given by $\{(0,0),(0,-a),\hat y(h,\epsilon)\}$;
		\item $\hat C_6=\hat C_6(h,\epsilon)$ given by $\{(a,0),(0,-a),\hat y(h,\epsilon)\}$.
	\end{itemize}
		By the definition of $h$ and $\hat h$ we have 
		\begin{equation}\label{proporcje5}
		\begin{split}
		&\frac{\operatorname{Leb}(C_1)}{\operatorname{Leb}(\hat C_1)}=\frac{\operatorname{Leb}(C_2)}{\operatorname{Leb}(\hat C_2)}=1-h\quad\text{and}\\	
		\frac{\operatorname{Leb}(C_3)}{\operatorname{Leb}(\hat C_3)}&=\frac{\operatorname{Leb}(C_4)}{\operatorname{Leb}(\hat C_4)}=\frac{\operatorname{Leb}(C_5)}{\operatorname{Leb}(\hat C_5)}=\frac{\operatorname{Leb}(C_6)}{\operatorname{Leb}(\hat C_6)}=1+h.
		\end{split}
		\end{equation}
	
	Define $H(h,\epsilon):V\to V$ as a  piecewise affine homeomorphism such that
\begin{align*}
		(\text{i})\quad &H(h,\epsilon)(C_i)=\hat C_i, H(h,\epsilon)|_{C_i}\text{ is affine for  }i=1,\ldots,6;\\
		(\text{ii})\quad &H(h,\epsilon)\text{ fixes }(0,0),(0,a),(0,-a),(a,0),\numberthis\label{Hdef} \\
		(\text{iii})\quad &H(h,\epsilon)(y)=(a\epsilon,0)\text{ and } H(h,\epsilon)(a\epsilon,0)=\hat y.
	\end{align*}
	
	\begin{figure}[h]
 \begin{center}
		\includegraphics[scale=0.45]{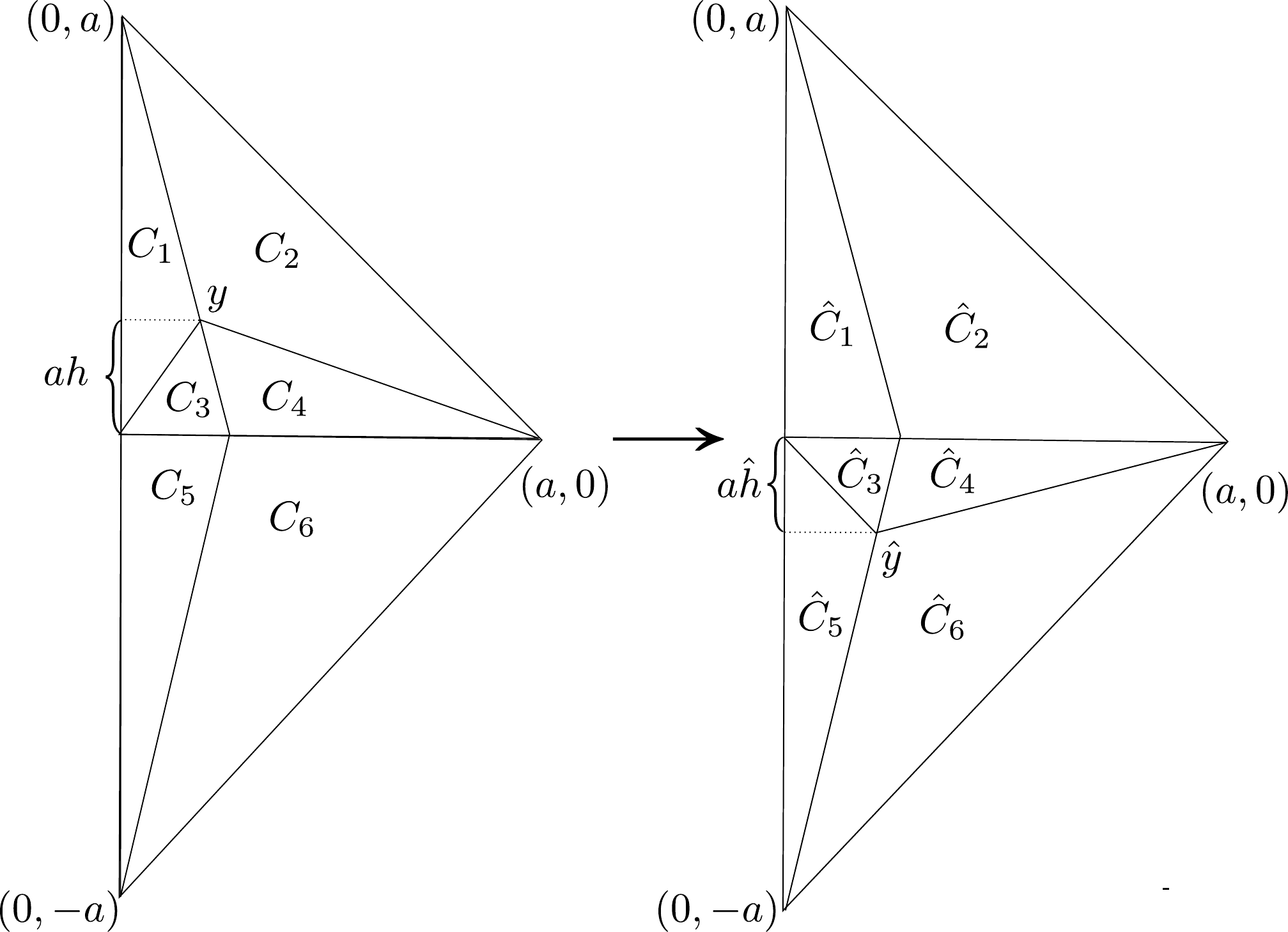}
		\caption{The division of $V$ into triangles and the map $H(h,\epsilon)$.}
 \end{center}
	\end{figure}
	

	Note that by Lemma \ref{zlep}, $H(h,\epsilon)$ is well defined everywhere on $V$ and also $H(h,\epsilon)|_{\partial V}=Id|_{\partial V}$.
	 Moreover
	 \[
	 \text{lin}(H(h,\epsilon)|_{C_1}):=
	 \begin{bmatrix}
	 1+\frac{h}{1-h} & 0\\
	 \frac{-h}{\epsilon (1-h)} & 1
	 \end{bmatrix};
	 \quad\quad
	 \text{lin}(H(h,\epsilon)|_{C_2}):=
	 \begin{bmatrix}
	 1-\frac{\epsilon h}{(1-h)(1-\epsilon )} & \frac{-\epsilon h}{(1-h)(1-\epsilon )}\\
	 \frac{h}{(1-h)(1-\epsilon )} & 1+\frac{h}{(1-h)(1-\epsilon )}
	 \end{bmatrix};
	 \]
	 \[
	 \text{lin}(H(h,\epsilon)|_{C_3}):=
	 \begin{bmatrix}
	 1-\frac{h}{1+h} & \frac{2\epsilon }{1+h}\\
	 \frac{-h}{\epsilon (1+h)} & 1-\frac{2h}{1+h}
	 \end{bmatrix};
	 \quad
	 \text{lin}(H(h,\epsilon)|_{C_4}):=
	 \begin{bmatrix}
	 1+\frac{\epsilon h}{(1+h)(1-\epsilon )} & \frac{\epsilon (2+h-2\epsilon)}{(1+h)(1-\epsilon )}\\
	 \frac{h}{(1+h)(1-\epsilon )} & 1-\frac{h(1-2\epsilon )}{(1+h)(1-\epsilon )}
	 \end{bmatrix};
	 \]
	 \[
	 \text{lin}(H(h,\epsilon)|_{C_5}):=
	 \begin{bmatrix}
	 1-\frac{h}{1+h} & 0\\
	 \frac{-h}{\epsilon (1+h)} & 1
	 \end{bmatrix};
	 \quad\quad
	 \text{lin}(H(h,\epsilon)|_{C_6}):=
	 \begin{bmatrix}
	 1+\frac{\epsilon h}{(1+h)(1-\epsilon )} & \frac{-h\epsilon }{(1+h)(1-\epsilon )}\\
	 \frac{h}{(1+h)(1-\epsilon )} & 1-\frac{h}{(1+h)(1-\epsilon )}
	 \end{bmatrix}.
	 \]
	 	By \eqref{proporcje5} we have 
	 	\[
	 	\det\big(\text{lin}(H(h,\epsilon)|_{C_1})\big)=\det\big(\text{lin}(H(h,\epsilon)|_{C_2})\big)=\frac{1}{1-h}\ge 1,
	 	\]
	 	and
	 	\[
	 	\begin{split}
	 	&\det\big(\text{lin}(H(h,\epsilon)|_{C_3})\big)=\det\big(\text{lin}(H(h,\epsilon)|_{C_4})\big)
	 	\\&=\det\big(\text{lin}(H(h,\epsilon)|_{C_5})\big)=\det\big(\text{lin}(H(h,\epsilon)|_{C_6})\big)=\frac{1}{1+h}\le 1.
	 	\end{split}
	 	\]
	 It is also worth noting that $(0,0)$ is fixed by the affine maps $H(h,\epsilon)|_{C_1}$, $H(h,\epsilon)|_{C_3}$ and $H(h,\epsilon)|_{C_5}$, while $(a,0)$ is a common fixed point for the transformations $H(h,\epsilon)|_{C_2}$, $H(h,\epsilon)|_{C_4}$ and $H(h,\epsilon)|_{C_6}$.
	

	
We can also define $H(h,\epsilon):V\to V$ for $-1<h\le 0$. Let $J:\re^2\to \re^2$ be the reflection across 
the $x$-axis. Note that $JV=V$, $JV_1=V_2$ and $JV_2=V_1$. Now define $\hat{h}:=\frac{h}{1+|h|}$,  $C_i(h,\epsilon):=J(C_i(-h,\epsilon))$, $\hat C_i(h,\epsilon):=J(\hat C_i(-h,\epsilon))$ and
\begin{equation}\label{inwol}
H(h,\epsilon):=J\circ H(-h,\epsilon)\circ J.
\end{equation}
For $i=1,\ldots,6$ we have
\[
\text{lin}(H(h,\epsilon)|_{C_i}=J\circ\text{lin}(H(-h,\epsilon)|_{C_i}\circ J.
\]
Since $J$ is an isometry, we also obtain
\[
	 	\det\big(\text{lin}(H(h,\epsilon)|_{C_1})\big)=\det\big(\text{lin}(H(h,\epsilon)|_{C_2})\big)=\frac{1}{1+h}\ge 1,
\]
and
\[
\begin{split}
	 	&\det\big(\text{lin}(H(h,\epsilon)|_{C_3})\big)=\det\big(\text{lin}(H(h,\epsilon)|_{C_4})\big)
	 	\\&=\det\big(\text{lin}(H(h,\epsilon)|_{C_5})\big)=\det\big(\text{lin}(H(h,\epsilon)|_{C_6})\big)=\frac{1}{1-h}\le 1.
\end{split}
\]
Hence in general for $-1<h<1$ we have
	 	\begin{equation}\label{proporcje2}
	 	\det\big(\text{lin}(H(h,\epsilon)|_{C_1})\big)=\det\big(\text{lin}(H(h,\epsilon)|_{C_2})\big)=\frac{1}{1-|h|}\ge 1,
	 	\end{equation}
	 	and
	 	\begin{equation}\label{proporcje3}
	 	\begin{split}
	 	&\det\big(\text{lin}(H(h,\epsilon)|_{C_3})\big)=\det\big(\text{lin}(H(h,\epsilon)|_{C_4})\big)
	 	\\&=\det\big(\text{lin}(H(h,\epsilon)|_{C_5})\big)=\det\big(\text{lin}(H(h,\epsilon)|_{C_6})\big)=\frac{1}{1+|h|}\le 1.
	 	\end{split}
	 	\end{equation}

\begin{lm}\label{Hodh}
	For 
	any fixed $\epsilon>0$ and for 
	every $h_1,h_2\in(-\frac{1}{2},\frac{1}{2})$ we have
	\begin{equation}\label{conteq}
	d_{\operatorname{Hom}}(H(h_1,\epsilon),H(h_2,\epsilon))\leq \frac{20a}{\epsilon}|h_2-h_1|.
	\end{equation}
	\end{lm}
	\begin{proof}
	We first prove that
\begin{equation}\label{conteq2}
\|H(h_1,\epsilon)(x)-H(h_2,\epsilon)(x)\|\leq \frac{20a}{\epsilon}|h_2-h_1|,
\end{equation}
for every $x\in V$.
Indeed, assume that $h_1$ and $h_2$ are non-negative numbers and $h_1\geq h_2$.
Consider the triangle $W_1$ with vertices $(0,0),y(h_1,\epsilon),y(h_2,\epsilon)$ and 
the triangle $W_2$ given by 
the points 
$(a,0),y(h_1,\epsilon), y(h_2,\epsilon)$.
Assume that $x\in V\setminus(W_1\cup W_2)$. Then $x\in C_i(h_1,\epsilon)\Leftrightarrow x\in C_i(h_2,\epsilon)$. This implies that both $H(h_1,\epsilon)$ and $H(h_2,\epsilon)$ act on $x$ by  affine transformations whose linear parts are of the same form.
Since for $i=1,\ldots, 6$ the affine maps $H(h_1,\epsilon)|_{C_i(h_1,\epsilon)}$ and $H(h_2,\epsilon)|_{C_i(h_2,\epsilon)}$ share a common fixed point,

we get that
\[H(h_1,\epsilon)(x)-H(h_2,\epsilon)(x)=\text{lin}(H(h_1,\epsilon)|_{C_i(h_1,\epsilon)})x-\text{lin}(H(h_2,\epsilon)|_{C_i(h_2,\epsilon)})x.\]
By using the formula \eqref{norma} for each $i=1,\ldots,6$ we get
\[
\|\text{lin}(H(h_1,\epsilon)|_{C_i(h_1,\epsilon)})-\text{lin}(H(h_1,\epsilon)|_{C_i(h_2,\epsilon)})\|\le \frac{10}{\epsilon}(h_1-h_2).
\]
Since the above norm is the operator norm for $H(h_1,\epsilon)-H(h_2,\epsilon)$ (which is a linear transformation), this implies that
\[
\sup_{x\in V\setminus(W_1\cup W_2)}\|H(h_1,\epsilon)(x)-H(h_2,\epsilon)(x)\|\leq \frac{10}{\epsilon}(h_1-h_2)\|x\|<\frac{20a}{\epsilon}(h_1-h_2).
\]
Next note that $W_1=C_3(h_1,\epsilon)\cap C_1(h_2,\epsilon)$ and $W_2=C_4(h_1,\epsilon)\cap C_2(h_2,\epsilon)$.
We now prove that for $x\in W_1\cup W_2$ we also have $\|H(h_1,\epsilon)(x)-H(h_1,\epsilon)(x)\|\leq \frac{20a}{\epsilon}(h_1-h_2)$.
Suppose that $x\in W_1$; the proof for $x\in W_2$ is analogous. Consider 
the segment $I_x\subset W_1$ with endpoints on 
the segments $\overline{(0,0),y(h_1,\epsilon)}$ and $\overline{(0,0), y(h_2,\epsilon)}$ such that $x\in I_x$ and 
which is parallel to $\overline{y(h_1,\epsilon), y(h_2,\epsilon)}$. Then
\begin{equation}
|I_x|\leq \|y(h_1,\epsilon)-y(h_2,\epsilon)\|=a\sqrt{1+\epsilon} (h_1-h_2)<2a(h_1-h_2).
\end{equation}
Note that $I_x$ divides 
the intervals $\overline{(0,0),y(h_1,\epsilon)}$ and $\overline{(0,0), y(h_2,\epsilon)}$ with the same ratio. Since affine transformations do not change the ratio of the lengths of 
collinear segments and
\[
H(h_1,\epsilon)(\overline{(0,0),y(h_1,\epsilon)})= H(h_2,\epsilon)(\overline{(0,0), y(h_2,\epsilon)})= \overline{(0,0),(0,
\epsilon a)},
\]
it follows  that 
the segments $H(h_1,\epsilon)(I_x)$ and $H(h_2,\epsilon)(I_x)$ share a common endpoint in $\overline{(0,0),(0,\epsilon a)}$.
Using again the conservation of the ratio of the lengths of collinear segments by affine transformations, we get 
\[
\frac{|H(h_1,\epsilon)(I_x)|}{|I_x|}=\frac{|H(h_1,\epsilon)(\overline{y(h_1,\epsilon),y(h_2,\epsilon)})|}{|\overline{y(h_1,\epsilon),y(h_2,\epsilon)}|}=\frac{|H(h_1,\epsilon)(\overline{y(h_1,\epsilon),(\epsilon a,0)})|}{|\overline{y(h_1,\epsilon),(\epsilon a,0)}|}=\frac{1}{1+h_1}\le 1,
\]
and
\[
\frac{|H(h_2,\epsilon)(I_x)|}{|I_x|}=\frac{|H(h_2,\epsilon)(\overline{y(h_1,\epsilon),y(h_2,\epsilon)})|}{|\overline{y(h_1,\epsilon),y(h_2,\epsilon)}|}=\frac{|H(h_2,\epsilon)(\overline{(0,a),y(h_2,\epsilon)})|}{|
\overline{(0,a),y(h_2,\epsilon)}|}=\frac{1}{1-h_2}\le 2.
\]
As $H(h_1,\epsilon)(x)\in H(h_1,\epsilon)(I_x)$ and $H(h_2,\epsilon)(x)\in H(h_2,\epsilon)(I_x)$, we obtain
\[
\begin{split}
\|H(h_1,\epsilon)(x)-H(h_2,\epsilon)(x)\|&\leq |H(h_1,\epsilon)(I_x)|+|H(h_2,\epsilon)(I_x)|\leq3|I_x|\\&< 6a(h_2-h_1)<\frac{20}{\epsilon}a(h_1-h_2).
\end{split}
\]
By proceeding analogously for $h_2\ge h_1$ we prove that
\[
\|H(h_1,\epsilon)(x)-H(h_2,\epsilon)(x)\|\leq \frac{20}{\epsilon}a(h_2-h_1).
\]
The case when $h_1$ and $h_2$ are non-positive is analogous. To prove 
the similar inequality when $h_1$ and $h_2$ are of different sign, let $h_0:=0$. Then $H(h_0,\epsilon)=Id$.
Using the previous case we show that
\[
\|H(h_2,\epsilon)(x)-x\|\leq \frac{20}{\epsilon}a|h_0-h_2|,
\]
and
\[
\|H(h_1,\epsilon)(x)-x\|\leq \frac{20}{\epsilon}a|h_1-h_0|.
\]
Since $h_1$, $h_2$ have different sign, the numbers $h_0-h_2$, $h_1-h_0$ are of the same sign. It follows that
\[
\|H(h_1,\epsilon)(x)-H(h_2,\epsilon)(x)\|\leq \|H(h_1,\epsilon)(x)-x\|+\|H(h_2,\epsilon)(x)-x\|\leq \frac{20}{\epsilon}a|h_2-h_1|.
\]

By proceeding 
as in the proof of \eqref{conteq2} and replacing $H(h_i,\epsilon)$ by $H^{-1}(h_i,\epsilon)$ for $i=1,2$, 
we can prove that for every $x\in V$ we also have
\begin{equation}\label{conteq1}
\|H^{-1}(h_2,\epsilon)(x)-H^{-1}(h_1,\epsilon)(x)\|\leq \frac{20}{\epsilon}a|h_2-h_1|,
\end{equation}
which concludes the proof of the lemma.
\end{proof}

\begin{lm}\label{trojkat0}
Let $V$, $V_1$ and $V_2$ be the triangles defined 
above.  Let  $0<\hat\ep<10^{-8}$
, and let $\kappa>0$. Suppose that $f\in L^1(V)$ satisfies
\begin{equation}\label{assum}
f>\kappa
;\quad
\frac{1}{1+\hat\ep}<f\text{ or }f<\frac{1}{1-\hat\ep};\quad
\int_Vf(x)dx=\operatorname{Leb}(V).
\end{equation}
Let $\mu_f:=f dx$. Then there exists  a piecewise affine homeomorphism $H_f:V\to V$ such that
\begin{itemize}
	\item[(i)] $(H_f)_*\mu_f(V_i)=\operatorname{Leb}(V_i)$ for $i=1,2$;
	\item[(ii)] $H_f|_{\partial V}=Id|_{\partial V}$;
	\item[(iii)] there 
	exists $-\hat\ep<h_f<\hat\ep$ such that  $\det(DH_f^{-1})$ is constant on each  $V_i$ and is equal to $1\pm h_f$;
	\item[(iv)] the Lipschitz constants of $H_f$ and $H^{-1}_f$ are less than $\frac{5}{4}$;
	\item[(v)] the maps $f\mapsto H_f\in Hom(V)$ and $f\mapsto\det(DH_f^{-1})\in L^{\infty}(V)$  are continuous 
	on the set of $f\in  L^1(V)$ satisfying~\eqref{assum} for a given $\kappa$.
\end{itemize}
\end{lm}

\begin{proof}
Since $\mu_f$ is an absolutely continuous measure with respect to $\operatorname{Leb}($, there are no segments of positive measure $\mu_f$ in $V$. Hence there exists a unique $-1<h_f< 1$ and a point $y=y_f=(\sqrt{\hat\ep}a(1-|h_f|),h_fa)$ such that the 
quadrilateral with vertices $\{(0,a),(0,0),(a,0),y\}$ and the 
quadrilateral with vertices $\{(0,-a),(0,0),(a,0),y\}$ have the same measure $\mu_f$ equal to $\frac{1}{2}\operatorname{Leb}(V)$.
	
	Consider the triangles $C_i=C_i^f:=C_i(h_f,\sqrt{\hat{\ep}})$ for $i=1,\ldots,6$.
	By the definition of $h_f$ we have
	\[
	\mu_f(C_1\cup C_2)=\mu_f(C_3\cup C_4\cup C_5\cup C_6)=\frac{1}{2}\operatorname{Leb}(V).
	\]
	We now evaluate the bounds on $h_f$. Assume that $f>\frac{1}{1+\hat\ep}$. Since $\operatorname{Leb}(V)=a^2$ we have
	\[
	\begin{split}
	\frac{1}{2}a^2&=\mu_f(C_3\cup C_4\cup C_5\cup C_6)=\int_{C_3\cup C_4\cup C_5\cup C_6}f(x)dx\\
	&>\frac{1}{1+\hat\ep}\operatorname{Leb}(C_3\cup C_4\cup C_5\cup C_6)=\frac{1}{1+\hat\ep}\big(\frac{1}{2}(a+|h_f|a)a\big).
	\end{split}
	\]
	Hence
	\begin{equation}
	f>\frac{1}{1+\hat\ep}\Rightarrow |h_f|<\hat\ep.
	\end{equation}
	Now assume that $f<\frac{1}{1-\hat\ep}$. Then we have
	\[
	\begin{split}
	\frac{1}{2}a^2&=\mu_f(C_1\cup C_2)=\int_{C_1\cup C_2}f(x)dx\\
	&<\frac{1}{(1-\hat\ep)}\operatorname{Leb}(C_1\cup C_2)=\frac{1}{1-\hat\ep}\big(\frac{1}{2}(a-|h_f|a)a\big).
	\end{split}
	\]
	This 
	shows
	\begin{equation}
	f<\frac{1}{1-\hat\ep}\Rightarrow |h_f|<\hat\ep.
	\end{equation}
	

	\textbf{Definition of $H_f$.}
	Define $H_f:=H(h_f,\sqrt{\hat{\ep}})$, a piecewise 
	affine homeomorphism on $V$. Note that by definition we have $H_f|_{\partial V}=Id|_{\partial V}$.
	 Moreover
	 \[
	 (H_f)_*\mu_f(V_1)=\mu_f(C_1\cup C_2)=\frac{1}{2}\operatorname{Leb}(V)=\operatorname{Leb}(V_1)
	 \]
	 and
	 \[
	 (H_f)_*\mu_f(V_2)=\mu_f(C_3\cup C_4\cup C_5\cup C_6)=\frac{1}{2}\operatorname{Leb}(V)=\operatorname{Leb}(V_2).
	 \]
	 Hence $H_f$ satisfies points (i) and (ii).

	Furthermore, by \eqref{proporcje2} and \eqref{proporcje3} we have that
	\begin{equation}\label{proporcje4}
	\det\big(\text{lin}(H_f|_{C_1})\big)=\det\big(\text{lin}(H_f|_{C_2})\big)=\frac{1}{1-|h_f|}\ge 1,
	\end{equation}
	and
	\begin{equation}\label{proporcje6}
	\begin{split}
	&\det\big(\text{lin}(H_f|_{C_3})\big)=\det\big(\text{lin}(H_f|_{C_4})\big)
	\\&\det\big(\text{lin}(H_f|_{C_5})\big)=\det\big(\text{lin}(H_f|_{C_6})\big)=\frac{1}{1+|h_f|}\le 1.
	\end{split}
	\end{equation}
	Note that $V_1=\hat C_1\cup \hat C_2$ and $V_2=\hat C_3\cup\hat C_4\cup\hat C_5\cup\hat C_6$ for $h_f\ge 0$ and $V_1=\hat C_3\cup\hat C_4\cup\hat C_5\cup\hat C_6$ and $V_2=\hat C_1\cup \hat C_2$ for $h_f\le 0$. Hence by \eqref{proporcje4} and \eqref{proporcje6} we have
	\begin{equation}\label{proporcje7}
	\det(\text{lin}(H_f^{-1}|_{V_1}))=1-h_f\quad\text{ and }\quad\det(\text{lin}(H_f^{-1}|_{V_2}))=1+h_f
	\end{equation}
	
	Thus $H_f$ satisfies (iii).
	
	\textbf{The norm of the linear part.}
	We will now prove that 
	$\|\text{lin}(H_f)|_{C_i})\|<\frac{5}{4}$ for $i=1,\ldots,6$. Note that each of the matrices $\text{lin}(H_f)|_{C_i}$ is of the form
	$M=\bigl(\begin{smallmatrix}
	1+b&c \\ d&1+e
	\end{smallmatrix} \bigr)$,
	where $|b|,|c|,|d|,|e|<3\sqrt{\hat\ep}$. Hence, using the formula \eqref{norma} and the fact that $\ep<10^{-8}$, we obtain that
	\begin{equation}
	\begin{split}
	      \|\text{lin}&(H_f)|_{C_i}\|<\\
	&<\sqrt{\frac{2(1+3\sqrt{\hat\ep})+2\cdot 3\sqrt{\hat\ep}+\sqrt{(2(1+3\sqrt{\hat\ep})+2\cdot 3\sqrt{\hat\ep})^2-4(\det(\text{lin}(H_f)|_{\hat C_i}))^2}}{2}}\\
	&<\sqrt{\frac{2+12\sqrt{\hat\ep}+36\hat\ep+\sqrt{(2+12\sqrt{\hat\ep}+36\hat\ep)^2-4(\frac{1}{1+\hat{\ep}})^2}}{2}}\\
	&<\sqrt{1+5\sqrt[4]{\hat\ep}}<\frac{5}{4}.
	\end{split}
	\end{equation}
	In the same way we prove that $\|\text{lin}(H_f)^{-1}|_{\hat C_i})\|<\frac{5}{4}$. Thus $H_f$ satisfies (iv).


\textbf{Continuity of $f\mapsto H_f$.}
 Suppose that $f,g\in L^1(V)$ satisfy \eqref{assum}. By Lemma \ref{Hodh}, we already know that
  \begin{equation}\label{conteq3}
  d_{\operatorname{Hom}}(H_f,H_g)\le\frac{20}{\sqrt{\hat{\ep}}}a|h_f-h_g|.
  \end{equation}
  We prove that
 \begin{equation}\label{hcont}
 |h_f-h_g|\le C\|f-g\|_{L^{1}},
 \end{equation}
for some constant $C>0$ depending only on $a$ and $\kappa$. Let $\de:=\|f-g\|_{L^1}$.

\textbf{Case $h_f$ and $h_g$ have the same sign.}
Assume that $h_f\geq h_g\geq 0$ or $0\geq h_g\geq h_f$. Then
	\[
	\begin{split}
	0&=\mu_f(C^f_1\cup C^f_2)-\mu_g(C_1^g\cup C_2^g)\\ &=\int_{C^f_1\cup C^f_2}f(x)dx-\int_{C^g_1\cup C^g_2}g(x)dx\\
	&=\int_{C^g_1\cup C^g_2}(f-g)(x)dx-
	      \int_{(C^g_1\cup C^g_2)\setminus(C^f_1\cup C^f_2)}g(x)dx\\
	&\le\de-|h_f-h_g|\frac{a\kappa}{2},
	\end{split}
	\]
	and hence
	\[
	|h_f-h_g|\le\frac{2\de}{a\kappa}.
	\]

Thus \eqref{hcont} holds with 
$C=\frac{2}{a\kappa}$.

\textbf{Case of $h_f$, $h_g$ with different sign.}
Suppose that $h_f\ge 0\ge h_g$. Then we have
	\[
	\begin{split}
	0&=\mu_f(C^f_1\cup C^f_2)-\mu_g(C_3^g\cup C_4^g\cup C_5^g\cup C_6^g)\\
	&=\int_{C^f_1\cup C^f_2}f(x)dx-\int_{C_3^g\cup C_4^g\cup C_5^g\cup C_6^g}g(x)dx\\
	&=\int_{C^f_1\cup C^f_2}(f-g)(x)dx-\int_{(C_3^g\cup C_4^g\cup C_5^g\cup C_6^g)\setminus(C^f_1\cup C^f_2)}g(x)dx\\
	&\le\de-|h_f-h_g|\frac{a\kappa}{2}.
	\end{split}
	\]
Thus we have
	\[
	0\le h_f-h_g\le\frac{2\de}{a\kappa},
	\]
which completes the proof of \eqref{hcont}.

By combining \eqref{conteq3} and \eqref{hcont} we obtain
\[
d_{Hom}(H_f,H_g)\leq \frac{10}{\sqrt{\hat{\ep}}}|h_f-h_g|<\frac{10}{\sqrt{\hat{\ep}}}C\|f-g\|_{L^{1}}.
\]
This concludes the proof of the continuity of $f\mapsto H_f$. By the formula given in \eqref{proporcje7}, the continuity of the map $f\mapsto h_f$ also implies the continuity of the map $f\mapsto\det(DH_f^{-1})\in L^\infty(V)$. Thus (v) is proved.

	\end{proof}
	
	Let $(X,\mu)$ be a standard metric probability space.
	For $0<s_1<s_2$, define $\mathcal W(X,s_1,s_2)\subset L^1(X,\mu)$ 
	by
	\begin{equation}\label{ineq+-}
	\mathcal W(X,s_1,s_2):=\{f\in L^1(X); s_1< f< s_2; \int_Xf\,d\mu(x)=\mu(X)\}.
	\end{equation}
	Let $V$ be 
	the triangle with vertices $(0,-a),(0,a),(a,0)$, 
	equipped with the (normalized) 2-dimensional Lebesgue measure.
	We need the following lemma.
	\begin{lm}\label{ciagsklad}
		Let $H:\mathcal W(V,s_1,s_2)\to \operatorname{Hom}(V)$ be a continuous map.
		Assume that there exists $\ell>0$ such that, for every $f\in \mathcal W(V,s_1,s_2)$, the homeomorphism $H(f)^{-1}$ is Lipschitz with constant $\ell$.
Then the transformation
		\[
		W(s_1,s_2)\ni f\mapsto f\circ H(f)\in L^1(V)
		\]
		is continuous.
			\end{lm}
			\begin{proof}
				Take $f\in \mathcal W(V,s_1,s_2)$ and $\epsilon>0$. Let $f_\epsilon:V\to\re$ be a uniformly continuous function such that $\|f_\epsilon-f\|_{L^1}<\epsilon$. Let $0<\de<\epsilon$ be such that
				\begin{equation}\label{jednciag}
				\|x-y\|<\delta\Rightarrow |f_\epsilon(x)-f_\epsilon(y)|<\epsilon.
				\end{equation}
				Consider $0<\de'<\epsilon$ such that for every $g\in W(V,s_1,s_2)$ we have
				\begin{equation}\label{Homblisko}
				\|f-g\|_{L^1}<\de'\Rightarrow d_{\operatorname{Hom}}(H(f),H(g))<\de,
				\end{equation}
				and let $g\in W(V,s_1,s_2)$ be such that $\|f-g\|_{L^1}<\de'$.
				Since $H(g)^{-1}$ is Lipschitz with constant $\ell$, $H(g)_*Leb$ is an absolutely continuous measure with density bounded by $\ell^2$. Hence for every $h\in L^1(V)$ we have
				\begin{equation}\label{lipnorma}
				\|h\circ H(g)\|_{L^1}=\int_V|h\circ H(g)(x)|dx\le\int_V\ell^2|h(x)|dx=\ell^2\|h\|_{L^1}.
				\end{equation}
				Then
				\[
				\|f\circ H(f)-g\circ H(g)\|_{L^1}\le\|f\circ H(f)-f\circ H(g)\|_{L^1}+\|f\circ H(g)-g\circ H(g)\|_{L^1}
				\]
				and, by \eqref{lipnorma},
				\[
				\|f\circ H(g)-g\circ H(g)\|_{L^1}\le
				\ell^2\|f-g\|_{L^1}.
				\]
				Moreover
				\[
				\begin{split}
				\|f\circ H(f)-f\circ H(g)\|_{L^1}\le&\|f\circ H(f)-f_\epsilon\circ H(f)\|_{L^1}+\|f_\epsilon\circ H(f)-f_\epsilon\circ H(g)\|_{L^1}\\&+\|f_\epsilon\circ H(g)-f\circ H(g)\|_{L^1}\\ \le&2\ell^2\|f-f_\epsilon\|_{L^1}+\|f_\epsilon\circ H(f)-f_\epsilon\circ H(g)\|_{L^1},
				\end{split}
				\]
				where the last inequality 
				comes from~\eqref{lipnorma}. By \eqref{Homblisko} and \eqref{jednciag}, we have
				\[
				\|f_\epsilon\circ H(f)-f_\epsilon\circ H(g)\|_{L^1}<\epsilon.
				\]
				To sum up we obtain
				\[
				\|f\circ H(f)-g\circ H(g)\|_{L^1}\le \ell^2\|f-g\|_{L^1}+2\ell^2\|f-f_\epsilon\|_{L^1}+\epsilon\le\big(3\ell^2+1\big)\epsilon,
				\]
				which proves the assertion.
			\end{proof}
			\begin{uw}\label{ciagsklad1}
				The statement of Lemma \ref{ciagsklad} remains valid if we replace $V$ with any $2$-dimensional Riemannian surface $M$.
			\end{uw}


\begin{lm}\label{trojkat}
Let $0<\hat\ep<10^{-8}$. 
Let $f\in\mathcal W(V,\frac{1}{1+\hat\ep},\frac{1}{1-\hat\ep})$
and 
$\mu_f:=f dx$. Then there exists a homeomorphism $H_f:V\to V$, depending continuously on $f$, such that $(H_f)_*\mu_f=Leb$ and $H_f|_{\partial V}=Id|_{\partial V}$.
\end{lm}

\begin{proof}
 We assume that $a=1$. The prove for $a\neq 1$ goes along the same lines. Let $f\in \mathcal W(V,\frac{1}{1+\hat\ep},\frac{1}{1-\hat\ep})$. Denote by $V_1^1$ and $V_2^1$ the two halves of $V$ which are both isosceles right triangles with $\operatorname{diam}(V_1^1)=\text{diam}(V_2^1)=\sqrt{2}$.

Inductively, for $n\in\n$ define the family $\{V_i^n\}_{i=1}^{2^n}$ of congruent right isosceles triangles which divide $V$, $V_i^n=V_{2i-1}^{n+1}\cup V_{2i}^{n+1}$ for $i=1,\ldots,2^n$ and they satisfy
\begin{equation}\label{diampodst}
\text{diam}(V_n^i)=\frac{1}{\sqrt{2}^{n-2}}.
\end{equation}

We will construct $H_f$ inductively as a limit of piecewise affine transformations.

In the first step, using Lemma~\ref{trojkat0}, we obtain a piecewise affine homeomorphism $H^1_f:V\to V$ such that
\[
(H^1_f)_*\mu_f(V_1^1)=(H^1_f)_*\mu_f(V_2^1)=\frac{1}{2}\operatorname{Leb}(V)\text{ and }H^1_f|_{\partial V}=Id|_{\partial V}.
\]
Moreover $\det(D(H_f^1)^{-1})$ is constant on each  $V_1^1$ and $V_2^1$.

Suppose now that for some $n\in\n$ we have constructed piecewise affine homeomorphisms $H^j_f:V\to V$ for $j=1,\ldots,n$ such that for all $i=1,\ldots,2^n$ we have
\begin{equation}\label{zalind}
(H^n_f\circ\ldots\circ H^1_f)_*\mu_f(V_i^n)=\frac{1}{2^n}\operatorname{Leb}(V)=\operatorname{Leb}(V_i^n)\text{ and }H^j_f|_{\partial V}=Id|_{\partial V}.
\end{equation}
Moreover, suppose that $\det(D(H^n_f\circ\ldots\circ H^1_f)^{-1})$ is constant on each $V_i^n$ and equals $d_i^n>0$.

With these assumptions the measure $(H^n_f\circ\ldots\circ H^1_f)_*\mu_f$ is  absolutely continuous  and its density $f_n:V\to\re_{>0}$ satisfies
\[f_n(x)=d_i^n \cdot f( (H^n_f\circ\ldots\circ H^1_f)^{-1}x)\text{ if }x\in V_i^n,\]
and by \eqref{zalind}
\[
\int_{V_i^n}f_n(x)dx=(H^n_f\circ\ldots\circ H^1_f)_*\mu_f(V_i^n)=\operatorname{Leb}(V_i^n).
\]

Take any $1\leq i\leq 2^n$. In view of \eqref{ineq+-}, if  $d_i^n<1$ then
\[f_n(x)<\frac{d_i^n}{1-\hat\ep}<\frac{1}{1-\hat\ep}\text{ for all }x\in V_i^n\]
and if $d_i^n\ge 1$ then \[f_n(x)>\frac{d_i^n}{1+\hat\ep}\geq\frac{1}{1+\hat\ep}\text{ for all }x\in V_i^n.\]
It follows that $f_n:V_i^n\to\re_{>0}$ is a positive density satisfying \eqref{assum} with $\kappa=\frac{d_i^n}{1+\hat\ep}$.
Hence we can apply Lemma~\ref{trojkat0} to the triangle $V_i^n$ and the density function  $f_n:V_i^n\to\re_{>0}$, thus obtaining a piecewise affine homeomorphism $H_f^{n+1,i}:V_i^n\to V_i^n$ such that
\begin{equation}\label{eq:trmiar}
(H_f^{n+1,i})_*(\mu_{f_n}|_{V_i^n})(V_{2i-1}^{n+1})=(H_f^{n+1,i})_*(\mu_{f_n}|_{V_i^n})(V_{2i}^{n+1})=\frac{1}{2}\operatorname{Leb}(V_i^n)=\frac{1}{2^{n+1}}\operatorname{Leb}(V),
\end{equation}
\begin{equation}\label{Hbrzeg}
H^{n+1,i}_f|_{\partial V_i^n}=Id|_{\partial V_i^n},
\end{equation}
and
\begin{equation}\label{eq:trdet}
\text{$\det D((H_f^{n+1,i})^{-1})$ is constant on both $V_{2i-1}^{n+1}$ and $V_{2i}^{n+1}$.}
\end{equation}
Finally we define a piecewise affine homeomorphism $H^{n+1}_f:V\to V$ given by
\[
H^{n+1}_f(x):=H^{n+1,i}_f(x)\text{ whenever }x\in V_i^n.
\]
Then $H^{n+1}_f(V^{n}_i)=V^n_i$ and, by \eqref{Hbrzeg}, we have  $H^{n+1}_f|_{\partial V}=Id|_{\partial V}$.
Moreover, by \eqref{eq:trmiar},
\[
(H^{n+1}_f\circ\ldots\circ H^1_f)_*\mu_f(V_{2i-1}^{n+1})=(H^{n+1}_f\circ\ldots\circ H^1_f)_*\mu_f(V_{2i}^{n+1})=\frac{1}{2^{n+1}}\operatorname{Leb}(V).
\]
In view of \eqref{eq:trdet},  $\det(D(H^{n+1}_f)^{-1})$ is constant on each $V^{n+1}_j$ for $j=1,\ldots,2^{n+1}$ and  then so 
is $\det(D(H^{n+1}_f\circ\ldots\circ H^1_f)^{-1})$. Thus, we have proved that $H_f^{n+1}$ satisfies the assumptions of the induction.

Note 
that, by (iii) in Lemma \ref{trojkat0}, we have 
\begin{equation}\label{jedenwyzn}
1-\hat{\ep}<\det(D(H^n_f)^{-1})<1+\hat{\ep}\text{ for every }n\in\n,
\end{equation}\
and since $(H^j_f)^{-1}$ are piecewise linear homeomorphisms, it follows that
\begin{equation}\label{iloczwyzn}
(1-\hat{\ep})^{n}\le\det(D(H^{n}_f\circ\ldots\circ H^1_f)^{-1})\le(1+\hat{\ep})^{n}\text{ almost everywhere}
\end{equation}
and the above inequalities do not depend on $f$.

We now show that
\begin{equation}\label{defH}
H_f(x):=\lim_{n\to\infty}H_f^n\circ\ldots\circ H_f^1(x)
\end{equation}
yields a homeomorphism $H_f:V\to V$.
First note that
\begin{equation}\label{niezmien}
H_f^m(V_i^n)=V_i^n\text{ for }i=1,\ldots,2^n \text{ and }m>n.
\end{equation}
Moreover, by \eqref{diampodst} we have 
\begin{equation}\label{gransred}
\max_{i=1,\ldots,2^n}\operatorname{diam}(V_i^n)\to 0\text{ for }n\to\infty.
\end{equation}
This implies that $\{H_f^m\circ\ldots\circ H_f^1\}_{n\in\n}$ is a Cauchy sequence. Indeed, for any $\epsilon>0$ by $\eqref{gransred}$ we can choose $N\in\n$ such that $ \max_{i=1,\ldots,2^N}\operatorname{diam}(V_i^N)<\epsilon$. Moreover, by \eqref{niezmien},  for all $m,n\ge N$ we have
\[
H_f^n\circ\ldots\circ H_f^1(x)\in V_i^N\Longleftrightarrow H_f^m\circ\ldots\circ H_f^1(x)\in V_i^N.
\]
Hence $\|H_f^n\circ\ldots\circ H_f^1(x)-H_f^m\circ\ldots\circ H_f^1(x)\|<\epsilon$ for all $x\in V$. It follows that the map $H_f:V\to V$ given by \eqref{defH} is well defined and the convergence in \eqref{defH} is uniform. This implies that $H_f$ is continuous. Since $H_f^n|_{\partial V}=Id|_{\partial V}$ for all $n\in\n$, we 
also have $H_f|_{\partial V}=Id|_{\partial V}$.

Set $W_i^n:=(H_f^n\circ\ldots\circ H_f^1)^{-1}(V^n_i)$. In view of \eqref{niezmien},
\begin{equation}\label{eq:relVW}
W_i^n=(H_f^m\circ\ldots\circ H_f^1)^{-1}(V^n_i)\text{ for }m>n.
\end{equation}
Therefore,
\begin{equation}\label{eq:rownW}
(H_f^n\circ\ldots\circ H_f^1)^{-1}(x)\in W_i^N\Longleftrightarrow (H_f^m\circ\ldots\circ H_f^1)^{-1}(x)\in W_i^N\text{ if }m,n\geq N.
\end{equation}
By (iv) in Lemma~\ref{trojkat0}, ${(H_f^n)}^{-1}$ is a Lipschitz automorphism with a Lipschitz constant $\frac{5}{4}$. Thus, by \eqref{diampodst}, we have
\begin{equation*}
\text{diam}(W_i^n)<\text{diam}(V_i^n)\Big(\frac{5}{4}\Big)^n=2\Big(\frac{5}{4\sqrt{2}}\Big)^{n},
\end{equation*}
so
\begin{equation}\label{diamobr}
\max_{i=1,\ldots,2^n}\operatorname{diam}(W_i^n)\to 0\text{ for }n\to\infty.
\end{equation}
Using \eqref{diamobr} and \eqref{eq:rownW} and  repeating the same arguments as for $H_f$ by replacing $V_i^n$ with $W_i^n$, we obtain that the  map
$G_f:V\to V$ given by
\[
G_f(x):=\lim_{n\to\infty}(H_f^n\circ\ldots\circ H_f^1)^{-1}(x)
\]
is well defined and continuous. We  now show that $H_f\circ G_f=Id$ and $G_f\circ H_f=Id$. First note that in view of \eqref{eq:relVW} and the compactness of $V^n_i$ and $W^n_i$
we have $H_f(W_i^n)=V_i^n$ and $G_f(V_i^n)=W_i^n$. Hence $H_f\circ G_f(V_i^n)=V_i^n$ and $G_f\circ H_f(W_i^n)=W_i^n$.
Let $\epsilon>0$ and $N\in\n$ be such that
\[
\max_{i=1,\ldots,2^N}\operatorname{diam}(V_i^N)<\epsilon\text{ and }\max_{i=1,\ldots,2^N}\operatorname{diam}(W_i^N)<\epsilon.
\]
This implies that
\[
\|H_f(G_f(x))-x\|<\epsilon\text{ and }\|G_f(H_f(x))-x\|<\epsilon\text{ for every }x\in V.
\]
Since $\epsilon$ was arbitrary, this shows that $H_f\circ G_f=Id$ and $G_f\circ H_f=Id$. Thus  $H_f$ is a homeomorphism.

Note that the family of sets $\{V_i^n;n\in\n,i=1,\ldots,2^n\}$ generates the Borel $\sigma$-algebra on V.
Since $(H_f)^{-1}(V_i^n)=W_i^n=(H_f^n\circ\ldots\circ H_f^1)^{-1}(V^n_i)$, by \eqref{zalind}, we have
\[(H_f)_*\mu_f(V_i^n)=(H_f^n\circ\ldots\circ H_f^1)_*\mu_f(V^n_i)=\operatorname{Leb}(V_i^n).\]
It follows that $(H_f)_*\mu_f=Leb$.
%

In the reminder of the proof we will show that $H_f$ depends continuously on $f$.
Fix $\epsilon>0$ and then choose $m\in\n$ such that $2(5/4\sqrt{2})^m<\epsilon/3$.
Then
\[\max_{i=1,\ldots,2^m}\operatorname{diam}(V_i^m)<\frac{\epsilon}{3}\text{ and }\max_{i=1,\ldots,2^m}\operatorname{diam}(W_i^m)<\frac{\epsilon}{3}.\]
Since $(H_f)^{-1}(V_i^m)=W_i^m=(H_f^m\circ\ldots\circ H_f^1)^{-1}(V^m_i)$, it follows that
\[\sup_{x\in V}\|H_f(x)-H_f^{m}\circ\ldots\circ H_f^1(x)\|<\frac{\epsilon}{3}\text{ and }\sup_{x\in V}\|(H_f)^{-1}(x)-(H_f^{m}\circ\ldots\circ H_f^1)^{-1}(x)\|<\frac{\epsilon}{3}.\]
Therefore
\begin{equation}\label{ciag1}
d_{Hom}(H_f,H_f^{m}\circ\ldots\circ H_f^1)<\frac{\epsilon}{3}\text{ for every }f\in \mathcal W(V,\tfrac{1}{1+\hat\ep},\tfrac{1}{1-\hat\ep}).
\end{equation}
By (v) in Lemma \ref{trojkat0}, the maps $f\mapsto H_f^1\in Hom(V)$ and $f\mapsto\det D(H_f^1)^{-1}\in L^\infty(V)$ are continuous.
Suppose now that for 
$k\ge 1$ we proved that
\begin{equation}\label{eq:ind}
f\mapsto H_f^k\circ\ldots\circ H_f^1\in Hom(V)\ \text{ and }\ f\mapsto \det D(H_f^{k}\circ\ldots\circ H_f^1)^{-1}\in L^\infty(V)
\end{equation}
are continuous. We now prove that
\[f\mapsto H_f^{k+1}\circ\ldots\circ H_f^1\ \text{ and }\ f\mapsto \det D(H_f^{k+1}\circ\ldots\circ H_f^1)^{-1}\]
are also continuous. Since for every $i=1,\ldots,k$, $H_f^i$ and $(H_f^i)^{-1}$ are Lipschitz homeomorphisms with constant $\frac{5}{4}$, we get
\begin{equation}\label{eq:lip}
H_f^k\circ\ldots H_f^1\text{ and $(H_f^k\circ\ldots H_f^1)^{-1}$ are Lipschitz with constant }(\tfrac{5}{4})^k.
\end{equation}
Moreover by \eqref{iloczwyzn} we have
\[f_k=\det D(H_f^k\circ\ldots\circ H_f^1)^{-1}\cdot \Big(f\circ(H_f^k\circ\ldots\circ H_f^1)^{-1}\Big)\in\mathcal W(V,\tfrac{(1-\hat\ep)^k}{1+\hat{\ep}},\tfrac{(1+\hat\ep)^k}{1-\hat{\ep}}).\]
In view of Lemma \ref{trojkat0}, $H_f^{k+1}$ depends continuously on $f_k$. By \eqref{eq:ind} and \eqref{eq:lip}, Lemma \ref{ciagsklad} implies that $f\mapsto f\circ(H_{f}^k\circ\ldots\circ H_{f}^1)^{-1}\in L^1(V)$ is continuous.
Together with \eqref{eq:ind} this gives the continuity of
\[f\mapsto f_k =f\circ(H_{f}^k\circ\ldots\circ H_{f}^1)^{-1}\cdot\det D(H_f^k\circ\ldots\circ H_f^1)^{-1}\in L^1(V).
\]
It follows that $H_f^{k+1}$ depends continuously on $f$.

Again, since $H_f^i$ and $(H_f^i)^{-1}$ are Lipschitz with constant $\frac{5}{4}$,  for any $x\in V$ and $f,g\in \mathcal W(V,\frac{1}{1+\hat\ep},\frac{1}{1-\hat\ep})$  we have
\[
\begin{split}
\|H^{k+1}_f&(H_{f}^k\circ\ldots\circ H_{f}^1(x))-H^{k+1}_{g}(H_{g}^k\circ\ldots\circ H_{g}^1(x))\|\\
&\le\|H^{k+1}_f(H_{f}^k\circ\ldots\circ H_{f}^1(x))-H^{k+1}_{g}(H_{f}^k\circ\ldots\circ H_{f}^1(x))\|\\
&\quad +\|H^{k+1}_{g}(H_{f}^k\circ\ldots\circ H_{f}^1(x))-H^{k+1}_{g}(H_{g}^k\circ\ldots\circ H_{g}^1(x))\|\\
&\le d_{Hom}(H^{k+1}_f,H^{k+1}_{g})+\frac{5}{4}d_{Hom}(H_{f}^k\circ\ldots\circ H_{f}^1,H_{g}^k\circ\ldots\circ H_{g}^1)
\end{split}
\]
and similarly
\[
\begin{split}
\|(H_{f}^k&\circ\ldots\circ H_{f}^1)^{-1}(H^{k+1}_f)^{-1}(x)-(H_{g}^k\circ\ldots\circ H_{g}^1)^{-1}(H^{k+1}_{g})^{-1}(x)\|\\
&\le \Big(\frac{5}{4}\Big)^kd_{Hom}(H^{k+1}_f,H^{k+1}_{g})+d_{Hom}(H_{f}^k\circ\ldots\circ H_{f}^1,(H_{g}^k\circ\ldots\circ H_{g}^1)).
\end{split}
\]
This proves the continuous dependence of $H^{k+1}_f\circ\ldots\circ H^1_f$ on $f$. Finally, since $H_f^i$ are piecewise linear homeomorphisms, we have 
\[
D(H_f^{k+1}\circ\ldots\circ H_f^1)^{-1}=D(H_f^{k+1})^{-1}D(H_f^k\circ\ldots\circ H_f^1)^{-1}\text{ almost everywhere}.
\]
By (v) in Lemma \ref{trojkat0}, $f_k\mapsto \det D(H_f^{k+1})^{-1}\in L^\infty(V)$ depends continuously on $f_k$. Since $f\mapsto f_k\in L^1(V)$ is continuous,  it follows that $f\mapsto \det D(H_f^{k+1})^{-1}\in L^\infty(V)$ is also continuous. The uniform boundaries in \eqref{jedenwyzn} and in \eqref{iloczwyzn}, together with \eqref{eq:ind}, yield the continuity of $f\mapsto \det D(H_f^{k+1}\circ\ldots\circ H_f^1)^{-1}\in L^\infty(V)$.

Fix any $f\in \mathcal W(V,\frac{1}{1+\hat\ep},\frac{1}{1-\hat\ep})$. Then there exists $\delta>0$ such that for any $g\in \mathcal W(V,\frac{1}{1+\hat\ep},\frac{1}{1-\hat\ep})$ with $\|f-g\|_{L^1}<\delta$ we have
\begin{equation*}
d_{Hom}(H_f^{m}\circ\ldots\circ H_f^1,H_{g}^{m}\circ\ldots\circ H_{g}^1)<\frac{\epsilon}{3}.
\end{equation*}
Combining this with \eqref{ciag1} we obtain
\[
d_{Hom}(H_f,H_{g})<\epsilon,
\]
which concludes the proof of 
the continuity of $f\mapsto H_f$.
\end{proof}
\begin{uw}
\label{rem57}
Note that the 
above lemma is 
also valid for any triangle, since every two triangles are conjugated by an affine map (although 
the restriction on $\hat\ep$ may vary).
\end{uw}

The following theorem is the main result of this section.
\begin{tw}\label{glow}
Let $(M,\Sigma,\zeta)$ be a translation surface. There exists $\ep_\zeta=\ep>0$ such that, for all
$f\in\mathcal W(M,\tfrac{1}{1+\ep},\tfrac{1}{1-\ep})$,
there exists a homeomorphism $\mathcal{H}_f:(M,\Sigma)\to(M,\Sigma)$ such that $(\mathcal{H}_f)_*\mu_f=\la_\zeta$, where $\mu_f:=f \la_\zeta$. Moreover, $\mathcal{H}_f$ depends continuously on $f$.
\end{tw}
\begin{proof}
On $(M,\Sigma)$ consider a triangulation of $m+1$ triangles such that the set of vertices of this triangulation contains the set $\Sigma$.
By connectedness,  there is an ordering $\{U_i:0\leq i\leq m\}$ of the triangles such that, for every $i=1,\ldots,m$, the triangle $U_i$ has a common edge with some $U_{k(i)}$ for $0\le k(i)<i$. Indeed, choose any triangle $U_0$ from the triangulation. Next choose any neighbouring triangle as $U_1$ and set $k(1)=0$. Now suppose that for some $1\leq \ell\le m$ we have chosen triangles $\{U_i:0\le i\le \ell\}$  such that $k(i)<i$ for $1\leq i\leq \ell$. If $\ell=m$ then the process is over. If $\ell< m$ then choose as $U_{\ell+1}$ any triangle that has 
a common boundary with $\bigcup_{i=0}^{\ell}U_i$. This triangle exists by connectedness. Finally, let $0\le k(\ell+1)\le l$ be such that $U_{\ell+1}$ has 
a common edge with $U_{k(\ell+1)}$.

For every $i=1,\ldots,m$ consider a small isosceles right triangle $W_i\subset U_i\cup U_{k(i)}$ such that its shortest height lies on the common edge of $U_i$ and $U_{k(i)}$ and $W_i\cap W_j=\emptyset$ whenever $i\neq j$. Furthermore, we assume that each of the triangles $W_i$ is of the same size and we choose a parametrization such that $W_i$ has vertices in points $(0,-a),(0,a)$ and $(a,0)$,
with $(0,-a)\in U_i$ for each $i=1,\ldots,m$ (see Fig.~2).

	\begin{figure}[h]
 \begin{center}
		\includegraphics[scale=0.5]{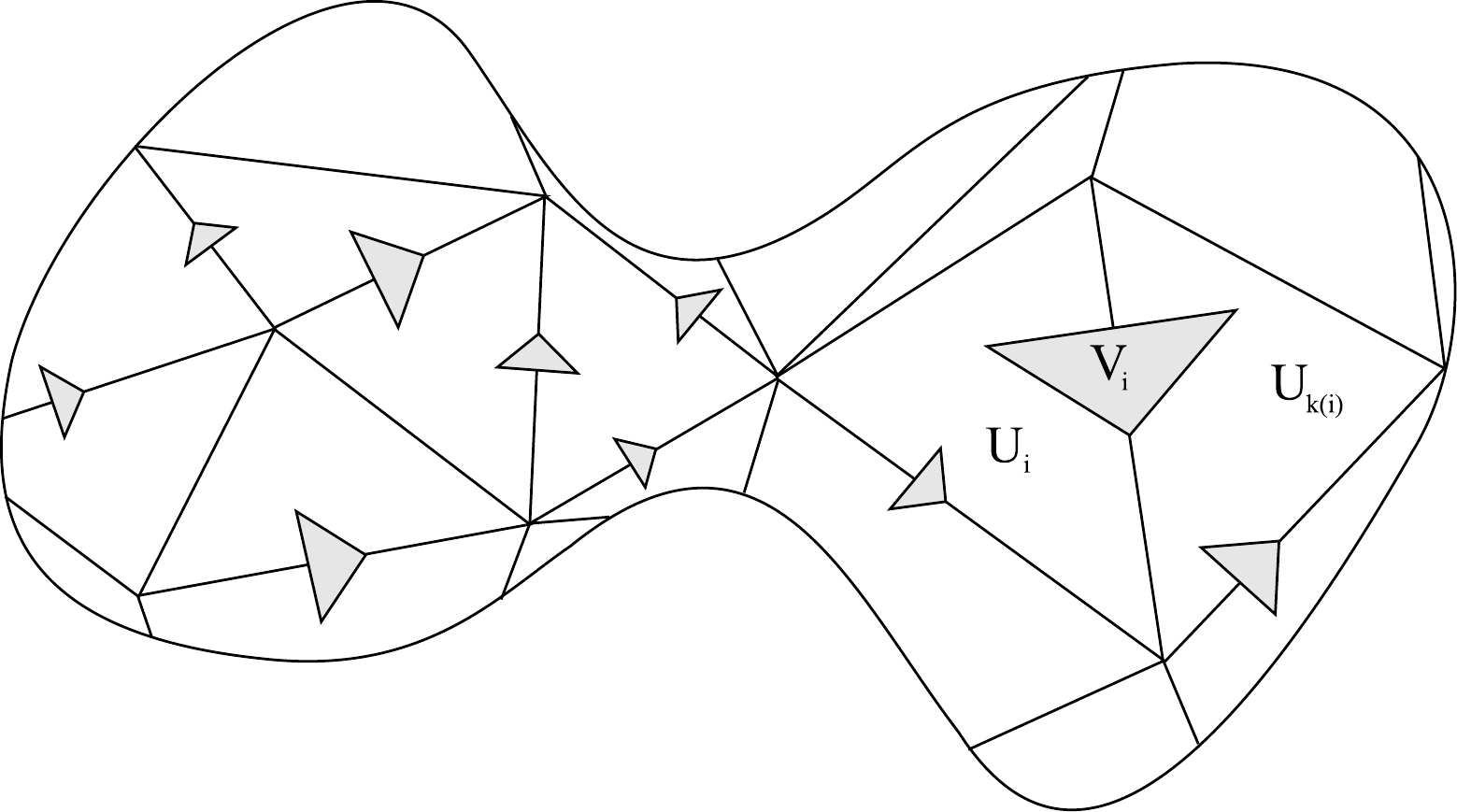}
		\caption{A triangulation of the surface $M$ and a choice of small triangles.}
 \end{center}
	\end{figure}

Let $B=[b_{ij}]_{i,j\in\{1,\ldots,m\}}$ be 
the matrix given by
\[
b_{ij}=\begin{cases}
1&\text{ if }i=j;\\
-1&\text{ if }i=k(j);\\
0&\text{ otherwise.}
\end{cases}
\]
For each positive 
$f\in\mathcal W(M,\tfrac{1}{1+\ep},\tfrac{1}{1-\ep})$, denote by $v^f\in\re^m$ the solution of the following system of linear  equations
\begin{equation}\label{uklrow}
Bv^f=[\la_\zeta(U_i)-\mu_f(U_i)]_{i=1,\ldots m}.
\end{equation}
Then $f\mapsto v^f$ is continuous.
Let $\hat\ep>0$ be such that Lemma \ref{trojkat} can be applied to any triangle $U_i$ (it exists due to Remark \ref{rem57}).
Observe that if $f$ is constant equal to $1$, then $v^f=(0,\ldots,0)$. By the continuity of $f\mapsto v^f$ we can  choose
$
0<\ep<\frac{\hat{\ep}}{3}
$
such that
\begin{equation}\label{warnatrojk2}
|v_i^f|<\frac{a^2\hat{\ep}}{12}
\end{equation}
for every 
$f\in\mathcal W(M,\tfrac{1}{1+\ep},\tfrac{1}{1-\ep})$.

Let 
$f\in\mathcal W(M,\tfrac{1}{1+\ep},\tfrac{1}{1-\ep})$. We now construct a 
family of piecewise affine homeomorphisms $G_f^i:W_i\to W_i$, 
$i=1,\ldots,m$, such that
 $G_f^i$ depends continuously on $f$,  $G^i_f|_{\partial W_i}=Id$,
\begin{equation}\label{linpart}
1-\frac{\hat{\ep}}{3}<\det(D(G_f^i)^{-1}(x))<1+\frac{\hat{\ep}}{3}\text{ whenever }D(G_f^i)^{-1}(x)\text{ is well defined,}
\end{equation}
and finally we require that 
the homeomorphism $\mathcal G_f:M\to M$ defined 
by
\[
\mathcal G_f(x):=
\begin{cases}
G_f^i(x)&\text{ if }x\in W_i\text{ for some }i=1,\ldots,m,\\
x&\text{ otherwise,}
\end{cases}
\]
satisfies
\begin{equation}\label{zlozmiara}
(\mathcal G_f)_*\mu_f(U_j)=\la_\zeta(U_j)\text{ for all }0\leq j\leq m.
\end{equation}

First note that for each $i=1,\ldots,m$ we can choose $-1<h_f^i<1$ such that the quadrilateral $Q_f^i\subset W_i$ with vertices in points $(0,0)(0,-a),(a,0)$ and $y^i_f:=(\sqrt{\hat{\ep}}a(1-|h_f^i|),h_f^ia)$ has measure $\mu_f$ equal to $\mu_f(W_i\cap U_i)+v_i^f$. Indeed,
 \[
 \mu_f(W_i\cap U_i)\ge\frac{1}{1+\ep}\la_\zeta(W_i\cap U_i)=\frac{a^2}{2(1+\ep)}\ge\frac{a^2}{4}>|v_i^f|
 \]
 and analogously
 \[
 \mu_f(W_i\cap U_{k(i)})>|v_i^f|.
 \]
Thus
\[
0<\mu_f(W_i\cap U_i)+v_i^f<\mu_f(W_i),
\]
which, together with the absolute continuity of $\mu_f$, yields the existence of $h_f^i$ for $i=1,\ldots,m$.

We now estimate $|h_f^i|$. Since $|v_i^f|$ is the $\mu_f$ measure  of the triangle with vertices $(0,0),(a,0)$ and $y_f^i$ in $W_i$ and $a|h_f^i|$ is its height, we have
\[
|v_i^f|>\frac{1}{1+\ep}\frac{a^2|h_f^i|}{2}.
\]
Hence, by \eqref{warnatrojk2},
\begin{equation}\label{hoszac}
|h_f^i|<\frac{2(1+\ep)|v_i^f|}{a^2}<\frac{(1+\ep)\hat{\ep}}{6}<\frac{\hat\ep}{3}.
\end{equation}

Let $G_f^i:W_i\to W_i$ be given by $G_f^i:=H(h_f^i,\sqrt{\hat{\ep}})$ for $i=1,\ldots,m$ as in \eqref{Hdef}.
Since \[(G_f^i)^{-1}(W_i\cap U_i)=Q_f^i,\] we have
\[
(\mathcal G_f)_*\mu_f(W_i\cap U_i)=\mu_f(Q_f^i)=\mu_f(W_i\cap U_i)+v_i.
\]
Analogously
\[
(\mathcal G_f)_*\mu_f(W_i\cap U_{k(i)})=\mu_f(W_i\setminus Q_f^i)=\mu_f(W_i)-\big(\mu_f(W_i\cap U_i)+v_i\big)=\mu_f(W_i\cap U_{k(i)})-v_i.
\]
By the definition of $v_i^f$ and by the fact that $W_i$ and $W_j$ are disjoint for $i\neq j$, we have 
for $i=1,\ldots,m$
\[
\begin{split}
(\mathcal G_f&)_*(\mu_f)(U_i)=
(\mathcal G_f)_*(\mu_f)\left(U_i\setminus (W_i\cup\bigcup_{\{j;k(j)=i\}}W_j)\right)\\
&+(\mathcal G_f)_*(\mu_f)(W_i\cap U_i)+\sum_{\{j;k(j)=i\}}(\mathcal G_f)_*(\mu_f)(W_{j}\cap U_i)\\
&=\mu_f\left(U_i\setminus (W_i\cup\bigcup_{\{j;k(j)=i\}}W_j)\right)+\big(\mu_f(W_i\cap U_i)+v^f_i\big)+\sum_{\{j;k(j)=i\}}\big(\mu_f(W_{j}\cap U_i)-v^f_{j}\big)\\
&=\mu_f(U_i)+v_i^f-\sum_{\{j;k(j)=i\}}v_j^f=\mu_f(U_i)+(Bv^f)_i=\la_\zeta(U_i).
\end{split}.
\]
Since $(G_f^i)_*(\mu_f)(W_i)=\mu_f(W_i)$, we also have
\[
(\mathcal G_f)_*(\mu_f)(M)=\mu_f(M)=\la_\zeta(M),
\]
and thus
\[
(\mathcal G_f)_*(\mu_f)(U_0)=\la_\zeta(U_0).
\]
Hence \eqref{zlozmiara} is satisfied. Moreover \eqref{hoszac} together with \eqref{proporcje2} and \eqref{proporcje3} yield \eqref{linpart}. What is left to prove is that $G_f^i$ depends continuously on $f$. By Lemma \ref{Hodh}, we only need to prove that $h_f^i$ depends continuously on $f$.

Take 
$f\in\mathcal W(M,\tfrac{1}{1+\ep},\tfrac{1}{1-\ep})$, and let $\de>0$. Since $v^g$ depends continuously on $g$, there exists $0<\delta'\leq\delta$ such that, 
for any $g\in\mathcal W(M,\tfrac{1}{1+\ep},\tfrac{1}{1-\ep})$,
\[
\|f-g\|_{L^1}<\delta'\Longrightarrow \max_{i=1,\ldots,m}|v_i^f-v_i^g|<\de.
\]
We now evaluate the difference 
between the respective Lebesgue measures of the quadrilaterals given by  $(0,0)$, $(0,-a)$, $(a,0)$, $y_f^i$ and 
by $(0,0)$, $(0,-a)$, $(a,0)$, $y_g^i$:
\[
\begin{split}
\frac{a^2|h_f^i-h_g^i|}{2}<&(1+\ep)|\mu_f(W_i\cap U_i)+v_i^f-\mu_g(W_i\cap U_i)-v_i^g+\int_{W_i}|f-g|d\la_\zeta|\\&\le(1+\ep)\big(|\mu_f(W_i\cap U_i)-\mu_g(W_i\cap U_i)|+|v_i^f-v_i^g|+\int_{W_i}|f-g|d\la_\zeta\big)\\&<(1+\ep)\big(\int_{W_i\cap U_i}|f-g|d\la_\zeta+\de+\int_{W_i}|f-g|d\la_\zeta|)\le(1+\ep)(2\|f-g\|_{L^1}+\de).
	\end{split}
	\]
Hence, for every 
$g\in\mathcal W(M,\tfrac{1}{1+\ep},\tfrac{1}{1-\ep})$ such that  $\|f-g\|_{L^1}<\delta'$, we have
\[
|h_f^i-h_g^i|<\frac{6(1+\ep)\de}{a^2},
\]
which implies that $h_f^i$ depends continuously on $f$.

Since $G_f^i$ 
depends continuously on $f$, by the definition $\mathcal G_f$ also depends continuously on $f$.
Since 
$f\in\mathcal W(M,\tfrac{1}{1+\ep},\tfrac{1}{1-\ep})$, and $\ep<\frac{\hat{\ep}}{3}$, 
by \eqref{linpart}, $(\mathcal G_f)_*\mu_f$ is an absolutely continuous measure with density $\hat f=(f\circ\mathcal G_f^{-1})\cdot \det D(\mathcal G_f^{-1})$ satisfying
\[
\frac{1}{1+\hat{\ep}}<\frac{1-\frac{\hat{\ep}}{3}}{1+\frac{\hat{\ep}}{3}}<\hat f<\frac{1+\frac{\hat{\ep}}{3}}{1-\frac{\hat{\ep}}{3}}<\frac{1}{1-\hat{\ep}}\] and
\[\int_{U_i}\hat fd\la_\zeta=(\mathcal G_f)_*\mu_f(U_i)=\la_\zeta(U_i)\text{ for every }i=0,\ldots,m.
\]
Therefore, on each $U_i$ the density $\hat{f}$ satisfies the assumptions of Lemma \ref{trojkat}.
Hence, for each $i=0,\ldots,m$, there exists a homeomorphism $H^i_{\hat f}:U_i\to U_i$ which transports 
the measure $(\mathcal G_f)_*\mu_{f}|_{U_i}$ to $\la_\zeta|_{U_i}$,
and 
such that $H^i_{\hat{f}}|_{\partial U_i}=Id$. Therefore, we can define a homeomorphism $H_{\hat f}:M\to M$  such that
\[
H_{\hat f}(x):=
H^i_{\hat f}(x)  \text{ whenever }x\in U_i.
\]
Then  $(H_{\hat f}\circ \mathcal G_f)_*\mu_f=( H_{\hat f})_*( (\mathcal G_f)_*\mu_f)=\la_\zeta$. Let $\mathcal H_f:=H_{\hat f}\circ \mathcal G_f$. What is left to prove is that $f\mapsto\mathcal H_f$ is continuous.

By Lemma \ref{trojkat} $H_{\hat f}^i$ depend continuously on $\hat f$ and hence $\hat f\mapsto H_{\hat f}$ is continuous. Moreover, (v) in Lemma \ref{trojkat0} implies that $f\mapsto\det(D\mathcal G_f^{-1})\in L^\infty(M)$ is continuous.
Furthermore, by (iv) in Lemma \ref{trojkat0}, the homeomorphism $\mathcal G_f^{-1}:M\to M$ is Lipschitz with constant $\frac{5}{4}$. Thus, by Lemma~\ref{ciagsklad} and Remark~\ref{ciagsklad1}, $f\mapsto f\circ\mathcal G_f^{-1}$ is continuous. Hence
\[
\mathcal W(M,\tfrac{1}{1+\ep},\tfrac{1}{1-\ep})\ni f\mapsto\hat f=(f\circ\mathcal G_f^{-1})\cdot \det D(\mathcal G_f^{-1})\in L^1(M)
\]
is continuous and 
this implies the continuity of $f\mapsto H_{\hat{f}}$.

 Now consider any 
 $f\in\mathcal W(M,\tfrac{1}{1+\ep},\tfrac{1}{1-\ep})$. Since $ H_{\hat{f}}:M\to M$  is uniformly continuous, for any 
 $\eta>0$ we can find $0<\de$ such that
\[
d_M(x,y)<\de\Rightarrow d_M( H_{\hat f}(x), H_{\hat f}(y))<\eta.
\]
Then, for every $x\in M$ and any 
$g\in\mathcal W(M,\tfrac{1}{1+\ep},\tfrac{1}{1-\ep})$ such that $d_{\operatorname{Hom}}( H_{\hat f}, H_{\hat g})<
\eta$ and
$d_{\operatorname{Hom}}(\mathcal G_f,\mathcal G_g)<\de$, we have
\[
\begin{split}
d_M(H_{\hat f}\circ \mathcal G_f(x), H_{\hat g}\circ \mathcal G_g(x))\le\ &
d_M(H_{\hat f}\circ \mathcal G_f(x), H_{\hat 
f}\circ \mathcal G_g(x))\\&+
d_M(H_{\hat f}\circ \mathcal G_g(x),H_{\hat g}\circ \mathcal G_g(x))<2
\eta.
\end{split}
\]
Analogously,
\[
d_M(( H_{\hat f}\circ \mathcal G_f)^{-1}(x),( H_{\hat g}\circ \mathcal G_g)^{-1}(x))<2
\eta.
\]
This concludes the proof of the continuity of $f\mapsto  \mathcal H_f$ and the proof of the whole theorem.

\end{proof}

\section{Local continuous embedding of the moduli space}\label{sec:emb}

In this section, we finalize the construction of a continuous mappings on open subsets of a connected component of any stratum, which is needed to prove the main result of this paper. We do it in two steps.

Firstly, for each $\zeta\in\mathcal M(M,\Sigma,\kappa)$ we construct 
a neighborhood $\mathcal U_{\zeta}\subset\mathcal M(M,\Sigma,\kappa)$ 
of $\zeta$, so that for every $\omega\in\mathcal U_\zeta$
there exists a piecewise affine homeomorphism $\mathfrak h_\omega:M\to M$ such that  $(\mathfrak h_\omega)_*\la_\omega=f_\omega\lambda_\zeta$  with $\frac{1}{1+\ep_\zeta}<f_\omega<\frac{1}{1-\ep_\zeta}$, where $\ep_\zeta>0$ is given by Theorem \ref{glow}. We will also require that $\omega\mapsto f_\omega\in L^1(M)$ is continuous.

Secondly, we use the results of the previous sections to show the existence of a homeomorphism $\mathcal H_\omega:M\to M$ such that $(\mathcal H_\omega\circ\mathfrak h_\omega)_*\la_\omega=\la_\zeta$. Moreover, we show that these homeomorphisms yield the existence of a continuous mapping $\mathfrak S:\mathcal U_\zeta\to \operatorname{Flow}(M,\zeta)$ such that $\mathfrak S(\omega)$ is isomorphic by a homeomorphism to $\mathcal T^\omega$ - the vertical translation flow on $(M,\omega)$.

\begin{lm}\label{techn}
	Let $\zeta\in\mathcal M(M,\Sigma,\kappa)$. There exists  a neigbourhood $\mathcal U_\zeta\subset M(M,\Sigma,\kappa)$ such that, for every $\omega\in \mathcal U_\zeta$, the following holds:
	\begin{enumerate}
		\item[(i)] there exists a triangulation $\mathcal Y(\omega)$ and a piecewise affine homeomorphism $\mathfrak h_\omega:(M,\omega)\to(M,\zeta)$ which is affine on elements of $\mathcal Y(\omega)$, fixes $\Sigma$ and 
		is Lipschitz with constant $\frac{11}{10}$,
		\item[(ii)] $(\mathfrak h_\omega)_*\la_\omega$ is an absolutely continuous measure with respect to $\la_\zeta$ with piecewise constant density $f_\omega$ satisfying $\frac{1}{1+\ep_\zeta}<f_\omega<\frac{1}{1-\ep_\zeta}$,
		\item[(iii)] the mapping $U_\zeta\ni\omega\mapsto f_\omega\in L^1(M,\la_\zeta)$ is continuous.
		\end{enumerate}
		Moreover, for given $\epsilon>0$, there exists $\de>0$ such that
		\begin{enumerate}
		\item[(iv)]
		for any $\bar\omega\in\mathcal U_\zeta$, if $d_{Mod}(\omega,\bar\omega)<\de$ then $\mathfrak h_\omega^{-1}\circ\mathfrak h_{\bar\omega}:(M,\bar{\omega})\to(M,\omega)$ is affine on elements of $\mathcal Y(\bar\omega)$, $\mathfrak h_{\bar{\omega}}^{-1}\circ\mathfrak h_{\omega}:(M,\omega)\to(M,\bar{\omega})$ is affine on elements of $\mathcal Y(\omega)$, they are both Lipschitz with constant $1+\epsilon$,
	\[
	\|Id-D(\mathfrak h_\omega^{-1}\circ\mathfrak h_{\bar\omega})|_A\|<\epsilon\text{ for every }A\in\mathcal Y(\bar\omega)
	\]
	and
		\[
		\|Id-D(\mathfrak h_{\bar\omega}^{-1}\circ\mathfrak h_{\omega})|_B\|<\epsilon\text{ for every }B\in\mathcal Y(\omega),
		\]
	\item[(v)] for any $\bar\omega\in\mathcal U_\zeta$ such that  $d_{Mod}(\omega,\bar\omega)<\de$ and for the set
	\begin{equation*}
	\tilde M(\omega):=\{x\in M;\inf_{\sigma\in\Sigma}d_\omega(\mathcal T_t^\omega(x),\sigma)>4\epsilon\text{ for all }t\in[-1,1]\},
	\end{equation*}
	we have $\la_\omega(\tilde M({\omega}))>1-K\epsilon$, where $K>0$ depends only on stratum, and for $x\in \tilde M({\omega})$ we have
	\[
	d_\omega(\mathcal T^\omega_t(x),\mathfrak h_\omega^{-1}\circ\mathfrak h_{\bar\omega}\circ\mathcal T^{\bar\omega}_t\circ\mathfrak h_{\bar\omega}^{-1}\circ\mathfrak h_{\omega}(x))<\epsilon\text{ for any }t\in[-1,1].
	\]
	and for every $\sigma\in \Sigma$
	\[
		d_\omega(\sigma,\mathfrak h_\omega^{-1}\circ\mathfrak h_{\bar\omega}\circ\mathcal T^{\bar\omega}_t\circ\mathfrak h_{\bar\omega}^{-1}\circ\mathfrak h_{\omega}(x))>3\epsilon\text{ for any }t\in[-1,1].
	\]
	\end{enumerate}
\end{lm}
\begin{proof}
	Let $\pi$ be a permutation of 
	the alphabet $\mathcal A$ 
	with $d$ elements which belongs to a Rauzy class corresponding to $\mathcal M(M,\Sigma,\kappa)$.
	We can assume that $\zeta$ has no vertical saddle-connections and thus there is a polygonal representation $(\pi,\la^\zeta,\tau^\zeta)$ of $\zeta$. Otherwise we can rotate $\zeta$ to obtain a form $\zeta'$ which does not admit vertical saddle connections, construct a triangulation $\mathcal Y(\zeta')$ and rotate it back together with this triangulation to obtain a triangulation $\mathcal Y(\zeta)$ (note that 
	a rotation is an isometry and 
	that it acts continuously on $\mathcal M(M,\Sigma,\kappa)$, see \cite{Yoccoz}).
	
	Let $\omega\in\mathcal M(M,\Sigma,\kappa)$, and assume that $\omega=M(\pi,\la^{\omega},\tau^{\omega})$ for some $\la^{\omega},\tau^{\omega}$. Let $\mathcal P(\omega)\subset \mathbb{C}$ be the polygon corresponding to $\omega$, whose vertices $R_0(\omega),R_1(\omega),\ldots,R_d(\omega),R_1'(\omega),\ldots,R_{d-1}'(\omega)$ are given by
	\[
	R_i(\omega):=\sum_{\{\al;\pi_0(\al)\le i\}}(\la_\al^\omega+i\tau_\al^\omega)\quad\text{ and }\quad R_i'(\omega):=\sum_{\{\al;\pi_1(\al)\le i\}}(\la_\al^\omega+i\tau_\al^\omega)\quad\text{ for }i=0,\ldots,d.
	\]
	Note that $R_0(\omega)=R_0'(\omega)=0$, $R_d(\omega)=R_d'(\omega)$. 
	For $i=1,\ldots,d-1$ consider the vertical segments connecting $R_i(\omega)$ and $R_i'(\omega)$ with the opposite side of $\mathcal P(\omega)$. Denote the other endpoints of those segments by $Q_i(\omega)$ and $Q_i'(\omega)$ respectively (see Fig. 3). Since each side on the upper half of the polygon is identified with one of the sides on the lower half of the polygon, there exist representations of $Q_i(\omega)$ and $Q_i'(\omega)$ on the opposite half of the polygon which we denote by $S_i(\omega)$ and $S_i'(\omega)$ respectively.
	 Note that
	 \[
	 \operatorname{Re}(S_i(\omega))=T_{\pi,\la^\omega}^{-1}(\operatorname{Re}(R_i(\omega)))\quad\text{and}\quad\operatorname{Re}(S_i'(\omega))=T_{\pi,\la^\omega}(\operatorname{Re}(R_i'(\omega))),
	 \]
	 where $T_{\pi,\la^\omega}$ is the IET given by $(\pi,\la)$.
	 	 \begin{figure}[h]\label{polygon1}
 \begin{center}
	 	 	\includegraphics[scale=0.4]{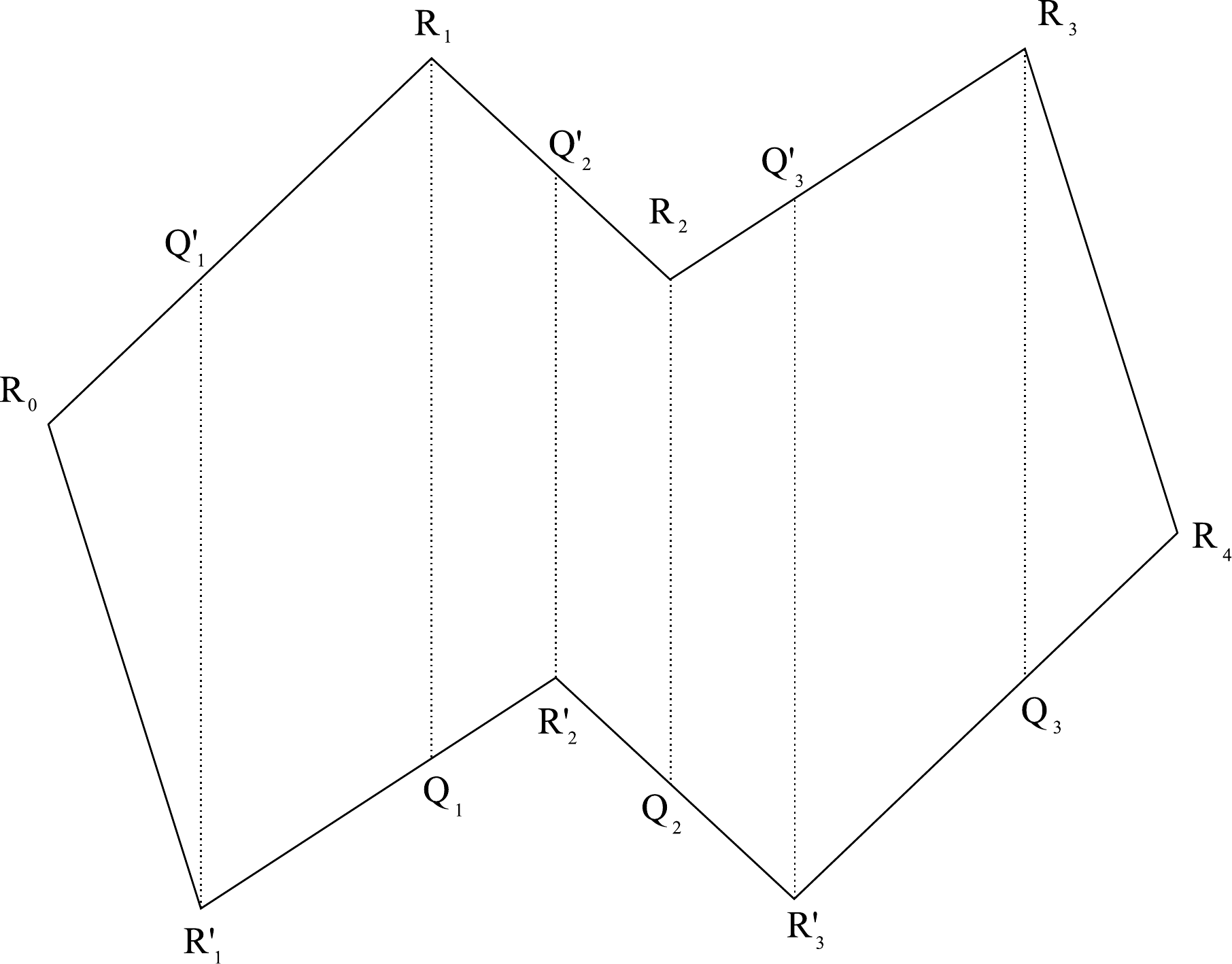}
	 	 	\caption{The vertices of $\mathcal P(\omega)$ and their projections on opposite sides.}
 \end{center}
	 	 \end{figure}
	 Let
	\[\mathcal V(\omega):=\{R_0(\omega),\ldots,R_d(\omega),R_1'(\omega),\ldots,R_{d-1}'(\omega),S_1(\omega),\ldots,S_{d-1}(\omega),S_1'(\omega),\ldots,S_{d-1}'(\omega)\}.
	\]
	  $\mathcal V(\omega)$ is fully determined by $(\la^\omega,\tau^\omega)$. If $\omega$ has no vertical saddle connections, then for all distinct 
	  $x,y\in\mathcal V(\omega)$ we have $\operatorname{Re}(x)\neq\operatorname{Re}(y)$. Consider the sequence $\{V_j(\omega)\}_{j=0}^{4d-2}$, which is an ordering of $\mathcal V(\omega)$, such that $\{\operatorname{Re}(V_j)\}_{j=0}^{4d-2}$ 
	  be an increasing sequence.

	Let $\ep_1>0$ be such that for every $\omega\in\mathcal M(M,\Sigma,\kappa)$ satisfying $d_{Mod}(\zeta,\omega)<\ep_1$, the orderings of the sets $\mathcal V(\zeta)$ and $\mathcal V(\omega)$ are the same, that is
	\[
	V_j(\omega)=R_i(\omega)\text{ iff } V_j(\zeta)=R_i(\zeta);\quad\quad
	V_j(\omega)=R_i'(\omega)\text{ iff }V_j(\zeta)=R_i'(\zeta),
	\]	
	\[
	V_j(\omega)=S_i(\omega)\text{ iff } V_j(\zeta)=S_i(\zeta)\quad\text{and}\quad
	V_j(\omega)=S_i'(\omega)\text{ iff }V_j(\zeta)=S_i'(\zeta),
	\]	
	for every $j=0,\ldots,2d$, and $\{\operatorname{Re}(V_i(\omega))\}_{i=0}^d$ is strictly increasing.
	
	For every $\omega\in \mathcal M(M,\Sigma,\kappa)$ with $d_{Mod}(\zeta,\omega)<\ep_1$, we now construct a triangulation  $\mathcal Y(\omega)$ of $\mathcal P(\omega)$. (We abuse the word ``triangulation'' since  the edges may connect vertices of triangulation which are actually the same points.) Let $\{r(k)\}_{0\leq k\leq 2d}$  be 
	the strictly increasing sequence in $\{0,\ldots,4d-2\}$ such that 
	$V_{r(0)},\ldots,V_{r(2d)}$ 
	be the vertices of $\mathcal P(\omega)$. Let $\tilde V_{r(k)}(\omega):=Q_i(\omega)$ whenever $V_{r(k)}(\omega)=R_i(\omega)$ and analogously let $\tilde V_{r(k)}(\omega):=Q_i'(\omega)$ whenever $V_{r(k)}(\omega)=R_i'(\omega)$.
	
	Consider the triangle given by 
	the points $V_{r(0)}=0,V_{r(1)},\tilde V_{r(1)}$. If $r(1)=1$, then this triangle belongs to $\mathcal Y(\omega)$. If $r(1)\neq 1$ and $\operatorname{Im}(V_1(\omega))>0$, then 
	we connect by segments all $V_i(\omega)$ such that $i\le r(1)$ and $\operatorname{Im}(V_i(\omega))<0$ with $V_1(\omega)$,
	and for all $i\le r(1)$ such that $\operatorname{Im}(V_i(\omega))>0$, we connect $V_i(\omega)$ with the point $V_{r(1)}(\omega)$ or $\tilde V_{r(1)}(\omega)$, whichever has the negative imaginary part. If $r(1)\neq 1$ and $\operatorname{Im}(V_1(\omega))<0$, then we proceed symmetrically. The triangles obtained by using the above segments are elements of $\mathcal Y(\omega)$. By applying vertical reflection, we use the same construction for the triangle given by points $V_{r(2d-1)}(\omega),\tilde V_{r(2d-1)}(\omega),V_{r(2d)}(\omega)$.
	
	For every $k=1,\ldots,2d-2$, consider 
	the trapezoid given by 
	$V_{r(k)}(\omega)$, $V_{r(k+1)}(\omega)$, $\tilde V_{r(k)}(\omega)$ and $\tilde V_{r(k+1)}(\omega)$. In each of those trapezoids, we take 
	the diagonal connecting the top-left vertex 
	with the bottom-right vertex. If $r(k+1)=r(k)+1$, then the two resulting triangles belong to $\mathcal Y(\omega)$. If $r(k+1)\neq r(k)+1$, then for every $r(k)<i<r(k+1)$, 
	we connect $V_i(\omega)$ with 
	the bottom-right vertex if $\operatorname{Im}(V_i(\omega))>0$,  and with 
	the top-left vertex if $\operatorname{Im}(V_i(\omega))<0$. 
	We include the resulting triangles into $\mathcal Y(\omega)$. 
	In this way we get a triangulation $\mathcal Y(\omega)$ of $\mathcal P(\omega)$ into triangles which have vertices in $\mathcal V(\omega)$ (see Fig.~4).
	\begin{figure}[h]
 \begin{center}
		\includegraphics[scale=0.4]{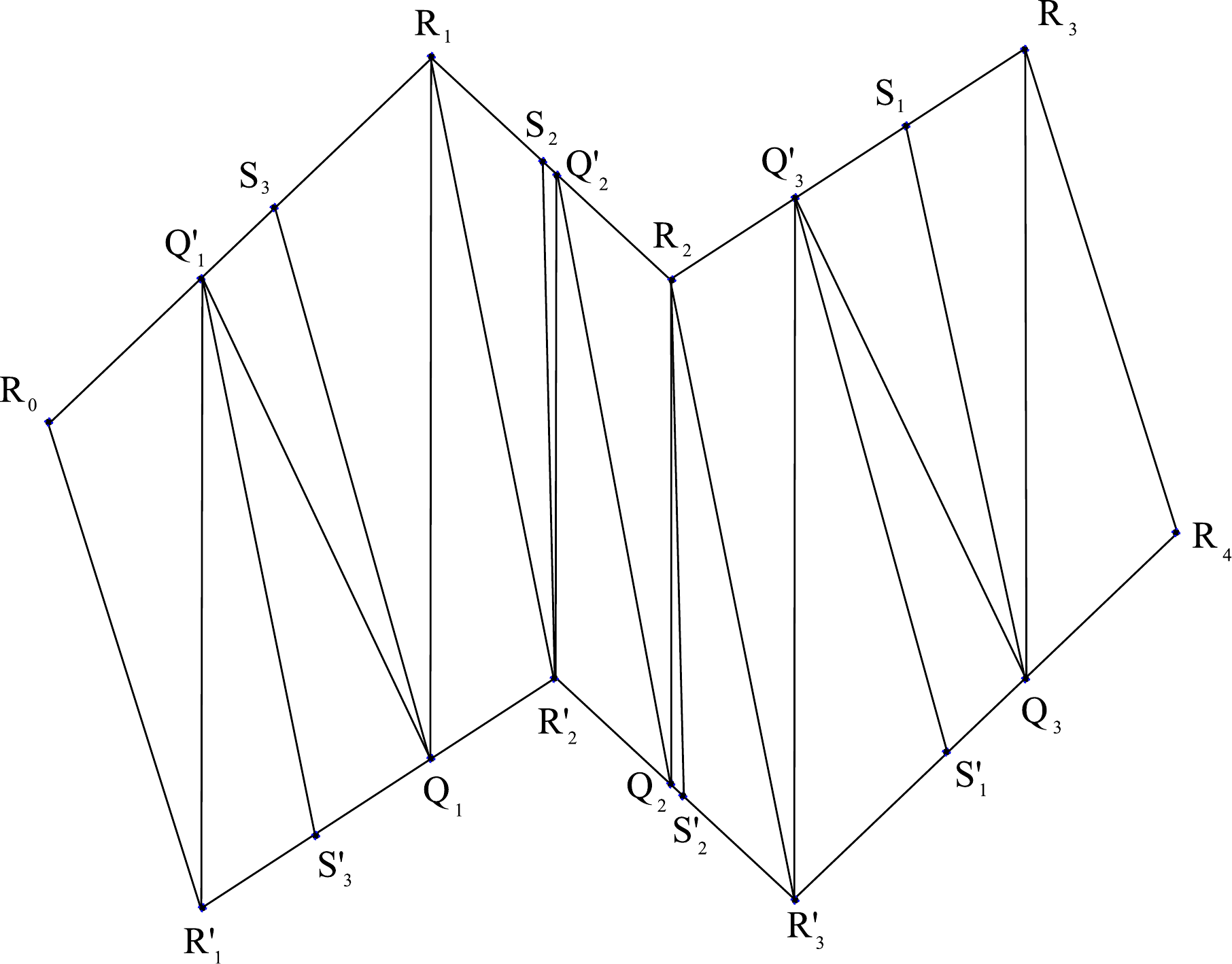}
		\caption{The triangulation $\mathcal Y(\omega)$ for the polygon in Fig. 3.}
 \end{center}
 		\end{figure}
	
	Define $\mathfrak{h}_\omega:(M,\omega)\to (M,\zeta)$ as 
	the piecewise affine transformation, such that
	\[
	\mathfrak h_\omega(V_i(\omega))=V_i(\zeta)\text{ for every }i=0,\ldots,4d-2
	\]
	and 
	sending affinely each triangle from $\mathcal Y(\omega)$ with vertices $V_j(\omega), V_k(\omega), V_{\ell}(\omega)$ onto the triangle with vertices $V_j(\zeta), V_k(\zeta), V_{\ell}(\zeta)$. Note that the map $\mathfrak{h}_\omega$ is uniquely determined by the points in $\mathcal V(\omega)$. Moreover, since $\Sigma\subset\mathcal V(\omega)$, 
	$\mathfrak h_\omega$ fixes $\Sigma$.
	
	Let $0<\ep_2<\ep_1$ be such that, defining $\tilde{\mathcal U}_\zeta:=\{\omega\in\mathcal M(M,\Sigma,\kappa),d_{Mod}(\zeta,\omega)<\ep_2 \}$, we have
	\begin{equation}
	\omega\in\tilde{\mathcal U}_\zeta\Rightarrow 1-\ep_\zeta<\frac{\la_\zeta(A)}{\la_\omega(\mathfrak h_\omega^{-1}(A))}<1+\ep_\zeta\text{ for every }A\in\mathcal Y(\zeta).
	\end{equation}
	This implies that $(\mathfrak h_\omega)_*\la_\omega$ is absolutely continuous with respect to $\lambda_\zeta$ and has a piecewise constant density $f_\omega$ given by
	\begin{equation}\label{wzornaf}
	f_\omega(x)=\frac{\la_\omega(\mathfrak h_\omega^{-1}(A))}{\la_\zeta(A)}\text{ for every }x\in A\text{ where }A\in\mathcal Y(\zeta).
	\end{equation}
	Hence
	\[
	\frac{1}{1+\ep_\zeta}<f_\omega<\frac{1}{1-\ep_\zeta}.
	\]
	Moreover, note that for every $A\in\mathcal Y(\zeta)$, the mapping $\omega\to\la_\omega(\mathfrak h_\omega^{-1}(A))$ is continuous. Hence, the formula \eqref{wzornaf} implies the continuity of $\omega\mapsto f_\omega\in L^1(M,\lambda_\zeta)$.
	
	 Let $\omega\in\tilde{\mathcal U}_\zeta$. Then for any $\bar{\omega}\in\tilde{\mathcal U}_\zeta$, $\mathfrak h_{\bar\omega}^{-1}\circ\mathfrak h_{\omega}$ is a continuous piecewise affine homoeomorphism which is affine on the elements of $\mathcal Y(\omega)$ and
	 \[
	 \mathfrak h_{\bar\omega}^{-1}\circ\mathfrak h_{\omega}(V_i(\omega))=V_i(\bar{\omega}).
	 \]
	 Take any $A\in\mathcal Y(\zeta)$ and let $V_j(\zeta),V_k(\zeta),V_{\ell}(\zeta)\in\mathcal V(\zeta)$ be its vertices. Then $\operatorname{lin}(\mathfrak h_{\bar\omega}^{-1}\circ\mathfrak h_{\omega}|_{\mathfrak h_\omega^{-1}A})$ is given by a matrix $B_A(\omega,\bar{\omega})=[b_{ij}(\omega,\bar{\omega})]_{i,j=1,2}$, where
	 \[
	 	 b_{1,1}(\omega,\bar{\omega})=\frac{\operatorname{Re}(V_k(\bar\omega)-V_j(\bar{\omega}))\operatorname{Im}(V_{\ell}(\omega)-V_j(\omega))-\operatorname{Re}(V_{\ell}(\bar\omega)-V_j(\bar{\omega}))\operatorname{Im}(V_k(\omega)-V_j(\omega))}{\operatorname{Re}(V_k(\omega)-V_j(\omega))\operatorname{Im}(V_{\ell}(\omega)-V_j(\omega))-\operatorname{Re}(V_{\ell}(\omega)-V_j(\omega))\operatorname{Im}(V_k(\omega)-V_j(\omega))},
	\]
	\[
		 	 b_{1,2}(\omega,\bar{\omega})=\frac{\operatorname{Re}(V_{\ell}(\bar\omega)-V_j(\bar{\omega}))\operatorname{Re}(V_k(\omega)-V_j(\omega))-\operatorname{Re}(V_k(\bar\omega)-V_j(\bar{\omega}))\operatorname{Re}(V_{\ell}(\omega)-V_j(\omega))}{\operatorname{Re}(V_k(\omega)-V_j(\omega))\operatorname{Im}(V_{\ell}(\omega)-V_j(\omega))-\operatorname{Re}(V_{\ell}(\omega)-V_j(\omega))\operatorname{Im}(V_k(\omega)-V_j(\omega))},
	\]
	\[
		 	 b_{2,1}(\omega,\bar{\omega})=\frac{\operatorname{Im}(V_k(\bar\omega)-V_j(\bar{\omega}))\operatorname{Im}(V_{\ell}(\omega)-V_j(\omega))-\operatorname{Im}(V_{\ell}(\bar\omega)-V_j(\bar{\omega}))\operatorname{Im}(V_k(\omega)-V_j(\omega))}{\operatorname{Re}(V_k(\omega)-V_j(\omega))\operatorname{Im}(V_{\ell}(\omega)-V_j(\omega))-\operatorname{Re}(V_{\ell}(\omega)-V_j(\omega))\operatorname{Im}(V_k(\omega)-V_j(\omega))},
	\]
	and
	\[
			 	 b_{2,2}(\omega,\bar{\omega})=\frac{\operatorname{Re}(V_k(\omega)-V_j(\omega))\operatorname{Im}(V_{\ell}(\bar\omega)-V_j(\bar{\omega}))-\operatorname{Im}\operatorname{Re}(V_{\ell}(\omega)-V_j(\omega))(V_k(\bar\omega)-V_j(\bar{\omega}))}{\operatorname{Re}(V_k(\omega)-V_j(\omega))\operatorname{Im}(V_{\ell}(\omega)-V_j(\omega))-\operatorname{Re}(V_{\ell}(\omega)-V_j(\omega))\operatorname{Im}(V_k(\omega)-V_j(\omega))}.
	\]
	Note that, to obtain a formula for $B_A(\bar{\omega},\omega)=\operatorname{lin}(\mathfrak h_{\omega}^{-1}\circ\mathfrak h_{\bar\omega}|_{\mathfrak h_{\bar\omega}^{-1}A})$, it is enough to switch $\omega$ with $\bar{\omega}$. Observe that $B_A(\omega,\omega)=Id$.
	Since all coefficients depend continuously on the elements of $\mathcal V(\bar\omega)$ and $\mathfrak h_\zeta=Id$, by taking $\bar\omega=\zeta$ we can find $0<\ep_3\le\ep_2$ such that for all $\omega\in\mathcal U_\zeta:=\{\omega\in\mathcal M(M,\Sigma,\kappa),d_{Mod}(\zeta,\omega)<\ep_3 \}$, $\mathfrak h_{\omega}$ and $\mathfrak h_{\omega}^{-1}$ are Lipschitz with constant $\frac{11}{10}$. Moreover, for any $\epsilon>0$ and any $\omega\in\mathcal U_\zeta$ we can find $\de>0$ such that for all $A\in\mathcal Y(\zeta)$ and for all $\bar{\omega}$ satisfying $d_{Mod}(\omega,\bar{\omega})<\de$, we have that $\mathfrak h_{\omega}^{-1}\circ\mathfrak h_{\bar\omega}$ and $\mathfrak h_{\bar\omega}^{-1}\circ\mathfrak h_{\omega}$ are Lipschitz with constant $1+\epsilon$ and $\|Id -B_A(\omega,\bar{\omega})\|<\epsilon$ and $\|Id-B_A(\bar{\omega},\omega)\|<\epsilon$.
	
	To prove (v), note  that the set $\tilde M(\omega)$ does not contain points which are in the $4\epsilon$-neighbourhood of ingoing and outgoing
	vertical separatrix segments 
	of length $1$ starting from singular points. That is the 
	complement $\tilde M(\omega)^c$ is of measure $\la_\omega$ at most $(1+4\epsilon)8\epsilon$ times the number of ingoing and outgoing separatrices, which determines the value of $K$.
		
	For every $\bar\omega\in\mathcal U_\zeta$ denote by $X^{\bar\omega}:M\to\re^2$ the unit constant vertical vector field on $(M,\bar{\omega})$ defined on $M\setminus\Sigma$ which generates $\mathcal T^{\bar\omega}$. Then by (iv) we have that
	\begin{equation}\label{odlegpol}
	\|X^\omega(x)-D(\mathfrak h_{\omega}^{-1}\circ \mathfrak h_{\bar\omega})_{h_{\bar\omega}^{-1}\circ \mathfrak h_{\omega}(x)}X^{\bar{\omega}}(\mathfrak h_{\bar\omega}^{-1}\circ \mathfrak h_{\omega}(x))\|<\epsilon.
	\end{equation}
	Note that the vector field $D(\mathfrak h_{\omega}^{-1}\circ \mathfrak h_{\bar\omega})_{h_{\bar\omega}^{-1}\circ \mathfrak h_{\omega}(x)}X^{\bar{\omega}}(h_{\bar\omega}^{-1}\circ \mathfrak h_{\omega}(x))$ is well defined everywhere except 
	on the edges of the triangulation $\mathcal Y(\omega)$. Since this vector field is constant on the interiors of 
	the elements of 
	the triangulation $\mathcal Y(\omega)$, we can define 
	it on $M\setminus \Sigma$ by choosing on each edge of $\mathcal Y(\omega)$ one of the vectors derived from one of the two triangles forming this edge. It is worth to mention that,
	if the direction of an edge of $\mathcal Y(\bar\omega)$ given by some triangles $\mathfrak h_{\bar\omega}^{-1}(A)$ and $\mathfrak h_{\bar\omega}^{-1}(B)\in\mathcal Y(\bar{\omega})$
	coincides with the direction of the flow $\mathcal T^{\bar{\omega}}$, then the vector field $D(\mathfrak h_{\omega}^{-1}\circ \mathfrak h_{\bar\omega})_{h_{\bar\omega}^{-1}\circ \mathfrak h_{\omega}(x)}X^{\bar{\omega}}(h_{\bar\omega}^{-1}\circ
	 \mathfrak h_{\omega}(x))$ takes the same values on 
	 the triangles $\mathfrak h_{\omega}^{-1}(A),\mathfrak h_{\omega}^{-1}(B)\in\mathcal Y({\omega})$. Hence the flow induced by this vector field is well defined and $\eqref{odlegpol}$ holds everywhere on $M\setminus \Sigma$.
	Thus, in view of \eqref{odlegpol}, for any $t\in[-1,1]$ we have
	\[
	\begin{split}
	&\quad\Big|\int_0^t\|X^\omega(\mathcal T_s^\omega(x))-D(\mathfrak h_{\omega}^{-1}\circ \mathfrak h_{\bar\omega})_{\mathcal T_s^{\bar\omega}\circ \mathfrak h_{\bar\omega}^{-1}\circ \mathfrak h_{\omega}(x)}X^{\bar\omega}(\mathcal T_s^{\bar\omega}\circ\mathfrak h_{\bar\omega}^{-1}\circ \mathfrak h_{\omega}(x))\|ds\Big|\\
	&=\Big|\int_{0}^{t}\|X^\omega(\mathfrak h_\omega^{-1}\circ\mathfrak h_{\bar{\omega}}\circ\mathcal T_s^{\bar\omega}\circ\mathfrak h_{\bar\omega}^{-1}\circ \mathfrak h_{\omega}(x))-D(\mathfrak h_{\omega}^{-1}\circ \mathfrak h_{\bar\omega})_{\mathcal T_s^{\bar\omega}\circ \mathfrak h_{\bar\omega}^{-1}\circ \mathfrak h_{\omega}(x)}X^{\bar\omega}(\mathcal T_s^{\bar\omega}\circ\mathfrak h_{\bar\omega}^{-1}\circ \mathfrak h_{\omega}(x))\|ds\Big|\\
	&<|t|\epsilon\le\epsilon.
	\end{split}
	\]
	Since for every $x\in\tilde M(\omega)$ we have
	\[
	\mathfrak h_{\omega}^{-1}\circ \mathfrak h_{\bar\omega}\circ \mathcal T_t^{\bar\omega} \circ \mathfrak h_{\bar\omega}^{-1}\circ \mathfrak h_{\omega}(x)=x+\int_0^tD(\mathfrak h_{\omega}^{-1}\circ \mathfrak h_{\bar\omega})_{\mathcal T_s^{\bar\omega}\circ \mathfrak h_{\bar\omega}^{-1}\circ \mathfrak h_{\omega}(x)}X^{\bar\omega}(\mathcal T_s^{\bar\omega}(\mathfrak h_{\bar\omega}^{-1}\circ \mathfrak h_{\omega}(x)))ds
	\]
	in local coordinates, 
	we deduce that
	\[
	\begin{split}
		d_\omega&(\mathcal T_t^\omega(x),\mathfrak h_{\omega}^{-1}\circ \mathfrak h_{\bar\omega}\circ \mathcal T_t^{\bar\omega} \circ \mathfrak h_{\bar\omega}^{-1}\circ \mathfrak h_{\omega}(x))\\
		&\le\Big|\int_0^t\|X^\omega(\mathcal T_s^\omega(x))-D(\mathfrak h_{\omega}^{-1}\circ \mathfrak h_{\bar\omega})_{\mathcal T_s^{\bar\omega}\circ \mathfrak h_{\bar\omega}^{-1}\circ \mathfrak h_{\omega}(x)}X^{\bar\omega}(\mathcal T_s^{\bar\omega}(\mathfrak h_{\bar\omega}^{-1}\circ \mathfrak h_{\omega}(x)))\|ds\Big|
		\le\epsilon.
		\end{split}
	\]
	Since $x\in\tilde M({\omega})$, this also implies that for every $\sigma\in\Sigma$ we have
	\[
	d_\omega(\sigma,\mathfrak h_{\omega}^{-1}\circ \mathfrak h_{\bar\omega}\circ \mathcal T_t^{\bar\omega} \circ \mathfrak h_{\bar\omega}^{-1}\circ \mathfrak h_{\omega}(x))
	>3\epsilon.
	\]
	This concludes the proof of (v) and thus the 
	proof of the whole lemma.	
\end{proof}

\begin{lm}\label{prostokat}
	Let $\omega\in\mathcal M(M,\Sigma,\kappa)$ and let $D$ be a rectangle in $(M,\omega)$. For every $\ep>0$ there exists $\de>0$ such that, for every $\la_{\omega}$-measure preserving $F:D\to(M,\omega)$ satisfying
	\begin{equation}\label{prostokat1}
	\sup_{x\in D}d_\omega(x,F(x))<\de,
	\end{equation}
	we have
	\begin{equation}\label{prostokat2}
	\la_\omega(D)-\la_\omega(D\cap F(D))<\ep.
	\end{equation}
	\begin{proof}
		Let $\ep>0$. We assume that $\ep<\la_\omega(D)$, otherwise the result is obvious. Choose $\de>0$ such that the set
		\[
		\hat D:=\{x\in D; \forall {y\in M},\ d_\omega(x,y)<\de\Rightarrow y\in D \}
		\]
		  has measure $\la_\omega(\hat D)>\la_\omega(D)-\ep$. Then, for any $\la_\omega$-measure preserving $F:D\to(M,\omega)$ satisfying \eqref{prostokat1}, we have 
		  $F(\hat D)\subset D$. Hence
		\[
		\la_\omega(D\cap F(D))\ge\la_\omega(F(\hat D))=\la_\omega(\hat D)>\la_\omega(D)-\ep.
		\]
	\end{proof}
\end{lm}

\begin{tw}\label{wloz}
Let $\zeta\in\mathcal M(M,\Sigma,\kappa)$ and let $\mathcal U_\zeta$ be 
the neighbourhood given by Lemma \ref{techn}. There exists a continuous mapping $\mathfrak{S}:\mathcal{U}_\zeta\to \operatorname{Flow}(M,\la_\zeta)$
such that for every $\omega\in\mathcal U_\zeta$ the vertical flow on $(M,\omega)$ is measure-theoretically isomorphic by a homeomorphism to the measure-preserving flow $\mathfrak{S}(\omega)$ on $(M,\la_\zeta)$.
\end{tw}
\begin{proof}
By (i) and (ii) in Lemma \ref{techn}, for every $\omega\in\mathcal U_\zeta$ there exists a homeomorphism $\mathfrak h_\omega:M\to M$ fixing $\Sigma$ and such that $(\mathfrak h_\omega)_*\la_\omega=f_\omega\la_\zeta$ , where $f_\omega$ satisfies $\frac{1}{1+\ep_\zeta}<f_\omega<\frac{1}{1-\ep_\zeta}$. Hence we can apply Theorem \ref{glow} to obtain a homeomorphism $\mathcal H_\omega:=\mathcal H_{f_\omega}:M\to M$, which depends continuously on $f_\omega$ and  $(\mathcal H_{f_\omega})_*(f_\omega\la_\zeta)=\la_\zeta$. By (iii) in Lemma \ref{techn}, it follows that the map $\omega\mapsto f_\omega$ is continuous. Hence $\omega\mapsto\mathcal H_\omega$, as a composition of two continuous mappings, is also continuous. Now define a homeomorphism of $M$
\[
\mathcal S_\omega:=\mathcal H_{\omega}\circ \mathfrak h_\omega.
\]
Note that $(\mathcal S_\omega)_*\la_\omega=\la_\zeta$ and the flow $\mathcal S_\omega\circ\mathcal T^\omega\circ\mathcal S_\omega^{-1}$ is $\la_\zeta$-measure preserving. To conclude the proof we show now that the mapping $\mathcal U_\zeta\ni\omega\mapsto \mathcal S_\omega\circ\mathcal T^\omega\circ\mathcal S_\omega^{-1}=:\mathfrak S(\omega)\in\operatorname{Flow}(M,\la_\zeta)$ is continuous.

 Fix $\omega\in\mathcal U_\zeta$. We now prove the continuity of $\mathfrak S$ in $\omega$. On $(M,\omega)$ choose a family $\mathcal Q$ of open rectangles,
 with vertical and horizontal sides, that generates 
 the Borel  $\sigma$-algebra on $M$. We may assume that for every $Q\in\mathcal Q$ we have $\la_\omega(Q)\le \frac{1}{4}$.
 Note that $\mathcal S_\omega^{-1}:(M,\zeta)\to(M,\omega)$ is a 
 measure-theoretic isomorphism. Hence, in view of Remark \ref{transport}, to prove that $\mathfrak S$ is continuous, it is sufficient to prove that the map $\mathcal U_\zeta\ni\bar{\omega}\mapsto \mathcal S_\omega^{-1}\circ\mathfrak S(\bar{\omega})_t\circ\mathcal S_\omega\in\operatorname{Flow}(X,\la_\omega)$ is continuous. That is for every $\ep>0$ and $Q\in\mathcal Q$ there exists $\de>0$ such that
\begin{equation}\label{ciagztrans}
d_{Mod}(\omega,\bar\omega)<\de\Rightarrow \sup_{t\in[-1,1]}\la_\omega(\mathcal T^\omega_t Q\ \triangle\ \mathcal S_\omega^{-1}\circ\mathfrak S(\bar{\omega})_t\circ\mathcal S_\omega Q)<\ep.
\end{equation}

Fix $Q\in\mathcal Q$ and $\ep>0$. We now prove $\eqref{ciagztrans}$ for $Q$. Denote by $k$ the number of times the ingoing and outgoing vertical separatrix segments of length $2$ starting from the singular points $\sigma\in\Sigma$ intersect with $Q$. By extending those segments if necessary, we obtain segments $v_j\subset Q$, for $j=1,\ldots,k$, such that the endpoints of $v_j$ lie on a horizontal sides of $Q$.
Let $0<\epsilon_Q<\ep$, and consider the subset $\tilde Q\subseteq Q$ obtained by cutting out from $Q$ all rectangles of which the segments $v_j$ are vertical sides and whose width is $4\epsilon_Q$. Assume that $\epsilon_Q$ is small enough so that
\begin{equation}\label{QdoQ}
\la_{\omega}(\tilde Q)>(1-\ep)\la_\omega(Q).
\end{equation}
Note that $\tilde Q$ is a union of $l\le k+1$ rectangles $D_j$ for $j=1,\ldots,l$. By Lemma \ref{prostokat}, there exists $\gamma>0$ such that for every $j=1,\ldots,l$ and every $\la_\omega$-preserving transformation $F:D_j\to M$ satisfying $\sup_{x\in D_j}d_\omega(x,F(x))<4\gamma$ we have
\begin{equation}\label{niedaleko}
\la_\omega(D_j\cap F(D_j))>(1-\ep)\la_\omega(D_j).
\end{equation}

Take $0<\epsilon<\min\{\gamma,\epsilon_Q\}$.
Since $\bar{\omega} \mapsto \mathcal H_{\bar{\omega}}$ is continuous, we can choose $\de>0$ such that for every $\bar{\omega}\in U_\zeta$
\begin{equation}\label{odl}
d_{Mod}(\omega,\bar{\omega})<\de\Longrightarrow\big(\sup_{x\in M}d_\zeta\big(\mathcal H_{\omega}^{-1}\circ\mathcal H_{\bar\omega}(x),x\big)<\epsilon\quad\wedge\quad\sup_{x\in M}d_\zeta\big(\mathcal H_{\bar\omega}^{-1}\circ\mathcal H_{\omega}(x),x\big)<\epsilon\big).
\end{equation}
Moreover, by applying (iv) from Lemma \ref{techn} for $\epsilon$ and taking smaller $\de$ if necessary, we get that $\mathfrak h_{\omega}^{-1}\circ \mathfrak h_{\bar\omega}:(M,\bar{\omega})\to(M,\omega)$ and $\mathfrak h_{\bar\omega}^{-1}\circ \mathfrak h_{\omega}:(M,{\omega})\to(M,\bar\omega)$ are Lipschitz piecewise affine homeomorphisms with constant $1+\epsilon$. Furthermore, (v) in Lemma \ref{techn} gives us that the set $\tilde M(\omega)$ satisfies  $\la_\omega(\tilde M(\omega))>1-K\epsilon$ and for $x\in \tilde M({\omega})$ we have
\begin{equation}\label{odlpola1}
d_\omega(\mathcal T^\omega_t(x),\mathfrak h_\omega^{-1}\circ\mathfrak h_{\bar\omega}\circ\mathcal T^{\bar\omega}_t\circ\mathfrak h_{\bar\omega}^{-1}\circ\mathfrak h_{\omega}(x))<\epsilon\text{ for any }t\in[-1,1].
\end{equation}
It also implies that, for every $\sigma\in \Sigma$, 
\begin{equation}\label{eq3ep}
d_\omega(\sigma,\mathfrak h_\omega^{-1}\circ\mathfrak h_{\bar\omega}\circ\mathcal T^{\bar\omega}_t\circ\mathfrak h_{\bar\omega}^{-1}\circ\mathfrak h_{\omega}(x))>3\epsilon\text{ for any }t\in[-1,1].
\end{equation}
Since $\epsilon<\epsilon_Q$, we have $\tilde Q\subset\tilde M(\omega)$.

We now estimate the distance between the orbits of the flows $\mathfrak h_{\omega}^{-1}\circ \mathfrak h_{\bar\omega}\circ \mathcal T_t^{\bar\omega} \circ \mathfrak h_{\bar\omega}^{-1}\circ \mathfrak h_{\omega}$ and
\[
\mathcal S_{\omega}^{-1}\circ \mathcal S_{\bar\omega}\circ \mathcal T_t^{\bar\omega} \circ \mathcal S_{\bar\omega}^{-1}\circ \mathcal S_{\omega}=\mathfrak h_\omega^{-1}\circ\mathcal H_{\omega}^{-1}\circ\mathcal H_{\bar\omega}\circ\mathfrak h_{\bar{\omega}}\circ T^{\bar\omega}_t\circ h_{\bar\omega}^{-1}\circ\mathcal H_{\bar\omega}^{-1}\circ\mathcal H_{\omega}\circ h_{\omega}.
\]
By \eqref{odl} we have that
\[
d_\zeta\big(\mathfrak h_{\omega}(x), \mathcal H_{\bar\omega}^{-1}\circ\mathcal H_{\omega}\circ\mathfrak h_{\omega} (x)\big)<\epsilon
\]
for every $x\in M$.
By (i) in Lemma \ref{techn}, $\mathfrak h_{\bar\omega}^{-1}:(M,\zeta)\to(M,\omega)$ is Lipschitz with constant $\frac{11}{10}$. Thus we have
\[
d_{\bar\omega}\big(\mathfrak h_{\bar{\omega}}^{-1}\circ \mathfrak h_{\omega}(x), \mathfrak h_{\bar\omega}^{-1}\circ\mathcal H_{\bar\omega}^{-1}\circ\mathcal H_{\omega}\circ \mathfrak h_{\omega} (x)\big)<\frac{11}{10}\epsilon.
\]
Since $\mathfrak h_{\omega}^{-1}\circ \mathfrak h_{\bar\omega}$ is Lipschitz with constant $1+\epsilon$ and 
fixes $\Sigma$, \eqref{eq3ep} implies that
\[
\min_{\sigma\in\Sigma}\inf_{t\in[-1,1]}d_{\bar{\omega}}(\mathcal T_t^{\bar\omega} \circ \mathfrak h_{\bar\omega}^{-1}\circ \mathfrak h_{\omega}(x),\sigma)>\frac{3\epsilon}{1+\epsilon}>2\epsilon
\]
for every $x\in\tilde M(\omega)$. Hence on the $2\epsilon$-neighbourhood of $\mathfrak h_{\bar\omega}^{-1}\circ \mathfrak h_{\omega}(x)$, $\{\mathcal T^{\bar\omega}_t\}_{t\in[-1,1]}$ acts isometrically. Thus
\[
d_{\bar\omega}\big(\mathcal T^{\bar\omega}_t(\mathfrak h_{\bar\omega}^{-1}\circ \mathfrak h_{\omega}(x)), \mathcal T^{\bar\omega}_t\circ\mathfrak h_{\bar\omega}^{-1}\circ\mathcal H_{\bar\omega}^{-1}\circ\mathcal H_{\omega}\circ \mathfrak h_{\omega} (x)\big)<\frac{11}{10}\epsilon,
\]
for $t\in[-1,1]$ and for every $x\in\tilde M(\omega)$.
Since $\mathfrak h_{\bar\omega}:(M,\bar{\omega})\to(M,\zeta)$ is Lipschitz with constant $\frac{11}{10}$, this implies that
\[
d_{\zeta}\big(\mathfrak h_{\bar{\omega}}\circ \mathcal T^{\bar\omega}_t\circ\mathfrak h_{\bar\omega}^{-1}\circ\mathfrak h_{\omega}(x),\mathfrak h_{\bar{\omega}}\circ \mathcal T^{\bar\omega}_t\circ\mathfrak h_{\bar\omega}^{-1}\circ\mathcal H_{\bar\omega}^{-1}\circ\mathcal H_{\omega}\circ\mathfrak h_{\omega} (x)\big)<\frac{121}{100}\epsilon.
\]
Again by using \eqref{odl} we obtain that
\[
d_{\zeta}\big(\mathfrak h_{\bar{\omega}}\circ \mathcal T^{\bar\omega}_t\circ\mathfrak h_{\bar\omega}^{-1}\circ \mathfrak h_{\omega}(x),\mathcal H_{\omega}^{-1}\circ\mathcal H_{\bar\omega}\circ\mathfrak h_{\bar{\omega}}\circ \mathcal T^{\bar\omega}_t\circ\mathfrak h_{\bar\omega}^{-1}\circ\mathcal H_{\bar\omega}^{-1}\circ\mathcal H_{\omega}\circ \mathfrak h_{\omega} (x)\big)<\frac{221}{100}\epsilon.
\]
Finally, since $\mathfrak h_\omega^{-1}$ is also Lipschitz with constant $\frac{11}{10}$, we obtain that
\[
d_{\omega}\big(\mathfrak h_\omega^{-1}\circ\mathfrak h_{\bar{\omega}}\circ \mathcal T^{\bar\omega}_t\circ h_{\bar\omega}^{-1}\circ h_{\omega}(x),\mathfrak h_\omega^{-1}\circ\mathcal H_{\omega}^{-1}\circ\mathcal H_{\bar\omega}\circ\mathfrak h_{\bar{\omega}}\circ \mathcal T^{\bar\omega}_t\circ h_{\bar\omega}^{-1}\circ\mathcal H_{\bar\omega}^{-1}\circ\mathcal H_{\omega}\circ h_{\omega} (x)\big)<\frac{2431}{1000}\epsilon.
\]
By combining this with \eqref{odlpola1} we obtain that for every $x\in \tilde M(\omega)$ we have
\begin{equation}\label{waznyoszac}
d_{\omega}\big(\mathcal T_t^\omega(x),\mathcal S_{\omega}^{-1}\circ\mathfrak S(\bar{\omega})_t\circ\mathcal S_{\omega}(x)\big)<\frac{3431}{1000}\epsilon<4\epsilon.
\end{equation}
By the definition of $\tilde M(\omega)$, $\{\mathcal T^\omega_t\}_{t\in[-1,1]}$ acts isometrically on 
the $4\epsilon$-neighbourhood of $x\in\tilde M(\omega)$. Hence
\begin{equation}\label{waznyoszac1}
d_{\omega}\big(x,\mathcal T_{-t}^\omega\circ\mathcal S_{\omega}^{-1}\circ\mathfrak S(\bar{\omega})_t\circ\mathcal S_{\omega}(x)\big)<4\epsilon.
\end{equation}

Since $D_j\subseteq\tilde Q\subset \tilde M(\omega)$, \eqref{waznyoszac1} is satisfied for all $x\in D_j$. Consider
\[
F:= \mathcal T_{-t}^\omega\circ\mathcal S_{\omega}^{-1}\circ\mathfrak S(\bar{\omega})_t\circ\mathcal S_{\omega}.
\]
Note that $F$ is $\la_\omega$-measure preserving. Thus, by \eqref{niedaleko}, we get
\[
\la_\omega\big(D_j\cap\mathcal T_{-t}^\omega\circ\mathcal S_{\omega}^{-1}\circ\mathfrak S(\bar{\omega})_t\circ\mathcal S_{\omega}(D_j)\big)>(1-\ep)\la_\omega(D_j).
\]
Together with $\mathcal T^\omega$-invariance of $\la_\omega$, this yields
\[
\la_\omega\big(\mathcal T_t^\omega(D_j)\cap\mathcal S_{\omega}^{-1}\circ\mathfrak S(\bar{\omega})_t\circ\mathcal S_{\omega} (D_j)\big)>(1-\ep)\la_\omega(\mathcal T_t^\omega(D_j)),
\]
for every $t\in[-1,1]$. By summing up over $j=1,\ldots,l$ we get
\[
\la_\omega\big(\mathcal T_t^\omega(\tilde Q)\cap\mathcal S_{\omega}^{-1}\circ\mathfrak S(\bar{\omega})_t\circ\mathcal S_{\omega}(\tilde Q)\big)>(1-\ep)\la_\omega(\mathcal T_t^\omega(\tilde Q))
\]
and by \eqref{QdoQ} this yields
\[
\la_\omega\big(\mathcal T_t^\omega(Q)\cap\mathcal S_{\omega}^{-1}\circ\mathfrak S(\bar{\omega})_t\circ\mathcal S_{\omega} (Q)\big)>(1-2\ep)\la_\omega(\mathcal T_t^\omega(Q)).
\]
Since $\la_\omega(Q)<\frac{1}{4}$, we have
\[
\la_\omega\big(\mathcal T_t^\omega(Q)\triangle\mathcal S_{\omega}^{-1}\circ\mathfrak{S}(\bar{\omega})_t\circ\mathcal S_{\omega} (Q)\big)<4\ep\la_\omega(\mathcal T_t^\omega(Q))\le \ep.
\]
Thus we get~\eqref{ciagztrans}, and this concludes the proof of the  theorem.
\end{proof}

Since the construction above is local, we need to show that this suffices to transport the $G_\de$-condition from the space of flows to the moduli space.

\begin{lm}\label{pokr}
	Let $X$ be a metric topological space. Let $\{U_i\}_{i\in\n}$ be a sequence of open subsets such that $\bigcup_{i\in\n}U_i=X$. If $V\subseteq X$ is such that
	\[
	V\cap U_i\text{ is a }G_\de\text{-set for 
	each }i\in\n,
	\]
	then $V$ is a $G_\de$-set.
\end{lm}
\begin{proof}
	Note that $V=\bigcap_{i\in\n}(V\cap U_i)\cup U_i^c$. Since $X$ 
	is metrizable, every closed set is a $G_\de$-set. To finish the proof it is enough to 
	observe that the union of two $G_\de$-sets is a $G_\de$-set.
\end{proof}

\section{Translation flows disjoint with their inverses are $G_\de$-dense}\label{sec:main}
We have the following result which follows from the proof of Corollary 3.3 in \cite{DaRy}.
\begin{tw}\label{danrhyz}
	Let $(X,\mathcal B(X),\mu)$ be 
	a nonatomic standard Borel probability space, and let $\operatorname{Flow}(X)$ be 
	the space of $\mu$-invariant flows on $X$. The set of flows which are weakly mixing and disjoint with their 
	inverse is $G_\de$-dense in $\operatorname{Flow}(X)$.
	\end{tw}
	The following result allows us to transfer the $G_\de$ condition onto any connected component of the moduli space.
	\begin{prop}\label{propback}
		Let $\mathfrak{Prop}$ be a property 
		of a measure-preserving flow
		such that the set of elements having this property is a $G_\de$ subset of the space $\operatorname{Flow}(X)$. Then in every connected component $C$ of $\mathcal M$, the set of translation structures for which the vertical flow has the property $\mathfrak{Prop}$ is a $G_\de$ set in $C$.
		\end{prop}
		
	\begin{proof}	
		Let $C$ be a connected component of $\mathcal M$. In view of Theorem \ref{wloz} for every $\zeta\in C$ there exists an open neighbourhood $\mathcal U_\zeta$ 
		of $\zeta$ and a continuous mapping $\mathfrak S_\zeta:\mathcal U_\zeta\to 
		\operatorname{Flow}(M,\la_\zeta)$ such that for every $\omega\in\mathcal U_\zeta$ the vertical flow $\mathcal T^\omega$ is measure-theoretically isomorphic to $\mathfrak S_{\zeta}(\omega)$. Since $C$ is a topological manifold, it is $\sigma$-compact. Thus there exists a sequence $\{\zeta_n\}_{n\in\n}$ of translation structures such that $\bigcup_{n\in\n}\mathcal U_{\zeta_n}=C$. For each $n\in\n$ we have that
		\[
		\begin{split}
		\mathcal Y_{\zeta_n}:&=\{\omega\in\mathcal U_{\zeta_n};\ \mathcal T^\omega\ \text{satisfies}\ \mathfrak{Prop}\}\\&=\{\omega\in\mathcal U_{\zeta_n};\ \mathfrak S_{\zeta_n}(\omega)\ \text{satisfies}\ \mathfrak{Prop}\}\\&=\mathfrak S_{\zeta_n}^{-1}\{\mathcal T\in
		\operatorname{Flow}(M,\la_{\zeta_n});\ \mathcal T\ \text{satisfies}\ \mathfrak{Prop}\}
		\end{split}
		\]
		 is a $G_\de$ set in $\mathcal U_{\zeta_n}$. By Lemma \ref{pokr}, this gives that the set of $\omega\in C$ such that $\mathcal T^\omega$ satisfies $\mathfrak{Prop}$ is a $G_\de$ set in $C$.
		\end{proof}
	
	By combining Theorem \ref{danrhyz} and Proposition \ref{propback} we get the following result.
	\begin{wn}\label{gdelta}
		The set of translation structures $\zeta$ such that the vertical flow on $(M,\zeta)$ is weakly mixing and disjoint with its inverse is a $G_\de$ set in every connected component $C$ of the moduli space.
	\end{wn}
	
	Throughout this section we use the following notation. Let $C\subset \mathcal M$ be a non-hyperelliptic connected component of 
	the moduli space, i.e. $C$ is not of the form $\mathcal M^{hyp}(2g-2)$ or $\mathcal M^{hyp}(g-1,g-1)$ for any $g\ge 2$. Let $\pi=(\pi_0,\pi_1)$ be a permutation of 
	the alphabet $\mathcal A$ of $d$ elements from the corresponding extended Rauzy class satisfying
	\[
	\pi_0^{-1}(1)=\pi_1^{-1}(d)\ \text{and}\ \pi_0^{-1}(d)=\pi_1^{-1}(1).
	\]
	This permutation exists due to Theorem \ref{twpierost} and by the choice of $C$, it is not symmetric. Let $\Omega_{\pi}$ be the translation matrix corresponding to $\pi$. By Corollary \ref{nozero} there exist symbols $a_1,a_2\in\mathcal A$ such that  $(\Omega_\pi)_{a_1a_2}=(\Omega_\pi)_{a_2a_1}=0$ and for any rationally independent vector $\tau\in\re^\mathcal A$ the numbers
	\[
	(\Omega_\pi\tau)_{a_2}-(\Omega_\pi\tau)_{a_1}\ \text{ and }\  (\Omega_\pi\tau)_{a_1}-((\Omega_\pi\tau)_{\pi_0^{-1}(1)}+(\Omega_\pi\tau)_{\pi_0^{-1}(d)})
	\]
	are rationally independent.
	The proof of the following lemma goes along the same lines as the proof of Lemma 14 in \cite{Fr}. It is mainly based on the recurrence of polygonal Rauzy-Veech induction.
	\begin{lm}\label{gesty}
		The set
		\[
		C_*:=\{M(\pi,\la,\tau)\in C;\ (\pi,\la,\tau)\in\Theta_\pi,\ \la_a=0\ \text{for}\ a\in\mathcal A\setminus\{\pi^{-1}_0(1),\pi^{-1}_0(d),a_1,a_2\}\}
		\]
		 is dense in $C$.
	\end{lm}
	\noindent Before heading to the proof of the main result we give the proof of Proposition \ref{odwr} which treats the density of translation structures on which the vertical flow is reversible.
	\begin{proof}[Proof of Proposition \ref{odwr}]
		By following again the proof of Lemma 14 in \cite{Fr} we prove that the set
	\[
	C_{**}:=\{M(\pi,\la,\tau)\in C;\ (\pi,\la,\tau)\in\Theta_\pi,\ \la_a=0\ \text{for}\ a\in\mathcal A\setminus\{\pi^{-1}_0(1),\pi^{-1}_0(d)\}\}
	\]
	is dense in $C$. The vertical flow on $M(\pi,\la,\tau)\in C_{**}$ is measure-theoretically isomorphic to the vertical flow on a torus given by $(\la_{\pi^{-1}_0(1)},\la_{\pi^{-1}_0(1)})$ and $\pi$ - the non-trivial permutation of two elements. Since the translation flows on tori are reversible, this concludes the proof.
	\end{proof}
	\noindent The special representations of the vertical flows associated to the translation structures from $C_*$ as given in Lemma \ref{gesty} are special flows over IETs of 3 intervals and under a roof functions which are piecewise constant and have discontinuity points which coincide with the discontinuity points of the IET and one additional discontinuity point inside the middle interval. After one step of either left- or right-hand side Rauzy-Veech induction we get a special representation over rotation and under a piecewise constant roof function with 4 discontinuity points. We now prove some properties of such flows.
	
			Let $\al\in[0,1)$ be an irrational number, and let $T_\al:\re/\z\to\re/\z$ be 
			the rotation by $\al$. For any $\beta\in[0,1)$ let
			\[
			\|\beta\|:=\min\big\{\{\beta\},\{1-\beta\}\big\}.
			\]
			Let $\{q_n\}_{n\in\n}$ be the sequence of partial denominators associated to $\al$. Recall that 
			for every odd $n\in\n$, we have a pair of Rokhlin towers 
\[\{T_{\al}^i[-\|q_n\al\|,0)\}_{i=0,\ldots,q_{n-1}-1}\quad \text{ and }\quad \{T_{\al}^i[0,\|q_{n-1}\al\|)\}_{i=0,\ldots,q_{n}-1}\] 
for $n\ge 1$, which covers $\re/\z$. As a corollary from Theorem 3.9 in \cite{BerkFr} we get the following result (recall the definitions of $f^{(r)}$ and $\operatorname{Leb}^f$ from subsection~\ref{sec:specialflows}, as well as the definition of $\mu_{t,s}$ given in \eqref{defjoin}).
	\begin{prop}\label{granicaistnieje}
		Let $\{(T_\al^f)_t\}_{t\in\re}$ be 
		the special flow over 
		the rotation by $\al\in[0,1)$ 
		under a positive roof function $f\in L^2([0,1),Leb)$. 
		Suppose that there exists a rigidity sequence $\{r_n\}_{n\in\n}$ for $T_\al$ (which is a subsequence of $\{q_n\}_{n\in\n}$) such that,
		setting $b_n:=r_n\int_0^1f(y)dy$, the sequence
		\begin{equation}\label{dkcalki}
		\left\{\int_0^1\Big|f^{(r_n)}(x)-b_n\Big|^2dx\right\},\ n\in\n
		\end{equation}
		is bounded. Then there exists a probability measure $P\in\mathcal P(\re^2)$ such that, up to taking a subsequence,
		\[
		\left(f^{(2r_n)}-2b_n,f^{(r_n)}-b_n\right)_*\operatorname{Leb}\to P \text{ weakly.}
		\]
		Moreover, along the same subsequence, we have
		\[
		\operatorname{Leb}^f_{2b_n,b_n}\to\int_{\re^2}\operatorname{Leb}^f_{-t,-u}dP(t,u).
		\]
		\end{prop}
			
		To prove the next result we need the following remark.
		\begin{uw}\label{wartsum}
		Let $f:[0,1)\to\re$ be a piecewise constant function. 
		Let $\beta_1,\ldots,\beta_k$ be the jumps of $f$ and let $d_1,\ldots,d_k$ be their respective values. Then for every $x\in[0,1)$ and every odd $n\in\n$,
		\[
		\sum_{i=0}^{q_n-1}\big(f(T_\al^{q_n+i}(x))-f(T_\al^{i}(x))\big)=\sum_{i=0}^{q_n-1}\sum_{j=1}^k -d_j\chi_{T_\al^{-i}[\beta_j,\beta_j+\|q_n\al\|)}(x).
		\]
		Indeed, the expression $f(T_\al^{q_n+i}(x))-f(T_\al^{i}(x))$ takes non-zero value if and only if there is a discontinuity point $\beta_j$ 
		in the interval $(T_\al^{q_n+i}(x),T_\al^{i}(x)]$. However, we have 
		$T_\al^{q_n+i}(x)=T_\al^{i}(x)-\|q_n\al\|$. Hence $\beta_j\in (T_\al^{q_n+i}(x),T_\al^{i}(x)]$ if and only if $x\in T_\al^{-i}[\beta_j,\beta_j+\|q_n\al\|)$. In other words, if we consider 
		the (not necessarily disjoint) towers 
		$U_j:=\bigcup_{i=0}^{q_n-1}T^{-i}[\beta_j,\beta_j+\|q_n\al\|)$ for $j=1,\ldots,k$, 
		then
		\[
		\sum_{i=0}^{q_n-1}\big(f(T_\al^{q_n+i}(x))-f(T_\al^{i}(x))\big)=\sum_{j=1}^k-d_j\chi_{U_j}(x).
		\]
		In particular, if $x\notin U_j$ for all $j=1,\ldots,k$, then $\sum_{i=0}^{q_n-1}\big(f(T_\al^{q_n+i}(x))-f(T_\al^{i}(x))\big)=0$. Otherwise, if $x\in U_j$ for some $j=1,\ldots,k$, then $d_j$ contributes to the value of 
		the considered expression.
		\end{uw}
			
			We need the following theorem which is a version of Theorem 7.3 in \cite{BerkFr}.
			Recall that $\xi:\re^2\to\re$ is defined by $\xi(t,u):=t-2u$.
	\begin{lm}\label{specmiary}
		There exists a set $\Lambda\subset\re/\z$ of full Lebesgue measure such that, for every $\al\in\Lambda$, there exists a set $D_\al\in(\re/\z)\times (\re/\z)$ of full Lebesgue measure with the property that, if $(\beta_1,\beta_2)\in D_\al$, then
		\begin{itemize}
		  \item the numbers $0,1-\al,\beta_1$ and $\beta_2$ are distinct;
		  \item  for every piecewise constant positive function $f:\re/\z\to\re$ with discontinuity points $0,1-\al,\beta_1,\beta_2$ and jumps $d_{\beta_1}$ and $d_{\beta_2}$
		at $\beta_1$ and $\beta_2$ respectively, 
		for the special flow $\{(T^f_\al)_t\}_{t\in\re}$ there exist probability measures $P,Q\in\mathcal P(\re^2)$ such that, up to a subsequence, we have
		\begin{equation}\label{zbieznmiar}
			\operatorname{Leb}^f_{2b_n,b_n}\to\int_{\re^2}\operatorname{Leb}^f_{-t,-u}dP(t,u) \text{ and }
			\operatorname{Leb}^f_{-2b_n,-b_n}\to\int_{\re^2}\operatorname{Leb}^f_{-t,-u}dQ(t,u) \text{ weakly,}
		\end{equation}
		where $b_n:=q_n\int_0^1f(x)dx$.
		Moreover, $\xi_*P$ is atomic with exactly 4 atoms in points $0$, $-d_{\beta_1}$, $-d_{\beta_2}$, $d_{\beta_1}+d_{\beta_2}$,
		while $\xi_*Q=(-\xi)_*P$ 
		has exactly 4 atoms in points $0$, $d_{\beta_1}$, $d_{\beta_2}$, $-(d_{\beta_1}+d_{\beta_2})$.
		\end{itemize}

		\end{lm}
		\begin{proof}
		Let $\Lambda\subset\re/\z$ be the set of irrational $\al\in\re/\z$ such that there exists a sequence $\{k_n\}_{n\in\n}$ of odd numbers such that, for some $\frac{1}{52}\leq\ep\le\frac{1}{25}$, we have $\lim_{n\to\infty}q_{k_n}\|q_{k_n}\al\|=\ep$. The set $\Lambda$ is of full Lebesgue measure. Indeed, the Gauss map $G(x)=\{\frac{1}{x}\}$ is mixing for the absolutely continuous measure with 
		density $\frac{1}{\ln 2}\frac{1}{1+x}$, hence in particular $G^2$ is ergodic.
		For any irrational $\al\in\re/\z$, let $\{a_n\}_{n\in\n}$ be the sequence of partial quotients of $\al$. Then we have 
		(see \cite{Khin})
		\begin{equation}\label{anqn}
		\frac{1}{2}\frac{1}{a_{n+1}+1}<q_n\|q_n\al\|<\frac{1}{a_{n+1}}.
		\end{equation}
		Recall that for any $m\in\n$, $G^n(\al)\in(\frac{1}{m+1},\frac{1}{m}]$ iff $a_n=m$. Hence by 
		ergodicity of $G^2$, for almost every $\al\in[0,1)$, $a_{n+1}=25$ 
		for infinitely many odd numbers $n$. 
		Thus we obtain the claim.
		
		Fix $\al\in\Lambda$. Recall that for every $n\in\n$, $\re/\z$ is covered by a pair of towers
		$$ \{T_{\al}^i[-\|q_{k_n}\al\|,0)\}_{i=0,\ldots,q_{k_n-1}-1}\quad \text{and}\quad \{T_{\al}^i[0,\|q_{k_n-1}\al\|)\}_{i=0,\ldots,q_{k_n}-1}.$$
		Note that $1-\al\in T_\al^{q_{k_n}-1}[0,\|q_{k_n-1}\al\|)$. More precisely, $1-\al=T_\al^{q_{k_n}-1}(0)+\|q_{k_n}\al\|$. Since $\{q_{k_n}\}_{n\in\n}$ is a rigidity sequence and $f$, regardless of the choice of discontinuity points, is of bounded variation, \eqref{zbieznmiar} follows from Proposition \ref{granicaistnieje} and Koksma-Denjoy inequality:
		\[
		|f^{(q_{k_n})}(x)-b_n|\le Var(f)\ \text{ for every }x\in[0,1).
		\]
		
		In view of Proposition \ref{granicaistnieje}, taking a subsequence if necessary, we get
		\[
		P=\lim_{n\to\infty}\big(f^{(2q_{k_n})}-2b_n,f^{(q_{k_n})}-b_n\big)_*Leb
		\]
		and
		\[
		Q=\lim_{n\to\infty}\big(f^{(-2q_{k_n})}+2b_n,f^{(-q_{k_n})}+b_n\big)_*Leb.
		\]
		By applying $\xi_*$ to both expressions we obtain
		\[
		\xi_*P=\lim_{n\to\infty}\left(\sum_{i=0}^{q_{k_n}-1}(f\circ T_\al^{q_{k_n}+i}-f\circ T_\al^i)\right)_{\!\!*}Leb
		\]
		and
		\[
		\xi_*Q=\lim_{n\to\infty}\left(\sum_{i=1}^{q_{k_n}}(f\circ T_\al^{-q_{k_n}-i}-f\circ T_\al^{-i})\right)_{\!\!*}Leb.
		\]		
		By using the invariance of 
		the Lebesgue measure under $T_\al^{2q_{k_n}}$, we get 
		$\xi_*Q=(-\xi)_*P$.
		
		Consider the sequence of pairs of disjoint Rokhlin towers
		\[ V_n:=\bigcup_{i=0}^{q_{k_n}-1}{T^i_\al[3\|q_{k_n}\al\|,\frac{1}{3}\|q_{k_n-1}\al\|)}\ \text{ and }\  W_n:=\bigcup_{i=0}^{q_{k_n}-1}{T^i_\al[\frac{2}{3}\|q_{k_n-1}\al\|,\|q_{k_n-1}\al\|-3\|q_{k_n}\al\||)}.
		\]
		We know (see e.g.~\cite{Khin}) that
		\[
		\frac{1}{2q_{k_n}}<\|q_{k_n-1}\al\|<\frac{1}{q_{k_n}}.
		\]
		Hence
		\[
		\frac{\|q_{k_n-1}\al\|}{\|q_{k_n}\al\|}<\frac{1}{q_{k_n}\|q_{k_n}\al\|}\to\frac{1}{\ep}\
		\text{ and }\
		\frac{\|q_{k_n-1}\al\|}{\|q_{k_n}\al\|}>\frac{1}{2q_{k_n}\|q_{k_n}\al\|}\to\frac{1}{2\ep}.
		\]
		By using the fact that $\ep<\frac{1}{25}$, for sufficiently large $n$ we obtain that
		\begin{equation}\label{szacqn}
			12<\frac{\|q_{k_n-1}\al\|}{\|q_{k_n}\al\|}<53.
			\end{equation}
			In view of \eqref{szacqn} we get
		\[
		\operatorname{Leb}(W_n)=\operatorname{Leb}(V_n)=q_{k_n}(\frac{1}{3}\|q_{k_n-1}\al\|-3\|q_{k_n}\al\|)>q_{k_n}\|q_{k_n}\al\|>\frac{1}{53},
		\]
		for sufficiently large $n\in\n$, that is
		the measures of $V_n$, $W_n$ are bounded away from $0$.

		By the remark to Lemma 3.4 in \cite{King}, this implies that, for almost every $\beta_1\in[0,1)$, there exists an infinite set $N_1$ of natural numbers  such that $\beta_1\in V_{n}$ for each $n\in N_1$.
		Once such $\beta_1$ and $N_1$ are fixed, the same argument yields that for almost every $\beta_2\in[0,1)$ there exists an infinite subset $N_2\subset N_1$  such that $\beta_2\in W_{n}$ for each $n\in N_2$.
		Using Fubini's theorem, we get that for almost every $(\beta_1,\beta_2)\in[0,1)\times[0,1)$ there exist infinitely many integers $n$ such that $\beta_1\in V_n$ and $\beta_2\in W_n$.  Let $D_\al$ be the set of such pairs $(\beta_1,\beta_2)$.

		Now we fix $(\beta_1,\beta_2)$ in $D_\al$. Extracting a subsequence if necessary, we may assume that $\beta_1\in V_n$ and $\beta_2\in W_n$ for all $n$.
		Since $V_n$ and $W_n$ are disjoint and $0,1-\al\notin W_n\cup V_n$ for all $n\in\n$, the points $0$, $1-\al$, $\beta_1$ and $\beta_2$ are distinct. 
		Note that, by the choice of $W_n$ and $V_n$ 
		the towers $\bigcup_{i=0}^{q_{k_n}-1}T^{-i}[\beta_1,\beta_1+\|q_{k_n}\al\|)$, $\bigcup_{i=0}^{q_{k_n}-1}T^{-i}[\beta_2,\beta_2+\|q_{k_n}\al\|)$ and $\bigcup_{i=0}^{q_{k_n}}T^{-n}[0,\|q_{k_n}\al\|)$ are pairwise disjoint.
		Indeed, we have the following inclusions
		\[
		\bigcup_{i=0}^{q_{k_n}-1}T^{-i}[\beta_1,\beta_1+\|q_{k_n}\al\|)\subset\bigcup_{i=0}^{q_{k_n}-1}{T^i_\al[3\|q_{k_n}\al\|,\frac{1}{3}\|q_{k_n-1}\al\|+\|q_{k_n}\al\|)},
		\]
		\[
		\bigcup_{i=0}^{q_{k_n}-1}T^{-i}[\beta_2,\beta_2+\|q_{k_n}\al\|)\subset\bigcup_{i=0}^{q_{k_n}-1}{T^i_\al[\frac{2}{3}\|q_{k_n-1}\al\|,\|q_{k_n-1}\al\|-2\|q_{k_n}\al\|)},
		\]		
		and
		\[
		\bigcup_{i=0}^{q_{k_n}}T^{-i}[0,\|q_{k_n}\al\|)\subset\bigcup_{i=0}^{q_{k_n}-1}{T^i_\al[0,2\|q_{k_n}\al\|)}.
		\]
		In view of \eqref{szacqn}, we get  $\frac{1}{3}\|q_{k_n-1}\al\|>4\|q_{k_n}\al\|$ for sufficiently large $n$. In particular
		\[
		\frac{1}{3}\|q_{k_n-1}\al\|+\|q_{k_n}\al\|<\frac{2}{3}\|q_{k_n-1}\al\|,
		\]
		which 
		shows that the intervals 
\[\big[0,2\|q_{k_n}\al\|\big),\quad  \Big[3\|q_{k_n}\al\|,\frac{1}{3}\|q_{k_n-1}\al\|+\|q_{k_n}\al\|\Big)\text{ and } \Big[\frac{2}{3}\|q_{k_n-1}\al\|,\|q_{k_n-1}\al\|-2\|q_{k_n}\al\|\Big)\]
are pairwise disjoint. This implies the desired disjointness of the aforementioned towers. This allow us to control the atoms of limit measures and their respective masses.
		
		Suppose now that $f$ has discontinuity points at $0$, $1-\al$, $\beta_1$ and $\beta_2$, where $(\beta_1,\beta_2)\in D_\al$ with jumps $d_0,d_{1-\al},d_{\beta_1}$ and $d_{\beta_2}$ respectively. 
		Since $k_n$ is odd,
		in view of Remark \ref{wartsum} 
		the expression $\sum_{i=0}^{q_{k_n}-1}\big(f\circ T_\al^{q_{k_n}+i}(x)-f\circ T_\al^i(x)\big)$ may only take non-zero values for $x$ on towers $U_1:=\bigcup_{i=0}^{q_{k_n}-1}T^{-i}[0,\|q_{k_n}\al\|)$, $U_2:=\bigcup_{i=0}^{q_{k_n}-1}T^{-i}[1-\al,1-\al+\|q_{k_n}\al\|)$, $U_3:=\bigcup_{i=0}^{q_{k_n}-1}T^{-i}[\beta_1,\beta_1+\|q_{k_n}\al\|)$ and $U_4:=\bigcup_{i=0}^{q_{k_n}-1}T^{-i}[\beta_2,\beta_2+\|q_{k_n}\al\|)$. We have proved though, that $U_3$ is disjoint with other towers and the same is true for $U_4$. Thus we get that
		\[
		\sum_{i=0}^{q_{k_n}-1}(f\circ T_\al^{q_{k_n}+i}(x)-f\circ T_\al^i(x))=-d_{\beta_1}\ \text{ for }\ x\in U_3,
		\]
		and
		\[
		\sum_{i=0}^{q_{k_n}-1}(f\circ T_\al^{q_{k_n}+i}(x)-f\circ T_\al^i(x))=-d_{\beta_1}\ \text{ for }\ x\in U_4.
		\]
		On the other hand,
		\[
		U_1\cap U_2=\bigcup_{i=1}^{q_{k_n}-1}T^{-i}[0,\|q_{k_n}\al\|,\ U_1\setminus U_2=[0,\|q_{k_n}\al\|),\  \text{and}\ U_2\setminus U_1=[\|q_{k_n}\al\|,2\|q_{k_n}\al\|).
		\]
		Hence
		\[
		\sum_{i=0}^{q_{k_n}-1}(f\circ T_\al^{q_{k_n}+i}(x)-f\circ T_\al^i(x))=-d_{0}-d_{1-\al}\ \text{ for }\ x\in U_1\cap U_2,
		\]
		\[
		\sum_{i=0}^{q_{k_n}-1}(f\circ T_\al^{q_{k_n}+i}(x)-f\circ T_\al^i(x))=-d_{0}\ \text{ for }\ x\in[0,\|q_{k_n}\al\|),
		\]
		and
		\[
		\sum_{i=0}^{q_{k_n}-1}(f\circ T_\al^{q_{k_n}+i}(x)-f\circ T_\al^i(x))=-d_{1-\al}\ \text{ for }\ x\in[\|q_{k_n}\al\|,2\|q_{k_n}\al\|).
		\]
		Finally
		\[
		\sum_{i=0}^{q_{k_n}-1}(f\circ T_\al^{q_{k_n}+i}(x)-f\circ T_\al^i(x))=0\ \text{ for }\ x\notin U_1\cup U_2\cup U_3\cup U_4.
		\]
		Note that $-d_0-d_{1-\al}=d_{\beta_1}+d_{\beta_2}$, since the sum of jumps of $f$ is $0$. Moreover
		\[
		\operatorname{Leb}\big([0,\|q_{k_n}\al\|)\big)=\operatorname{Leb}\big([\|q_{k_n}\al\|,2\|q_{k_n}\al\|)\big)=\|q_{k_n}\al\|\to 0,
		\]
		\[
		\operatorname{Leb}(U_3)=\operatorname{Leb}(U_4)=q_{k_n}\|q_{k_n}\al\|\to \ep
		\]
		and
		\[
		\operatorname{Leb}(U_1\cap U_2)=q_{k_n}\|q_{k_n}\al\|-\|q_{k_n}\al\|\to \ep.
		\]
		Hence we get that $\xi_*P=\lim_{n\to\infty}(\sum_{i=0}^{q_{k_n}-1}f\circ T_\al^{q_{k_n}+i}-f\circ T_\al^i)_*Leb$ is a measure such that $\xi_*P(\{-d_{\beta_1}\})=\xi_*P(\{-d_{\beta_2}\})=\xi_*P(\{d_{\beta_1}+d_{\beta_2}\})=\ep$ and $\xi_*P(\{0\})=1-3\ep$. Since $\xi_*Q=(-\xi)_*P$, we obtain the final claim.
			\end{proof}
		We now state 
		a reformulation of the above lemma for rotations on arbitrary large circles. If needed, throughout the proof of this lemma, we identify $\re/x\z$ with $[0,x)$ for every $x\in\re_{>0}$.
		\begin{lm}\label{specmiary1}
			There exists a subset $\Delta_0\subset\Delta:=\{(x,y)\in\re_{>0}^2;0< y<x\}$ of full Lebesgue measure in $\Delta$
            with the property that for every $(l,\al)\in\Delta_0$ there exists a set $D_{l,\al}\subset(\re/l\z)\times(\re/l\z)$ of full Lebesgue measure such that for every $(\beta_1,\beta_2)\in D_{l,\al}$ we have
			\begin{itemize}
			\item the numbers $0$, $l-\al$, $\beta_1$ and $\beta_2$ are distinct in $\re/l\z$ and
			\item if $T_\al$ is the rotation on $\re/l\z$ by $\al\in\re/l\z$ and $h:\re/l\z\to\re_{>0}$ is a piecewise constant function with  exactly $4$ discontinuity points at $0,l-\al,\beta_1,\beta_2$ and rationally independent jumps at $\beta_1$ and $\beta_2$, then the special flow $T_\al^h$ is weakly mixing and disjoint with its inverse.
			\end{itemize}
			\end{lm}
			\begin{proof}
			Let $\Lambda\subset[0,1)$ and $D_\al\subset[0,1)\times[0,1)$ for $\al\in\Lambda$ be the sets given by Lemma \ref{specmiary}.
			For any $l\in\re_{>0}$, 
			let us also denote by $l:[0,1)\to[0,l)$ the map given by $l(x):=lx$.
			We also consider $l$ as a map between $\re/\z$ and $\re/l\z$. For any $l\in\re_{>0}$ let $\Lambda_l:=l(\Lambda)\subset [0,l)$ and set $D_{l,\al}:=(l\times l)(D_{l^{-1}\al})\subset[0,l)\times[0,1)$ for any $\al\in[0,l)$.
			Define $\Delta_0:=\{(x,y);\ x\in\re_{>0},\,y\in\Lambda_x\}$. Note that $\Delta_0$ is of full Lebesgue measure in $\Delta$ and for every $l\in\re_{>0}$ and $\al\in\Lambda_l$ the set $D_{l,\al}$ is of full Lebesgue measure in $(\re/l\z)\times(\re/l\z)$.
			
			Take $(l,\al)\in\Delta_0$ and $(\beta_1,\beta_2)\in D_{l,\al}$. By the definition of $\Lambda$ and $D_{l^{-1}\al}$, the points $0$, $l-\al$, $\beta_1$ and $\beta_2$ are distinct. Let $h\colon\re/l\z\to\re_{>0}$ be a piecewise constant function which has exactly 4 discontinuity points at $0$, $l-\al$, $\beta_1$ and $\beta_2$.
			Assume that the jumps $d_{\beta_1}$ and $d_{\beta_2}$ at $\beta_1$ and $\beta_2$ are rationally independent.
			Consider the special flow $T_{\al}^h$ on $[0,l)^h$. The map $(l^{-1}\times Id):[0,l)^h\to[0,1)^{h\circ l}$ establishes  an isomorphism of flows $T_{\al}^h$ and $T_{l^{-1}\al}^{h\circ l}$.
			The roof function $h\circ l$ has discontinuities at $0,1-l^{-1}\al,l^{-1}\beta_1$ and $l^{-1}\beta_2$ and has jumps $d_{\beta_1}$ and $d_{\beta_2}$ at $l^{-1}\beta_1$ and $l^{-1}\beta_2$ respectively.
			Moreover, $l^{-1}\al\in\Lambda$ and $(l^{-1}\beta_1,l^{-1}\beta_2)\in D_\al$. In view of Lemma \ref{specmiary}, this gives
			\[				
			\operatorname{Leb}^{h\circ l}_{2b_n,b_n}\to\int_{\re^2}\operatorname{Leb}^{h\circ l}_{-t,-u}dP(t,u),\ \text{ and }\
			\operatorname{Leb}^{h\circ l}_{-2b_n,-b_n}\to\int_{\re^2}\operatorname{Leb}^{h\circ l}_{-t,-u}dQ(t,u)\ \text{ weakly,}
			\]
			for some increasing to infinity real sequence $\{b_n\}_{n\in\n}$ and measures $P,Q\in\mathcal P(\re^2)$. Furthermore, $\xi_*P$ is atomic and has atoms at $0$, $-d_{\beta_1}$, $-d_{\beta_2}$ and $d_{\beta_1}+d_{\beta_1}$, while $\xi_*Q$ is also atomic and has atoms at $0$, $d_{\beta_1}$, $d_{\beta_2}$ and $-(d_{\beta_1}+d_{\beta_1})$. Since $d_{\beta_1}$ and $d_{\beta_2}$ are rationally independent, Proposition \ref{weakmix} implies that $T_{l^{-1}\al}^{h\circ l}$ is weakly mixing. Moreover, the rational independence of $d_{\beta_1}$ and $d_{\beta_2}$ also gives $\xi_*P\neq\xi_*Q$ which yields $P\neq Q$. In view of Corollary \ref{kryterium1} this gives that $T_{l^{-1}\al}^{h\circ l}$ is disjoint with its inverse. Since $T_{l^{-1}\al}^{h\circ l}$ and $T_{\al}^h$ are isomorphic, $T_{\al}^h$ is weakly mixing and disjoint with its inverse.
				\end{proof}

We are now ready to give the proof of the main result of this paper.
\begin{proof}[Proof of Theorem \ref{main}.]
In view of Corollary \ref{gdelta}, the set of translation structures whose associated vertical flow is weakly mixing and disjoint with its inverse is a $G_\de$ set in every connected component of the moduli space. We now show that there is a dense subset of translation structures in each non-hyperelliptic connected component $C$ so that the associated vertical flows are weakly mixing and disjoint with their inverses.

	Fix a non-hyperelliptic connected component $C$ of the moduli space. Recall that for some $d\ge 2$ and an alphabet $\mathcal A$ of $d$ elements, there is a permutation $\pi=(\pi_0,\pi_1)\in S_0^\mathcal A$ in the extended Rauzy class associated with $C$, such that
	\[
	\pi_1(\pi_0^{-1}(1))=d\quad \text{and}\quad \pi_1(\pi_0^{-1}(d))=1.
	\]
	Let $\Omega:=\Omega_\pi$ be the translation matrix corresponding to $\pi$. Then, in view of Corollary \ref{nozero}, there exist symbols $a_1,a_2\in\mathcal A$ such that
	\[
	\Omega_{a_1a_2}=\Omega_{a_2a_1}=0
	\]
	and the numbers
	\[
	(\Omega\tau)_{a_2}-(\Omega\tau)_{a_1}\ \text{ and }\  (\Omega\tau)_{a_i}-((\Omega\tau)_{\pi_0^{-1}(1)}+(\Omega\tau)_{\pi_0^{-1}(d)})
	\]
	are rationally independent for $i=1,2$ whenever $\tau$ is rationally independent.

Let
		\[
		\Xi_*:=\big\{(\pi,\la,\tau)\in \Theta_\pi;\la_a=0\ \text{for}\ a\in\mathcal A\setminus\{\pi^{-1}_0(1),\pi^{-1}_0(d),a_1,a_2\}\big\}
		\]
Let $C_*:=\{M(\pi,\la,\tau)\in C;\ (\pi,\la,\tau)\in\Xi_*\}$. In view of Lemma \ref{gesty}, this is a dense subset of $C$. Hence to prove the density of the desired property in $C$, it is enough to prove that this property holds for a dense set in $C_*$.
We prove 
this by finding a dense subset of parameters in $\Xi_*$ such that the associated translation structures have the sought properties.

Note that the set $\Xi\subset\Xi_*$ given by
\[
\Xi:=\{(\pi,\la,\tau)\in\Xi_*;\ T_{\pi,\la}\ \text{is ergodic};\ \la_{\pi_0^{-1}(1)}\neq\la_{\pi_0^{-1}(d)};\  \tau\ \text{is rationally independent}\}
\]
is dense in $\Xi_*$.
Let $\zeta=M(\pi,\la,\tau)\in C_*$ with $(\pi,\la,\tau)\in\Xi$.
Let $\mathcal T^\zeta$ be the corresponding vertical flow. Recall that it has a special representation $T_{\pi,\la}^h$ over the IET $T_{\pi,\la}:[0,|\la|)\to[0,|\la|)$ and under a piecewise constant roof function $h\colon[0,|\la|)\to\re_{>0}$ which is constant over exchanged intervals. Moreover, if we consider $h=\{h_a\}_{a\in\mathcal A}$ as a vector of values, where $h_a$ is the value of $h$ over the interval corresponding to $a$, then $h_a=-(\Omega\tau)_a$. However, since $(\pi,\la,\tau)\in\Xi_*$, we have that $\la_a=0$ for $a\in\mathcal A\setminus\{\pi^{-1}_0(1),\pi^{-1}_0(d),a_1,a_2\}$. Thus we can reduce the data describing the above special representation.

Let $\hat{\pi}=(\hat\pi_0,\hat\pi_1)$ be 
the permutation on the alphabet $\hat{\mathcal A}:=\{\pi^{-1}_0(1),\pi^{-1}_0(d),a_1,a_2\}$ given by
\[
\hat\pi_0(\pi^{-1}_0(1))=1,\quad \hat\pi_0(\pi^{-1}_0(d))=4,\quad \hat\pi_0(a_1)=2,\quad \hat\pi_0(a_2)=3
\]
and
\[
\hat\pi_1(\pi^{-1}_0(1))=4,\quad \hat\pi_1(\pi^{-1}_0(d))=1,\quad \hat\pi_1(a_1)=2,\quad \hat\pi_1(a_2)=3.
\]
For $a\in\hat{\mathcal A}$ let $\hat{\la}_{a}:=\la_a$. Moreover, since the intervals corresponding to $a\in\mathcal A\setminus\hat{\mathcal A}$ are empty, $h$ can be considered as a vector  $\{h_a\}_{a\in\hat{\mathcal A}}$. Then $\mathcal T^\zeta$ has a special representation $T_{\hat{\pi},\hat{\la}}^h$ over the IET $T_{\hat{\pi},\hat{\la}}:[0,|\la|)\to[0,|\la|)$.

Consider the sets $\Xi_0,\Xi_1\subset\Xi$ given by
\[
\Xi_0:=\{(\pi,\la,\tau)\in\Xi;\ \la_{\pi^{-1}_0(1)}>\la_{\pi^{-1}_0(d)}\}\ \text{and}\ \Xi_1:=\{(\pi,\la,\tau)\in\Xi;\ \la_{\pi^{-1}_0(1)}<\la_{\pi^{-1}_0(d)}\}.
\]
We have $\Xi_0\cup\Xi_1=\Xi$. Suppose first that $(\pi,\la,\tau)\in\Xi_0$ that is $\la_{\pi^{-1}_0(1)}>\la_{\pi^{-1}_0(d)}$. Let $\phi:\{(x,y,z,v)\in\re_{>0}^4;\ x>v\}\to\re^4_{>0}$ be 
the diffeomorphism given by
\[
\phi(x,y,z,v):=(x-v,v,y,z).
\]
Then after one step of the polygonal right hand side Rauzy-Veech induction on $\mathcal T_{\hat{\pi},\hat{\la}}^h$ we get a special
flow $T_\al^{\hat{h}}$ over the rotation $T_\al:[0,\hat\la_{\pi^{-1}_0(1)}+\hat\la_{a_1}+\hat\la_{a_2})\to[0,\hat\la_{\pi^{-1}_0(1)}+\hat\la_{a_1}+\hat\la_{a_2})$ by $\al=\al(\hat \lambda):=\hat\la_{a_1}+\hat\la_{a_2}+\hat\la_{\pi^{-1}_0(d)}$
under a piecewise constant function $\hat h\colon[0,\hat\la_{\pi^{-1}_0(1)}+\hat\la_{a_1}+\hat\la_{a_2})\to\re_{>0}$ with values $h_{\pi^{-1}_0(1)},h_{\pi^{-1}_0(1)}+h_{\pi^{-1}_0(d)},h_{a_1},h_{a_2}$ over the consecutive intervals of lengths given by the vector $\phi(\hat\la_{\pi^{-1}_0(1)},\hat\la_{a_1},\hat\la_{a_2},\hat\la_{\pi^{-1}_0(d)})$.
Recall that the flows $\mathcal T_{\hat{\pi},\hat{\la}}^h$ and $T_\al^{\hat{h}}$ are isomorphic.
Let $l=l(\hat\lambda):=\hat\la_{\pi^{-1}_0(1)}+\hat\la_{a_1}+\hat\la_{a_2}$, $\beta_1=\beta_1(\hat\lambda):=\hat\la_{\pi^{-1}_0(1)}$ and  $\beta_2=\beta_2(\hat\lambda):=\hat\la_{\pi^{-1}_0(1)}+\hat\la_{a_1}$.
Then $\hat h:[0,l)\to\re_{>0}$ has discontinuities at points $l-\al$, $\beta_1$ and $\beta_2$.
The jump at the point $\beta_1$ is equal to $h_{a_1}-(h_{\pi^{-1}_0(1)}+h_{\pi^{-1}_0(d)})$, while at the point $\beta_2$ equals $h_{a_2}-h_{a_1}$. Moreover, we have
\[
h_{a_1}-(h_{\pi^{-1}_0(1)}+h_{\pi^{-1}_0(d)})=-(\Omega\tau)_{a_1}+((\Omega\tau)_{\pi_0^{-1}(1)}+(\Omega\tau)_{\pi_0^{-1}(d)}),
\]
and
\[
h_{a_2}-h_{a_1}=-(\Omega\tau)_{a_2}+(\Omega\tau)_{a_2}.
\]
Since $\tau$ is a rationally independent vector, Corollary \ref{nozero} yields the rational independence of the jumps at $\beta_1$ and $\beta_2$. From now on we treat $T_\al$ as a rotation on $\re/l\z$. Furthermore, we also treat $\hat h$ as a piecewise constant function on $\re/l\z$. Then $\hat h\colon\re/l\z\to\re_{>0}$ gets an additional discontinuity point at $0$.

Let us consider 
the diffeomorphism  $\psi:\re^4_{>0}\to\{(x,y,z,v)\in\re_{>0}^4;\ 0<x-y<z<v<x\}$ given by
\[
\psi(x,y,z,v):=(x+y+z+v,y+z+v,x+y,x+y+z).
\]
Then
\[
\psi\circ\phi:\{(x,y,z,v)\in\re^4_{>0};\ x>v\}\to\{(x,y,z,v)\in\re_{>0}^4;\ 0<x-y<z<v<x\}
\]
is a diffeomorphism and  $\psi\circ\phi\,(\hat\la)=(l,\al,\beta_1,\beta_2)$.

Let $\Delta_0\subset\{(x,y)\in\re_{>0}^2;y\in(0,x)\}$ and $D_{l,\al}\subset(\re/l\z)\times(\re/l\z)$ for $(l,\al)\in\Delta_0$ be sets given by Lemma \ref{specmiary1}.
Then, by Lemma \ref{specmiary1}, for every $(l,\al)\in\Delta_0$ and $(\beta_1,\beta_2)\in D_{l,\al}$, the special flow $T_\al^{\hat h}$
over 
the rotation by $\al$ on $\re/l\z$ 
and under a piecewise constant roof function with discontinuity points $0,l-\al,\beta_1,\beta_2$ and 
with rationally independent jumps at $\beta_1$ and $\beta_2$ is weakly mixing and disjoint with its inverse. Consider
\[
\mathscr{G}:=\{(x,y,z,v)\in\re_{>0}^4;\ (x,y)\in\Delta_0,\,\frac{y}{x}\in\re\setminus\q,\,(z,v)\in D_{x,y},\, 0<x-y<z<v<x\}.
\]
In view of Lemma \ref{specmiary1}, $\Delta_0$ is dense in $\{(x,y)\in\re_{>0}; y<x \}$ and $D_{x,y}$ is dense in $(0,x)\times(0,x)$. Therefore $\mathscr{G}$ is a dense set in $\{(x,y,z,v)\in\re_{>0}^4;\ 0<x-y<z<v<x\}$. As $\psi\circ\phi$ is a diffeomorphism, the set $(\psi\circ\phi)^{-1}(\mathscr{G})$ is dense in $\{(x,y,z,v)\in\re_{>0}^4;\ x>v\}$. Hence the set
 \[
 \Gamma_0:=\{(\pi,\la,\tau)\in\Xi_0;\ \hat\la\in(\psi\circ\phi)^{-1}\,(\mathscr{G})\;\text{ and $\tau$ is a rationally independent vector}\}
 \]
 is dense in $\Xi_0$. By going along the same lines and by using the left-hand side polygonal Rauzy-Veech induction, we find a dense set $\Gamma_1\subset\Xi_1$ which has analogous properties.

 If $\g=(\pi,\la,\tau)\in\Gamma_0\cup\Gamma_1$ then $\big(l(\hat\la),\al(\hat\la)\big)\in\Delta_0$, $\big(\beta_1(\hat\la),\beta_2(\hat\la)\big)\in D_{l(\hat\la),\al(\hat\la)}$ and the vertical flow on $M(\g)$ is isomorphic to a special flow $T_{\al(\hat\la)}^{\hat h}$ on $(\re/l(\hat\la)\z)^{\hat h}$, where $\hat h:\re/l(\hat\la)\z\to\re_{>0}$ is a piecewise constant roof function with discontinuities at $0$, $l(\hat\la)-\al(\hat\la)$, $\beta_1(\hat\la)$, $\beta_2(\hat\la)$ and the jumps at $\beta_1(\hat\la)$ and $\beta_2(\hat\la)$ are rationally independent.
 In view of Lemma \ref{specmiary1}, those flows are weakly mixing and disjoint with their inverses.
 As $\Gamma_0\cup\Gamma_1$ is dense in $\Xi$,  it is also dense in $\Xi_*$.
 Since $M:\Theta_\pi\to C$ given by $(\pi,\la,\tau)\mapsto M(\pi,\la,\tau)$ is continuous and $M(\Xi_*)=C_*$, we have that $M(\Gamma_0\cup\Gamma_1)$ is dense in $C_*$. Moreover, by Lemma~\ref{gesty}, $C_*$ is dense in $C$ which yields the result.

\end{proof}
\section*{Acknowledgments}
The authors would like to thank M. Lemańczyk for fruitful discussions and for proposing the main ideas used in section 3. We would also like to thank S. Gouezel for pointing out the article \cite{GP} and for giving some ideas used in section 6. P. Berk and K. Frączek are partially supported by NCN grant nr  2014/13/B/ST1/03153.


\begin{thebibliography}{9}

 \bibitem{Ambr}W.~Ambrose,
 \emph{Representation of ergodic flows},
 Ann.~of~Math.\ (2) 42 (1941), 723–739.
	
 \bibitem{Anz}H.~Anzai,
 \emph{On an example of a measure preserving transformation which is not conjugate to its inverse},
 Proc.\ Japan Acad.\ 27 (1951), 517-522.	
	
 \bibitem{AF} A.~Avila, G.~Forni,
 \emph{Weak mixing for interval exchange transformations and translation flows},
 Ann.\ of\ Math.\ (2)\ 165\ (2007), 637–664.
	
 \bibitem{BerkFr}P.\ Berk, K.\ Fr\k{a}czek,
 \emph{On special flows that are not isomorphic to their inverses},
 Discrete\ Contin.\ Dyn.\ Syst.\ 35 (2015),  829–855.
	
 \bibitem{DaRy} A.I.\ Danilenko, V.V.\ Ryzhikov,
 \emph{On self-similarities of ergodic flows},
 Proc.\ Lond.\ Math.\ Soc.\ 104 (2012), 431-454.
	
 \bibitem{dJun} A.\ del Junco,
 \emph{Disjointness of measure-preserving transformations, minimal self-joinings and category},
 Ergodic theory and dynamical systems, I (College Park, Md., 1979-80), pp.\ 81-89,	
 Progr.\ Math., 10, Birkhäuser, Boston, Mass., 1981.
	
 \bibitem{Fox} R.H.~Fox, R.B.~Kershner,
 \emph{Concerning the transitive properties of geodesics on a rational polyhedron},
  Duke Math.\ J.\ 2 (1936), 147–150.
	
 \bibitem{Fr}K.\ Fr\k{a}czek,
 \emph{Density of mild mixing property for vertical flows of Abelian differentials},
 Proc.\ Amer.\ Math.\ Soc.\ 137 (2009), 4129-4142.
	
 \bibitem{FrKuLem} K.~Fr\k{a}czek, J.~Ku\l aga-Przymus, M.~Lema\'nczyk,
 \emph{Non-reversibility and self-joinings of higher orders for ergodic flows},
 J.\ Anal.\ Math.\ 122 (2014), 163-227.
	
 \bibitem{Glas} E.\ Glasner,
 \emph{Ergodic theory via joinings},
 Mathematical Surveys and Monographs, 101.	American Mathematical Society, Providence, RI, 2003.
	
 \bibitem{GP}  C.\ Goffman, G.\ Pedrick,
 \emph{A proof of the homeomorphism of Lebesgue-Stieltjes measure with Lebesgue measure},
 Proc.\ Amer.\ Math.\ Soc.\ 52 (1975), 196–198.
%

 \bibitem{Halmos} P.R.\ Halmos,
 \emph{Ergodic Theory},
 Chelsea, New York, 1956.

 \bibitem{Ka} A.B.~Katok,
 \emph{Interval exchange transformations and some special flows are not	mixing},
 Israel\ J.\ Math.\ 35\ (1980), 301–310.

 \bibitem{Katok}  A.B.~Katok, A.N.~Zemljakov,
 \emph{Topological transitivity of billiards in polygons},
 Mat.\ Zametki 18 (1975),  291–300.

 \bibitem{King}J.\ King,
 \emph{Joining-rank and the structure of finite rank mixing transformations},
 J.\ Analyse Math.\ 51 (1988), 182-227.
	
 \bibitem{Khin}A.Y.\ Khinchin,
 \emph{Continued fractions},
 The University of Chicago Press, Chicago-London,	1964.
	
 \bibitem{KonZo}M.\ Kontsevich, A.\ Zorich,
 \emph{Connected components of the moduli spaces of Abelian differentials with prescribed singularities},
 Invent.\ Math.\ 153 (2003), 631–678.
	
 \bibitem{Maz}H.\ Masur,
 \emph{Interval exchange transformations and measured foliations},
 Ann.\ of Math.\ (2) 115 (1982), 169–200.
	
 \bibitem{Mos}J.\ Moser,
 \emph{On the volume elements on a manifold},
 Trans.\ Amer.\ Math.\ Soc.\ 120 (1965), 286–294.
	
 \bibitem{Ryz1} V.V.\ Ryzhikov,
 \emph{Partial multiple mixing on subsequences can distinguish between automorphisms $T$ and $T^{-1}$},
 Math.\ Notes 74 (2003), 841-847.
	
 \bibitem{Borel}S.M.\ Srivastava,
 \emph{A course on Borel sets},
 Graduate Texts in Mathematics, 180. Springer-Verlag, New York, 1998.
	
 \bibitem{Veech}W.A.\ Veech,
 \emph{Gauss measures for transformations on the space of interval exchange maps},
 Ann.\ of Math.\	115 (1982), 201–242.
	
 \bibitem{Yoccoz}J.C.\ Yoccoz,
 \emph{Interval exchange maps and translation surface},
 Homogeneous flows, moduli spaces and arithmetic, 1–69, Clay Math. Proc., 10, Amer. Math. Soc., Providence, RI, 2010.
	
\end{thebibliography}
\end{document}